\documentclass[a4paper,reqno]{amsart}

\usepackage[utf8]{inputenc}
\usepackage{amssymb}
\usepackage{enumitem}
\usepackage{mathrsfs}
\usepackage{mathtools}
\usepackage{amsmath}
\usepackage{color}
\usepackage[colorlinks=true]{hyperref}

\hypersetup{linkcolor=magenta,urlcolor=cyan,citecolor=cyan}

\setlist[enumerate]{label=\emph{(\roman*)}}

\usepackage{environ}

\newtheorem{theorem}{Theorem}[section]
\newtheorem{corollary}[theorem]{Corollary}
\newtheorem{lemma}[theorem]{Lemma}
\newtheorem{proposition}[theorem]{Proposition}

\theoremstyle{definition}

\newtheorem{remark}[theorem]{Remark}
\numberwithin{equation}{section}
\newcommand{\R}{\mathbb{R}}

\def \C {{\mathbb{C}}}
\def \d {{\rm{d}}}
\def \pt{\partial_{t}}
\def \bchar{{\boldsymbol{1}}}
\def \cint{{\mathcal{C}_{\rm{int}}}}
\def \cext{{\mathcal{C}_{\rm{ext}}}}
\begin{document}

\parindent=0pt

	\title[Global solution of 2D DKG systems]
	{Global behavior of small data solutions for the 2D Dirac--Klein-Gordon equations}
	
		\author[S.~Dong]{Shijie Dong}
	\address{Southern University of Science and Technology, SUSTech International Center for Mathematics, and Department of Mathematics, 518055 Shenzhen, P.R. China.}
	\email{dongsj@sustech.edu.cn, shijiedong1991@hotmail.com}

	\author[K.~Li]{Kuijie Li}
	\address{Nankai University, School of Mathematical Sciences and LPMC, Tianjin, 300071, P.R. China.}
	\email{kuijiel@nankai.edu.cn}
	
		\author[Y.~Ma]{Yue Ma}
	\address{Xi'an Jiaotong University, School of Mathematics and Statistics, 28 West Xianning Road, Xi'an Shaanxi 710049, P.R. China.}
	\email{yuemath@xjtu.edu.cn}

	\author[X.~Yuan]{Xu Yuan}
	\address{Department of Mathematics, The Chinese University of Hong Kong, Shatin, N.T., Hong Kong, P.R. China.}
	\email{xu.yuan@cuhk.edu.hk}
	\begin{abstract}
    In this paper, we are interested in the two-dimensional Dirac--Klein-Gordon system, which is a basic model in particle physics. We investigate the global behaviors of small data solutions to this system in the case of a massive scalar field and a massless Dirac field. More precisely, our main result is twofold: 1) we show sharp time decay for the pointwise estimates of the solutions which imply the asymptotic stability of this system;  2) we show the linear scattering result of this system which is a fundamental problem when it is viewed as dispersive equations. 
    Our result is valid for general small, high-regular initial data, in particular, there is no restriction on the support of the initial data.
	\end{abstract}
	\maketitle
\section{Introduction}

\subsection{Main result}

The Dirac--Klein-Gordon system is a basic model in particle physics, describing a scalar field and a Dirac field through Yukawa interactions. In this paper, we study the two-dimensional Dirac--Klein-Gordon system with a massive scalar field and a massless Dirac field.
The model of interest takes the following mathematical formulation
\begin{equation}\label{equ:DKG}
\left\{\begin{aligned}
-i\gamma^{\mu}\partial_{\mu}\psi&=v\psi,\quad \quad \ \ (t,x)\in [0,\infty)\times \R^{2},\\
-\Box v+v&=\psi^{*}\gamma^{0}\psi,\quad (t,x)\in [0,\infty)\times \R^{2}.
\end{aligned}\right.
\end{equation}
The initial data are prescribed on the slice $t=0$
\begin{equation}\label{equ:initial}
\left(\psi,v,\pt v\right)_{|t=0}=\left(\psi_{0},v_{0},v_{1}\right),
\end{equation}
in which $(\psi_{0},v_{0},v_{1}):\mathbb{R}^2 \to \mathbb{C}^2\times \R\times \R$.
The Dirac matrices~$\left\{\gamma^{0},\gamma^{1},\gamma^{2}\right\}$ satisfy the identities
\begin{equation}\label{equ:gamma}
\left\{\gamma^{\mu},\gamma^{\nu} \right\}:=\gamma^\mu \gamma^\nu + \gamma^\nu \gamma^\mu = -2\eta_{\mu\nu}I_{2},\quad 
(\gamma^{\mu})^{*}=-\eta_{\mu\nu}\gamma^{\nu},
\end{equation}
in which $A^{*}$ is the conjugate transpose of the matrix $A$, $I_{2}$ is the $2\times2$ identity matrix, and $\eta=\rm{diag}(-1,1,1)$ denotes the Minkowski metric in $\R^{1+2}$. More specifically, Dirac matrices take the explicit expressions
\begin{equation*}
\gamma^{0}=\begin{pmatrix}
1 & 0\\
0 & -1
\end{pmatrix}, \quad 
\gamma^{1}=\begin{pmatrix}
0 & 1\\
-1 & 0
\end{pmatrix},\quad 
\gamma^{2}=\begin{pmatrix}
0 & -i\\
-i & 0
\end{pmatrix}.
\end{equation*}

\smallskip
The local existence of the Cauchy problem \eqref{equ:DKG}--\eqref{equ:initial} can be found in Bournaveas \cite{Boura}, the global existence of the Cauchy problem \eqref{equ:DKG}--\eqref{equ:initial} was proved in Gr\"{u}nrock-Pecher \cite{GH10}, and in both works, the authors considered low regularity and large initial data.  Based on the local and global existence results, a fundamental problem is to understand the global behavior of the solution to the Dirac--Klein-Gordon system. This problem is challenging even in the 3D case where both the scalar and the Dirac fields enjoy faster pointwise decay. For a 3D Dirac--Klein-Gordon system with a massless scalar field and a massive Dirac field, the Cauchy problem was tackled in the pioneering work of Bachelot \cite{Bache}, where the author demonstrated the sharp pointwise decay of the solutions together with a scattering result.

\smallskip
 For the 2D Dirac--Klein-Gordon system, it is regarded to be more difficult as we lose $t^{-1/2}$ decay rate in the pointwise estimates of the solutions. The global behavior is well-understood in the case of a massive scalar field and a massive Dirac field, where the Dirac--Klein-Gordon system can be reduced to a system of coupled Klein-Gordon equations; see the works \cite{SimTaf} by Simon-Taflin and \cite{OTT} by Ozawa-Tsutaya-Tsutsumi.  However, the global behavior of the solutions to the 2D Dirac--Klein-Gordon system \eqref{equ:DKG} was unknown until the recent work \cite{DongWyatt} by Dong-Wyatt. In this work, the authors illustrated the sharp time decay for the solutions in the case of small, high-regular initial data with compact support.

\smallskip
In the present paper, we revisit the 2D Cauchy problem \eqref{equ:DKG}--\eqref{equ:initial}: we prove the global existence, the sharp time decay, and the scattering result of the solutions for all initial data which are small, high-regular, but not necessarily compactly supported. Our main result is formulated as follows.

\begin{theorem}\label{thm:main}
	Let $N\in \mathbb{N}$ with $N\ge 14$ and $\mathcal{X}_{N}=H^{N-1}\times H^{N}\times H^{N-1}$. There exists an $\varepsilon_0>0$ such that for all initial data $\left(\psi_{0},v_{0},v_{1}\right)$ 
	satisfying the smallness condition
	\begin{equation}\label{est:smallness}
	\begin{aligned}
	&\|\psi_{0}\|_{L^{1}}+\sum_{k\le N} \left(\left\|\langle |x|\rangle^{k+1}\nabla_{x}^{k}\psi_{0}\right\|_{L^{2}}
	+\left\|\langle |x|\rangle^{k+2} \log(2+|x|) \nabla_{x}^{k}v_{1}\right\|_{L^{2}}\right)\\
	&+\sum_{k\le N+1}\left\|\langle |x|\rangle^{k+1} \log(2+|x|) \nabla_{x}^{k}v_{0}\right\|_{L^{2}}\le \varepsilon<\varepsilon_0,
	\end{aligned}
	\end{equation}
	the Cauchy problem~\eqref{equ:DKG}-\eqref{equ:initial} admits a global-in-time solution $(\psi,v)$, which enjoys the following sharp time decay estimates
	\begin{equation}\label{est:thmpoint}
	\left|\psi(t,x)\right|\lesssim \varepsilon \langle t\rangle^{-\frac{1}{2}},\quad 
	\left|v(t,x)\right|\lesssim \varepsilon\langle t+|x|\rangle^{-1}.
	\end{equation}
	Moreover, the solution $(\psi,v)$ scatters to a free solution in $\mathcal{X}_{N}$ for $t\to \infty${\rm{:}} there exists $(\psi_{0\ell},v_{0\ell},v_{1\ell})\in \mathcal{X}_{N}$ such that 
	\begin{equation}\label{est:thmsca}
	\begin{aligned}
	\lim_{t\to \infty}\left\|(\psi,v,\pt v)-(\psi_{\ell},v_{\ell},\pt v_{\ell})\right\|_{\mathcal{X}_{N}}=0,
	\end{aligned}
	\end{equation}
	where $(\psi_{\ell},v_{\ell})$ is the solution to the 2D linear homogeneous Dirac--Klein-Gordon system with the initial data $(\psi_{0\ell},v_{0\ell},v_{1\ell})$.
\end{theorem}

\begin{remark}
In \cite{GH10}, the global existence to the 2D Cauchy problem \eqref{equ:DKG}--\eqref{equ:initial} was first obtained, in which the description of the asymptotic behavior was absent. Very recently in \cite{DongWyatt}, the asymptotic behavior was described for the solution with compactly supported initial data. To the best of our knowledge, our  current result appears to be the first asymptotic behavior result for non-compactly supported initial data.
\end{remark}

\begin{remark}
For the 2D Dirac-Klein-Gordon system \eqref{equ:DKG}, which can be regarded as dispersive equations, it is particularly interesting to investigate its scattering aspect. In view of \eqref{est:thmpoint}, neither of the pointwise decay rates of the solutions is integrable in time. Thus the linear scattering result \eqref{est:thmsca} is not a direct consequence of the pointwise decay. On the other hand, the decay result \eqref{est:thmpoint} is not a consequence of the scattering result either. Nevertheless, we manage to obtain the scattering result by carefully exploring different properties of the scalar field and the Dirac field (see more details in \S\ref{Se:Nonlin}-\ref{Se:Endthm}). To the best of our knowledge, this appears to be the first scattering result for the 2D Dirac--Klein-Gordon system \eqref{equ:DKG}.
\end{remark}

\begin{remark}
The proof for Theorem \ref{thm:main} should be applicable to the 3D Dirac--Klein-Gordon system in the case of a massive scalar field and a massless Dirac field, where both fields gain an extra $t^{-1/2}$ decay rate. This will lead us to the same (even a better) result of pointwise decay and linear scattering of the small data solutions.
\end{remark}

\begin{remark}
In the proof, one key ingredient is the observation that the nonlinear term in the equation of the Dirac field $\psi$ enjoys an extra $KG \times KG$ structure, which originally is of the form $KG \times Dirac$. This structure does not directly appear in the Dirac equation, but does appear when we do the energy estimate (see for instance Lemma \ref{le:EneryDirac}). This observation, together with Lemma \ref{lem:hidden} or \cite[Lemma 3.2]{DongWyatt}, helps us to improve the energy estimate for the massless Dirac field (see more details in Lemma~\ref{le:highest}).
\end{remark}

\subsection{References}
The Dirac--Klein-Gordon system has attracted people's attention for decades, and there exists a large literature on its mathematical study.
The questions of vital importance to the Cauchy problem of this system include but are not limited to: 1) global existence; 2) global behavior, such as pointwise decay and linear scattering, of the solutions when certain conditions are posed on the initial data.

\smallskip
For the 3D case, Chadam and Glassey \cite{ChGl74} first proved global existence for a Dirac--Klein-Gordon system with certain special initial data. Then Choquet-Bruhat \cite{Ch-Br81} showed global existence and scattering for a massless Dirac--Klein-Gordon system. Later on, Bachelot \cite{Bache} proved global existence and showed global behavior of a Dirac--Klein-Gordon system (in the case of a massless scalar field and a massive Dirac field) with small, high-regular, compactly supported initial data, and this work is closely related to our study. 
Recently, Wang \cite{WangX} and Bejenaru-Herr \cite{Bejenaru-Herr} proved global existence and scattering for a Dirac--Klein-Gordon system (with certain restrictions on the mass of the scalar field and the mass of the Dirac field) with low-regular initial data. There also exists the work \cite{DFS-07} as well as others in the 3D case for this system, but we are not going to be exhaustive here.

\smallskip
For the 2D case, Bournaveas \cite{Boura} proved local existence to the system with low-regular initial data motivated by an earlier work of Zheng \cite{ZhengY} on a modified Dirac--Klein-Gordon system. Later on, D’Ancona-Foschi-Selberg \cite{DFS-07-2D} improved this local existence result. After that, Gr\"{u}nrock-Pecher \cite{GH10} first 
showed global existence to the Dirac--Klein-Gordon system for low-regular (and also high-regular) initial data. 
Very recently, Dong-Wyatt \cite{DongWyatt} demonstrated the first asymptotic stability result for system \eqref{equ:DKG} with small, high-regular, and compactly supported initial data. In the proof of~\cite{DongWyatt}, the authors relied on the vector field method and the hyperboloidal foliation of the spacetime (i.e., $s=\sqrt{t^2-|x|^2}=constant$ slices), and in contrast, in the present paper we will use the flat foliation of the spacetime (i.e., $t=constant$ slices). For the 1D case, we would like to draw one's attention to the works \cite{Boura00, ChGl74}.

\smallskip
A model problem closely relevant to our study is the Klein-Gordon--Zakharov system, which can be regarded as semilinear coupled wave and Klein-Gordon equations. In the 3D case, global existence and global behavior were obtained in the pioneering work \cite{OTT95} for small perturbations around $0$; in a recent work \cite{Guo-N-W}, the authors proved linear scattering for solutions in the energy space with radial symmetry. In the 2D case, very recently the pointwise asymptotic behavior of the small smooth solutions to the system was illustrated; see for instance \cite{Dong2006, DM20, DongMa}. It is worth mentioning that, in \cite{DongMa}, the authors obtained a scattering result for the solutions to the Klein-Gordon part of the 2D Klein-Gordon--Zakharov system.

\smallskip
Finally, we want to lead one to some Dirac-related or Klein-Gordon-related works \cite{Bejenaru-Herr2D, Candy-Herr2, Chen-Dirac, DLW, Dong-Li21, DMZ22, Ionescu-P-WKG, Ionescu-P-EKG, Klainerman-WY, PLF-YM-arXiv1, ZangJDE, Stingo}, which are also relevant to our study.

\subsection{Major difficulties}
A direct problem we encounter when studying the Dirac--Klein-Gordon system \eqref{equ:DKG}, using Klainerman's vector field method, is that we cannot commute the system with the scaling vector field $S = t\partial_t + x^a \partial_a$, due to the presence of a massive Klein-Gordon field. This obstacle makes it already non-trivial to study the global existence and the global behavior of the solutions to this system in the 3D case. Moreover, in the 2D case, the following two extra difficulties arise.

\smallskip
First, the slow decay nature of both the 2D Dirac and 2D Klein-Gordon fields causes serious problems. Recall that in the 2D case, the free Dirac field's decay rate is $t^{-1/2}$ (there also exists some $\langle t-r\rangle$ decay rate which however is gone near $r=t$), and the free Klein-Gordon field's decay rate is $t^{-1}$, both of which are non-integrable quantities. Thus, this fact makes it highly non-trivial to show the boundedness of the $L^1_tL^2_x$ norm of the nonlinearities, which is required for closing the bootstrap, as they are just quadratic in \eqref{equ:DKG}.

\smallskip
Second, one problem of the insufficient decay of the Dirac field $\psi$ occurs when applying the ghost weight method to close the highest-order energy estimates for both fields, which is the most severe obstacle in the present paper. This problem does not exist in the case of compact initial data; see for instance \cite{DongWyatt}. Recall that it seems necessary to have,
\begin{equation}\label{equ:nece-decay}
|\psi(t, x)| \lesssim \langle t\rangle^{-1/2} \langle t-r\rangle^{-1/2 -\delta},\quad \mbox{for some}\ 0<\delta\ll 1,
\end{equation}
when applying the ghost weight method to close the highest-order energy estimates (according to the spacetime integral parts in \eqref{equ:Denergy} and \eqref{def:KGenergy}). However, even in \cite{DongWyatt} where the initial data are compactly supported, the authors only can show
\begin{equation*}
|\psi(t, x)| \lesssim \langle t\rangle^{-1/2} \langle t-r\rangle^{-1/2}.
\end{equation*}
which suffices to close the highest-order energy estimates by using the hyperboloidal foliation method.
Thus, intuitively one does not expect \eqref{equ:nece-decay}, and needs to find another way to go through the highest-order energy estimates.

\subsection{Key points of the proof}\label{subsec:key-idea}
We now explain how to solve the problems mentioned above. We apply the strategies used in \cite{Dong2005,DongMa,DongMaYuan} to circumvent problems caused by the lack of the scaling vector field. As for the issue of the slow time decay rates for 2D Dirac and Klein-Gordon fields, we rely mainly on two ingredients to conquer it: 1) we rely on various non-trivial nonlinear transformations to transform the original nonlinear terms to faster decaying terms (either null nonlinearities or cubic ones), which are essentially in the spirit of Shatah's normal form \cite{Shatah} (we also refer to \cite{DongWyatt, Tsutsumi}); 2) we rely on an observation that the Klein-Gordon component $v$ is able to transform the $\langle t-r\rangle$ decay of the Dirac component $\psi$ to the favorable $\langle t+r\rangle$ decay, and this is a vital ingredient used in \cite{DongMa} in the context of 2D Klein-Gordon--Zakharov system. We refer to \cite[Introduction]{DongMa} for a more detailed illustration. 

\smallskip
Next, we explain in detail how we resolve the highest-order energy estimates.
Heuristically, we consider the first order energy and chose $\psi\partial_{\alpha}v$ as the representative nonlinear term. Note that, the term $\psi\partial_{\alpha}v$ appears in the energy estimate of $\partial_{\alpha}\psi$, due to the Leibniz rule and $\left[-i\gamma^{\mu}\partial_{\mu},\partial_{\alpha}\right]=0$,
\begin{equation*}
-i\gamma^\mu \partial_\mu \partial_\alpha \psi 
=\partial_\alpha(v \,\psi)
=\psi\partial_{\alpha}v+v\partial_{\alpha}\psi.
\end{equation*}
We first assume 
\begin{equation}\label{equ:achievable}
|\psi(t, x)| \lesssim \langle t+r\rangle^{-1/2+\delta} \langle t-r\rangle^{-1},
\end{equation}
which is achievable (to be explained below). 
There is no special structure if we directly estimate the nonlinear term in \eqref{est:Energyphi1} of Lemma \ref{le:EneryDirac},
\begin{equation*}
\begin{aligned}
\int_0^t \| \psi\partial_{\alpha}v \| \d s
&\lesssim
\bigg( \int_0^t \left\| \frac{\partial_{\alpha}v}{\langle r-s\rangle^{3/5}} \right\|_{L_{x}^2}^2\d s \bigg)^{1/2} \left( \int_0^t \left\|\langle r-s\rangle^{3/5}\psi \right\|_{L_{x}^\infty}^2 \, \d s \right)^{1/2}
\\
&\lesssim
\langle t\rangle^\delta \bigg( \int_0^t \left\| \frac{\partial_\alpha v}{\langle r-s\rangle^{3/5}} \right\|_{L_{x}^2}^2\d s \bigg)^{1/2},
\end{aligned}
\end{equation*}
in which we used \eqref{equ:achievable} in the last step, and this does not allow one to close the bootstrap due to the growth $\langle t\rangle^\delta$. However, a special structure can be discovered, based on a result in Lemma \ref{lem:hidden} or \cite[Lemma 3.2]{DongWyatt}. More precisely, if we apply \eqref{est:Energyphi2} in Lemma \ref{le:EneryDirac}, we need to treat
\begin{equation*}
(\partial_\alpha v) \psi^* \gamma^0 \partial_\alpha\psi,
\end{equation*}
instead of the original $\psi\partial_\alpha v$. Thus, according to Lemma \ref{lem:hidden}, we need to bound
\begin{equation*}
\int_0^t \int_{\mathbb{R}^2} \left|(\partial_\alpha v) \psi^* \gamma^0 \partial_\alpha \psi \right| \d x \d s
\lesssim
\int_0^t \int_{\mathbb{R}^2} \left| \partial_\alpha v \right| \left|\psi \right| \left| [\partial_\alpha \psi]_- \right|\d x \d s
+ \mbox{similar terms},
\end{equation*}
with $[\cdot]_-$ defined in \eqref{def:phi-}. Using \eqref{equ:achievable} in the above, we have 
\begin{equation*}
\begin{aligned}
&\int_0^t \int_{\mathbb{R}^2} \left| \partial_\alpha v \right| \left|\psi \right| \left| [\partial_\alpha \psi]_- \right|\d x \d s
\\
&\lesssim
\int_0^t \left\| \frac{\partial_\alpha v}{\langle s-r\rangle^{3/5}} \right\|_{L_{x}^2}  \left\| \frac{[\partial_\alpha \psi]_-}{\langle s-r\rangle^{3/5}} \right\|_{L_{x}^2}\d s
\\
&\lesssim
\bigg(\int_0^t \left\| \frac{\partial_\alpha v}{\langle s-r\rangle^{3/5}} \right\|^2_{L_{x}^2}\d s\bigg)^{1/2}
\bigg(\int_0^t \left\| \frac{[\partial_\alpha \psi]_-}{\langle s-r\rangle^{3/5}} \right\|^2_{L_{x}^2} \, \d s \bigg)^{1/2},
\end{aligned}
\end{equation*}
in which we used the fact $|\langle s-r\rangle^{6/5} \psi| \lesssim 1$. This is not yet sufficient for closing the bootstrap, but the advantage is that there is no extra time growth in the right hand side. Some other technical tricks would help to close the bootstrap, but we do not exhaust those details here.
 
\smallskip
We end this part by briefly showing how to get \eqref{equ:achievable}. We introduce an auxiliary variable $\Psi$, which is the solution to some wave equation, and which is connected with $\psi$ via the relation $|\psi| \lesssim |\partial \Psi|$. This idea originally came from \cite[Lemma 3]{Boura00} when treating 1D Dirac--Klein-Gordon equations. For the variable $\Psi$, we show that 
\begin{equation}
 |L_{1} \Psi|+ |L_{2} \Psi|+ |S \Psi| \lesssim \langle t+r\rangle^{-1/2+\delta},
\end{equation}
and then we are led to \eqref{equ:achievable} by
$$
 |\psi| \lesssim |\partial \Psi|
 \lesssim \langle t-r\rangle^{-1} \left(  |L_1 \Psi|+|L_{2}\Psi| + |S \Psi| \right)
 \lesssim  \langle t+r\rangle^{-1/2+\delta} \langle t-r\rangle^{-1},
$$
in which $0<\delta \ll 1$, and $L_1$, $L_{2}$ and $S$ are standard vector fields.

\subsection{Outline of the paper}
The paper is organized as follows. First, in section~\ref{Se:Pre}, we state notation and all the technical tools involved in an energy approach to the small data solution problem of the system~\eqref{equ:DKG}: the estimates on null form and vector fields, global Sobolev inequality, pointwise and energy estimates of the Dirac spinor and wave-type field, and the hidden structures within the system~\eqref{equ:DKG}. Second, in subsection~\ref{Se:Boot}, we introduce the bootstrap setting of this paper. Then, in subsection~\ref{Se:Pointwise}-\ref{Se:Nonlin}, we derive the pointwise estimates for solution and the space-time norm estimates for nonlinear terms under the bootstrap assumption. Last, we finish the proof of Theorem~\ref{thm:main} in subsection~\ref{Se:End}-\ref{Se:Endthm}.

\section{Preliminaries}\label{Se:Pre}

\subsection{Notation and conventions}
Our problem is in (1+2) dimensional spacetime $\R^{1+2}$. We denote a spacetime point in $\R^{1+2}$ by 
$(t,x)=(x_{0},x_{1},x_{2})$, and its spatial radial by $r=\sqrt{x_{1}^{2}+x_{2}^{2}}$. Set $\omega_{a}=\frac{x_{a}}{r}$ for $a=1,2$ and $x=(x_{1},x_{2})\in \R^{2}$. Spacetime indices are represented by Greek letters $\left\{\alpha,\beta,\dots\right\}$ and spatial indices are denoted by Roman indices $\left\{a,b,\dots\right\}$. We use Japanese bracket to denote $\langle \rho\rangle=(1+\rho^{2})^{\frac{1}{2}}$ for $\rho\in \R$.
We write $A \lesssim B$ to indicate $A\leq C_0 B$ with $C_0$ a universal constant. Unless otherwise specified we will always adopt the Einstein summation convention for repeated lower and upper indices.

Following Klainerman~\cite{Klainerman86}, we introduce the following vector fields
\begin{enumerate}
	\item [(1)] Translations: $\partial_{\alpha}:=\partial_{x_{\alpha}}$, for $\alpha=0,1,2$.
	\item [(2)] Rotation: $\Omega:=x_{1}\partial_{2}-x_{2}\partial_{1}$.
	\item [(3)] Scaling vector field: $S:=t\pt +x^{a}\partial_{a}$.
	\item [(4)] Lorentz boosts: $L_{a}:=t\partial_{a}+x_{a}\partial_{t}$, for $a=1,2$.
\end{enumerate}
Following Bachelot~\cite{Bache}, we also introduce the modified rotation and Lorentz boots,
\begin{equation*}
\widehat{\Omega}=\Omega-\frac{1}{2}\gamma^{1}\gamma^{2}\quad \mbox{and}\quad 
\widehat{L}_{a}=L_{a}-\frac{1}{2}\gamma^{0}\gamma^{a},\ \mbox{for}\ a=1,2.
\end{equation*}
We define the two ordered sets of vector fields,
\begin{equation*}
\begin{aligned}
\left(\Gamma_{1},\Gamma_{2},\Gamma_{3},\Gamma_{4},\Gamma_{5},\Gamma_{6}\right)&=\left(\partial_{0},\partial_{1},\partial_{2},\Omega,L_{1},L_{2}\right),\\
\left(\widehat{\Gamma}_{1},\widehat{\Gamma}_{2},\widehat{\Gamma}_{3},\widehat{\Gamma}_{4},\widehat{\Gamma}_{5},\widehat{\Gamma}_{6}\right)&=\left(\partial_{0},\partial_{1},\partial_{2},\widehat{\Omega},\widehat{L}_{1},\widehat{L}_{2}\right).
\end{aligned}
\end{equation*}
Moreover, for all multi-index $I=(i_{1},i_{2},i_{3},i_{4},i_{5},i_{6})\in \mathbb{N}^{6}$, we denote
\begin{equation*}
\Gamma^{I}=\prod_{k=1}^{6}\Gamma_{k}^{i_{k}}\quad \mbox{and}\quad 
\widehat{\Gamma}^{I}=\prod_{k=1}^{6}\widehat{\Gamma}_{k}^{i_{k}}.
\end{equation*}
Additionally, we introduce the good derivatives $G_{a}=\partial_{a}+\omega_{a}\partial_{t}$ for $a=1,2$.

To shorten notation, we denote 
\begin{equation*}
\begin{aligned}
\Gamma=\left(\Gamma_{1},\Gamma_{2},\Gamma_{3},\Gamma_{4},\Gamma_{5},\Gamma_{6}\right)\quad \mbox{and}\quad \widehat{\Gamma}=\left(\widehat{\Gamma}_{1},\widehat{\Gamma}_{2},\widehat{\Gamma}_{3},\widehat{\Gamma}_{4},\widehat{\Gamma}_{5},\widehat{\Gamma}_{6}\right).
\end{aligned}
\end{equation*}
More specifically, for any $\C$-valued or $\C^{2}$-valued function $f$, we have 
\begin{equation*}
\begin{aligned}
\left|\Gamma f\right|=\left(\sum_{k=1}^{6}\left|\Gamma_{k}f\right|^{2}\right)^{\frac{1}{2}}\quad \mbox{and}\quad \big|\widehat{{\Gamma}}f\big|=\left(\sum_{k=1}^{6}\big|\widehat{\Gamma}_{k}f\big|^{2}\right)^{\frac{1}{2}}.
\end{aligned}
\end{equation*}

To state the global Sobolev inequality, we also define 
\begin{equation*}
\left(\Lambda_{1},\Lambda_{2}\right)=\left(\partial_{r},\Omega\right)\quad \mbox{and}\quad
\Lambda^{I}=\Lambda_{1}^{i_{1}}\Lambda_{2}^{i_{2}},\ \  \mbox{for}\ I=(i_{1},i_{2})\in \mathbb{N}^{2}.
\end{equation*}

Following Dong-Wyatt~\cite[Section 3]{DongWyatt}, we introduce the following notation in order to explain the hidden structure of the nonlinear term of~\eqref{equ:DKG}: for all $\phi=\phi(t,x):\R^{1+2}\to \C^{2}$, we set 
\begin{equation}\label{def:phi-}
\begin{aligned}
\left[\phi\right]_{+}(t,x)&:=\phi(t,x)+\omega_{a}\gamma^{0}\gamma^{a}\phi(t,x),\\
\left[\phi\right]_{-}(t,x)&:=\phi(t,x)-\omega_{a}\gamma^{0}\gamma^{a}\phi(t,x).
\end{aligned}
\end{equation}

We fix a suitable cut-off $C^{1}$ function $\chi:\R\to \R$ such that 
\begin{equation}\label{def:chi}
\chi'\ge 0\quad \mbox{and}\quad  \chi=\left\{\begin{aligned}
&0,\quad \mbox{for}\ x\le 1,\\
&1,\quad \mbox{for}\ x\ge 2.
\end{aligned}\right.
\end{equation}

Let $\left\{\chi_{j}\right\}_{j=0}^{\infty}$ be a Littlewood-Paley partition of unity, i.e.
\begin{equation*}
1=\sum_{j=0}^{\infty}\chi_{j}(\rho),\ \rho\ge 0,\ \chi_{j}\in C_{0}^{\infty}\left(\R^{2}\right),\quad \chi_{j}\ge 0\quad \mbox{for all}\ j\ge 0,
\end{equation*}
and
\begin{equation*}
\mbox{supp}\chi_{0}\cap [0,\infty)=[0,2],\quad 
\mbox{supp}\chi_{j}\subset \left[2^{j-1},2^{j+1}\right]\quad \mbox{for all}\ j\ge 1.
\end{equation*}

For simplicity of notation, we denote the initial data $(\psi_{0},v_{0},v_{1})$ by $\left({\psi}_{0},\vec{v}_{0}\right)$.

	Last, for further reference, we recall the following standard Gronwall's inequality. Let $T>0$ and $f$, $g$, $h\in C\left([0,T]:\R\right)$ be nonnegative functions with the property that
\begin{equation*}
f(t)\le g(t)+\int_{0}^{t}h(s)f(s)\d s,\quad \mbox{for all}\ t\in [0,T].
\end{equation*}
Then, we have 
\begin{equation}\label{est:Gron}
f(t)\le g(t)+\int_{0}^{t}g(s)h(s)\exp\left(\int_{s}^{t}h(\tau)\d \tau\right)\d s,\quad \mbox{for all}\  t\in [0,T].
\end{equation}

\subsection{Preliminary estimates}
In this subsection, we recall several preliminary estimates that are related to the vector fields and the null form. First, from the fact that the differences between (unmodified) vector fields $\Gamma$ and modified vector fields $\widehat{{\Gamma}}$ are constant matrices, we have 
\begin{equation}\label{est:hatGa}
\sum_{|I|\le M}\big|\widehat{{\Gamma}}^{I}f\big|\lesssim \sum_{|I|\le M}\big|{\Gamma}^{I}f\big|\lesssim \sum_{|I|\le M}\big|\widehat{{\Gamma}}^{I}f\big|,
\end{equation}
\begin{equation}\label{est:hatparGa}
\sum_{|I|\le M}\big|\partial\widehat{{\Gamma}}^{I}f\big|\lesssim \sum_{|I|\le M}\big|\partial{\Gamma}^{I}f\big|\lesssim \sum_{|I|\le M}\big|\partial\widehat{{\Gamma}}^{I}f\big|,
\end{equation}
for any smooth $\mathbb{R}$-valued or $\mathbb{C}^{2}$-valued function $f$ and $M\in \mathbb{N}^{+}$.

Second, we recall the following estimates related to the vector fields $\Gamma$.

\begin{lemma}[\cite{Sogge}]
	For any smooth function $f=f(t,x)$, we have 
	
	\begin{equation}\label{est:comm}
	\left|\left[\partial,\Gamma^{I}\right]f\right|+\left|\left[S,\Gamma^{I}\right]f\right|\lesssim \sum_{|J|< |I|}\left|\partial \Gamma^{J}f\right|,
	\end{equation}
	\begin{equation}\label{est:Gaparf}
	\langle t+r\rangle \sum_{a=1,2}|G_{a}f|+\langle t-r\rangle |\partial f|\lesssim \left(|Sf|+\left|\Gamma f\right|\right).
	\end{equation}
\end{lemma}

\begin{proof}
	The proof of these inequalities directly follows from some elementary
	computations (see for instance~\cite[Page 38 and Proposition 1.1]{Sogge}), and we omit it.
	\end{proof}

Next, we recall the following estimates related to the null form $Q_{0}$.
\begin{lemma}
	For any smooth $\mathbb{C}$-valued or $\mathbb{C}^{2}$-valued functions $f$ and $g$, the following estimates hold.
	\begin{enumerate}
		\item {\rm {Estimate of $Q_{0}(f,g)$}}. We have 
	\begin{equation}\label{est:Q0fg}
	\left|Q_{0}(f,g)\right|\lesssim \langle t+r\rangle^{-1}\left(\left|Sf\right|+\left|\Gamma f\right|\right)\left|\Gamma g\right|,
	\end{equation}

\item {\rm{Estimate of $Q_{0}(f,g)$ in the interior region.}} We have 
	\begin{equation}\label{est:Q0inside}
	\begin{aligned}
	\left|Q_{0}(f,g)\right|
	&\lesssim \frac{\langle t-r\rangle}{\langle t+r\rangle}|\partial f||\partial g|+\langle t\rangle^{-1}|\Gamma f||\Gamma g|,\quad \mbox{for}\ r\le 3t+3.
	\end{aligned}
	\end{equation}
	
\end{enumerate}
\end{lemma}

\begin{proof}
	Proof of (i). See~\cite[Lemma 3.3]{Sogge} for the complete proof.
	
	Proof of (ii). The proof is inspired by the previous works~\cite[Section 2]{Geor} and
	\cite[Section 4]{LeFlochMa} on estimating the null form.~Note that, the null form can be rewritten as 
	\begin{equation*}
	\begin{aligned}
	Q_{0}(f,g)&=\left(1-\frac{r^{2}}{t^{2}}\right)(\pt f)(\pt g)-\frac{1}{t^{2}}\sum_{a=1,2}\left(L_{a}f\right)\left(L_{a}g\right)\\
	&+\frac{x^{a}}{t^{2}}\left[(\pt f)(L_{a}g)+(L_{a}f)(\pt g)\right],
	\end{aligned}
	\end{equation*} 
	which implies~\eqref{est:Q0inside}.
	\end{proof}
Recall that, from~\cite[Page 58]{Sogge}, for any multi-index ${I}\in \mathbb{N}^{+}$, we also have 
\begin{equation}\label{est:GamQ0}
\Gamma^{I}Q_{0}(f,g)=\sum_{I_{1}+I_{2}=I}Q_{0}\left(\Gamma^{I_{1}}f,\Gamma^{I_{2}}g\right).
\end{equation}
Combining~\eqref{est:Q0fg} and~\eqref{est:GamQ0}, for any $I\in \mathbb{N}^{6}$, we have 
\begin{equation}\label{est:GaQ0fg}
\left|\langle t+r\rangle \Gamma^{I}Q_{0}(f,g)\right|\lesssim \sum_{|I_{1}|+|I_{2}|\le |I|}\left(\left|S\Gamma^{I_{1}}f\right|+\left|\Gamma \Gamma^{I_{1}}f|\right|\right)\left|\Gamma \Gamma^{I_{2}}g\right|.
\end{equation}
In contrast to the case of the Klein-Gordon operator, the Dirac operator commutes with the modified vector field $\widehat{{\Gamma}}$ instead of the unmodified one $\Gamma$. Hence, the following two Lemmas that are related to the unmodified and modified vector fields $(\Gamma,\widehat{{\Gamma}})$ are required, in order to estimate the Dirac-Dirac interaction term $\psi^{*}\gamma^{0}\psi$ appearing as sourcing for the Klein-Gordon equation.
\begin{lemma}
	For any multi-index $I\in \mathbb{N}^{6}$, smooth $\mathbb{R}$-valued function $f_{1}$ and $\mathbb{C}^{2}$-valued function $f_{2}$, we have 
		\begin{equation}\label{est:hatGfPhi}
	\big|\big[\widehat{{\Gamma}}^{I}(f_{1}f_{2})\big]_{-}\big|\lesssim \sum_{|I_{1}|+|I_{2}|\le |I|}\left|\Gamma^{I_{1}}f_{1}\right|\big|\big[\widehat{{\Gamma}}^{I_{2}}f_{2}\big]_{-}\big|.
	\end{equation}
\end{lemma}

\begin{proof}
 By an elementary computation, we have 
	\begin{equation*}
	\begin{aligned}
	\big[\widehat{\Omega}(f_{1}f_{2})\big]_{-}&=\left(\Omega f_{1}\right)\left[f_{2}\right]_{-}+f_{1}\big[\widehat{\Omega}f_{2}\big]_{-},\\
	\big[\widehat{L}_{1}(f_{1}f_{2})\big]_{-}&=\left(L_{1} f_{1}\right)\left[f_{2}\right]_{-}+f_{1}\big[\widehat{L}_{1}f_{2}\big]_{-},\\
	\big[\widehat{L}_{2}(f_{1}f_{2})\big]_{-}&=\left(L_{2} f_{1}\right)\left[f_{2}\right]_{-}+f_{1}\big[\widehat{L}_{2}f_{2}\big]_{-}.
	\end{aligned}
	\end{equation*}
	Combining the above identities with the Leibniz rule of $\partial$ and then using an induction argument, we obtain~\eqref{est:hatGfPhi}.
	\end{proof}

\begin{lemma}
	For any multi-index $I\in \mathbb{N}^{6}$, and any smooth $\mathbb{C}^{2}$-valued functions $\Phi_{1}$ and $\Phi_{2}$, we have 
	\begin{equation}\label{est:hatGG}
	\left|\Gamma^{I}\left(\Phi_{1}^{*}\gamma^{0}\Phi_{2}\right)\right|\lesssim \sum_{|I_{1}|+|I_{2}|\le |I|}\big|\big(\widehat{{\Gamma}}^{I_{1}}\Phi_{1}\big)^{*}\gamma^{0}\big(\widehat{{\Gamma}}^{I_{2}}\Phi_{2}\big)\big|.
	\end{equation}
\end{lemma}

\begin{proof}
	By~\eqref{equ:gamma} and an elementary computation, we have 
	\begin{equation*}
	\begin{aligned}
	\Omega\left(\Phi_{1}^{*}\right)=\big(\widehat{\Omega}\Phi_{1}\big)^{*}-\frac{1}{2}\Phi_{1}^{*}\gamma^{1}\gamma^{2}\quad \mbox{and}\quad L_{a}\left(\Phi_{1}^{*}\right)=\big(\widehat{L}_{a}\Phi_{1}\big)^{*}+\frac{1}{2}\Phi_{1}^{*}\gamma^{0}\gamma^{a},
	\end{aligned}
	\end{equation*}
	and so 
	\begin{equation*}
	\begin{aligned}
	\Omega\left(\Phi_{1}^{*}\gamma^{0}\Phi_{2}\right)
	&=\big(\widehat{\Omega}\Phi_{1}\big)^{*}\gamma^{0}\Phi_{2}+\Phi_{1}^{*}\gamma^{0}\big({\widehat{\Omega}}\Phi_{2}\big)-\frac{1}{2}\Phi_{1}^{*}\left(\gamma^{1}\gamma^{2}\gamma^{0}-\gamma^{0}\gamma^{1}\gamma^{2}\right)\Phi_{2}\\
	&=\big(\widehat{\Omega}\Phi_{1}\big)^{*}\gamma^{0}\Phi_{2}+\Phi_{1}^{*}\gamma^{0}\big({\widehat{\Omega}}\Phi_{2}\big),
	\end{aligned}
	\end{equation*}
	\begin{equation*}
	\begin{aligned}
	L_{a}\left(\Phi_{1}^{*}\gamma^{0}\Phi_{2}\right)
	&=\big(\widehat{L}_{a}\Phi_{1}\big)^{*}\gamma^{0}\Phi_{2}+\Phi_{1}^{*}\gamma^{0}\big({\widehat{L}_{a}}\Phi_{2}\big)+\frac{1}{2}\Phi_{1}^{*}\left(\gamma^{0}\gamma^{a}\gamma^{0}+\gamma^{0}\gamma^{0}\gamma^{a}\right)\Phi_{2}\\
	&=\big(\widehat{L}_{a}\Phi_{1}\big)^{*}\gamma^{0}\Phi_{2}+\Phi_{1}^{*}\gamma^{0}\big({\widehat{L}_{a}}\Phi_{2}\big).
	\end{aligned}
	\end{equation*}
	Combining the above identities with the Leibniz rule of $\partial$ and then using an induction argument, we obtain~\eqref{est:hatGG}.
	\end{proof}

Last, for future reference, we recall the following two global Sobolev inequalities.
\begin{lemma}\label{le:GlobalSobolev}
For all sufficiently regular function $f=f(t,x)$, we have 
\begin{align}
|f(t,x)|&\lesssim \langle t+r\rangle^{-\frac{1}{2}}\sum_{|I|\le 3}\left\|\Gamma^{I}f(t,x)\right\|_{L_{x}^{2}},\label{est:GloSob}\\
|f(t,x)|&\lesssim \langle r\rangle^{-\frac{1}{2}}\sum_{\substack{|I|\le 2\\(i_{1},i_{2})\ne (2,0)}}\|\Lambda^{I}f(t,x)\|_{L_{x}^{2}}\label{est:GloSobr}.
\end{align}
\end{lemma}
\begin{proof}
	The inequality~\eqref{est:GloSob} is due to~Georgiev~\cite{Geor}. We refer to~\cite[Lemma 2.4]{Geor} for a complete proof and also refer to~\cite[Lemma 2.5 and Remark 2.6]{DongMaYuan} for an explanation.
	The inequality~\eqref{est:GloSobr} is a consequence of the standard Sobolev inequality on circle $S^{1}$. We refer to~\cite[Proposition 1]{Klainerman86} for a complete proof.
\end{proof}

\subsection{Estimates for the Dirac spinor}\label{Se:Dirac}
In this subsection, we present the energy estimate and extra pointwise decay of gradient for the Dirac spinor. We start with the ghost weight energy (introduced by Alinhac in \cite{Alinhac01b}) for all $\phi=\phi(t,x):\R^{1+2}\to \C^{2}$, denoting
\begin{equation}\label{equ:Denergy}
\mathcal{E}^{D}(t,\phi)=\int_{\R^{2}}|\phi(t,x)|^{2}\d x+\int_{0}^{t}\int_{\R^{2}}\frac{\left|[\phi]_{-}(s,x)\right|^{2}}{\langle r-s \rangle^{\frac{6}{5}}}\d x\d s.
\end{equation}

\begin{lemma}\label{le:EneryDirac}
	Suppose $\phi=\phi(t,x)$ is the solution to the Cauchy problem
	\begin{equation}\label{equ:Dirac}
	-i\gamma^{\mu}\partial_{\mu}\phi=G\quad \mbox{with}\quad \phi (0,x)=\phi_{0}(x),
	\end{equation}
	where $G=G(t,x):\R^{1+2}\to \C^{2}$ is a sufficiently nice function, 
	then the following estimates are true.
	\begin{enumerate}
		\item {\rm{Ghost weight estimate.}} It holds
	\begin{equation}\label{est:Energyphi1}
	\begin{aligned}
	\mathcal{E}^{D}(t,\phi)^{\frac{1}{2}}&\lesssim \left(\int_{\R^{2}}|\phi_{0}(x)|^{2}\d x\right)^{\frac{1}{2}}+\int_{0}^{t}\|G(s,x)\|_{L_{x}^{2}} \d s, 
	\end{aligned}
	\end{equation}
    \begin{equation}\label{est:Energyphi2}
	\begin{aligned}
	\mathcal{E}^{D}(t,\phi)&\lesssim \int_{\R^{2}}|\phi_{0}(x)|^{2}\d x+\int_{0}^{t}\int_{\R^{2}}\left|\phi^{*}(s,x)\gamma^{0}G(s,x)\right|\d x \d s.
	\end{aligned}
	\end{equation}
	
	\item {\rm {Modified ghost weight estimate.}} For $0<\delta_{1}\ll 1$, we have 
	\begin{equation}\label{est:ModiGhostphi}
	\begin{aligned}
	&\int_{0}^{t}\langle s\rangle^{-\delta_{1}}\int_{\R^{2}}\frac{\left|[\phi]_{-}(s,x)\right|^{2}}{\langle r-s \rangle^{\frac{6}{5}}}\d x\d s\\
	&\lesssim \int_{0}^{t}|\phi_{0}(x)|^{2}\d x+\int_{0}^{t}\langle s\rangle^{-\delta_{1}}\int_{\R^{2}}\left|\phi^{*}(s,x)\gamma^{0}G(s,x)\right|\d x \d s.
	\end{aligned}
	\end{equation}
	
	\item {\rm{Energy estimate on an exterior region.}} It holds
	\begin{equation}\label{est:cutenergyphi}
	\begin{aligned}
	&\left\|\langle r-t\rangle\chi(r-2t)\phi(t,x)\right\|_{L_{x}^{2}}\\
	&\lesssim \left\|\langle r\rangle\phi_{0}(x)\right\|_{L_{x}^{2}}+\int_{0}^{t}\left\|\langle r-s\rangle\chi(r-2s)G(s,x)\right\|_{L_{x}^{2}}\d s.
	\end{aligned}
	\end{equation}
\end{enumerate}
\end{lemma}

\begin{remark}
	The advantage of \eqref{est:Energyphi2} over \eqref{est:Energyphi1} is that a structure can be discovered to benefit us with the aid of Lemma \ref{lem:hidden}. This was explained with an example in Section \ref{subsec:key-idea}. Roughly speaking, the estimate \eqref{est:cutenergyphi} provides us with an extra $\langle r-t\rangle^{-1}$ decay, which automatically turns to the strong $\langle r+t\rangle^{-1}$ decay in the exterior region $\{ r \geq 2t + 3 \}$.
\end{remark}

\begin{proof}[Proof of Lemma~\ref{le:EneryDirac}]
	Proof of (i). For $(t,x)\in [0,\infty)\times \R^{2}$, we set 
	\begin{equation*}
	p(t,x)=\int_{-\infty}^{r-t}\langle \tau\rangle^{-\frac{6}{5}}\d \tau,\quad \mbox{where}\ r=|x|.
	\end{equation*}
	Multiplying $i e^p  \phi^* \gamma^0$ to both sides of the equation~\eqref{equ:Dirac}, we have
	\begin{equation} \label{eq:ghostdirac1}
	e^p \phi^* \partial_t \phi + e^p \phi^*\gamma^0\gamma^a \partial_a \phi = i e^p \phi^*\gamma^0G.
	\end{equation}
	Taking the complex conjugate of \eqref{eq:ghostdirac1} and using~\eqref{equ:gamma}, one can obtain 
	\begin{equation} \label{eq:ghostdirac2}
	e^p (\partial_t \phi)^* \phi + e^p (\partial_a \phi)^*\gamma^0\gamma^a \phi = -i e^p G^*\gamma^0 \phi.
	\end{equation}
	Combining \eqref{eq:ghostdirac1} and \eqref{eq:ghostdirac2}, we see that
	\begin{align*}
	\partial_t\left(e^p|\phi|^2\right) + \partial_a(e^p\phi^*\gamma^0\gamma^a \phi) +  \frac{e^p}{\langle r-t\rangle^{\frac{6}{5}}} (|\phi|^{2} - \omega_a \phi^*\gamma^0\gamma^a\phi) =-2e^p {\rm{Im}}(\phi^*\gamma^0G).
	\end{align*}
	On the other hand, by~\eqref{equ:gamma} and the definition of $[\phi]_{-}$ in~\eqref{def:phi-}, we find
	\begin{align*}
	|[\phi]_{-}|^2 & = (\phi-\omega_a \gamma^0\gamma^a \phi)^*(\phi-\omega_b\gamma^0\gamma^b \phi) \\
	&=\left(\phi^{*}-\omega_{a}\phi^{*}\gamma^{0}\gamma^{a}\right)(\phi-\omega_b\gamma^0\gamma^b \phi)\\
	& =  2|\phi|^{2} - 2\omega_a \phi^* \gamma^0\gamma^a \phi - \sum_{0<a<b} \omega_a\omega_b \phi^*(\gamma^a\gamma^b + \gamma^b\gamma^a) \phi.
	\end{align*}
	Note that, the last term in the above identity vanishes, which implies that
	\begin{equation}\label{equ:phi-}
	|[\phi]_{-}|^2=2(|\phi|^{2} - \omega_a \phi^* \gamma^0\gamma^a \phi).
	\end{equation}
	Combining the above identities, we obtain
	\begin{align*} 
	\partial_t(e^p|\phi|^2) + \partial_a(e^p\phi^*\gamma^0\gamma^a \phi) +  \frac{e^p}{2\langle r-t\rangle^{\frac{6}{5}}} \left|[\phi]_{-}\right|^2 =-2e^p {\rm{Im}}(\phi^*\gamma^0G).
	\end{align*}
	Integrating the above identity over the domain $[0,t]\times \R^{2}$ and then using the fact that $e^{p}\sim 1$, we obtain~\eqref{est:Energyphi1}--\eqref{est:Energyphi2}.
	
	\smallskip
	Proof of (ii). Here, we apply the multiplier $i \langle t\rangle^{-\delta_{1}} e^p  \phi^* \gamma^0$ instead of $i e^p  \phi^* \gamma^0$. Using a similar argument as in the proof of Lemma~\ref{le:EneryDirac} (i), we have
	\begin{align*} 
	&\partial_t(\langle t\rangle^{-\delta_{1}}e^p|\phi|^2) + \partial_a(\langle t\rangle^{-\delta_{1}}e^p\phi^*\gamma^0\gamma^a \phi) +  \frac{\langle t\rangle^{-\delta_{1}}e^p}{2\langle r-t\rangle^{\frac{6}{5}}} \left|[\phi]_{-}\right|^2 \\
	&+\delta_{1}t\langle t\rangle^{-(\delta_{1}+2)}e^{p}|\phi|^{2}=-2\langle t\rangle^{-\delta_{1}}e^p {\rm{Im}}(\phi^*\gamma^0G).
	\end{align*}
	Integrating the above identity over the domain $[0,t]\times \R^{2}$ and then using again the fact that $e^{p}\sim 1$, we obtain~\eqref{est:ModiGhostphi}.
	
	\smallskip
	Proof of (iii).	Multiplying $i\langle r-t\rangle^{2}\chi^{2}(r-2t)\phi^{*}\gamma^{0}$ to both sides of the equation~\eqref{equ:Dirac}, 
	\begin{equation} \label{equ:chidirac}
 \langle r-t\rangle^{2}\chi^{2}(r-2t)\left(\phi^* \partial_t \phi + \phi^*\gamma^0\gamma^a \partial_a \phi\right) = i \langle r-t\rangle^{2}\chi^{2}(r-2t)\phi^*\gamma^0G.
	\end{equation}
	Taking the complex conjugate of \eqref{equ:chidirac} and then using~\eqref{equ:gamma}, we deduce that
	\begin{equation} \label{equ:chidirac2}
	\langle r-t\rangle^{2}\chi^{2}(r-2t)\left((\pt\phi)^* \phi + (\partial_{a}\phi)^*\gamma^0\gamma^a \phi\right) = -i \langle r-t\rangle^{2}\chi^{2}(r-2t)G^*\gamma^0\phi.
	\end{equation}
	Combining~\eqref{equ:chidirac} and~\eqref{equ:chidirac2}, we get 
	\begin{equation*}
	\langle r-t\rangle^{2}\chi^{2}(r-2t)\left(\pt|\phi|^{2} + \partial_{a}(\phi^*\gamma^0\gamma^a \phi)\right) =-2 \langle r-t\rangle^{2}\chi^{2}(r-2t){\rm{Im}}(\phi^*\gamma^0G).
	\end{equation*}
	Then, by an elementary computation, we have 
	\begin{equation*}
	\begin{aligned}
	&\langle r-t\rangle^{2}\chi^{2}(r-2t)\pt|\phi|^{2}\\
	&=\pt(\langle r-t\rangle^{2}\chi^{2}(r-2t)|\phi|^{2})\\
	&+2\chi(r-2t)\left[(r-t)\chi(r-2t)+2\langle r-t\rangle^{2}\chi'(r-2t)\right]|\phi|^{2},\quad \quad 
	\end{aligned}
	\end{equation*}
	and 
	\begin{equation*}
	\begin{aligned}
	&\langle r-t\rangle^{2}\chi^{2}(r-2t)\partial_{a}\left(\phi^*\gamma^0\gamma^a \phi\right)\\
	&=\partial_{a}(\langle r-t\rangle^{2}\chi^{2}(r-2t)\phi^{*}\gamma^{0}\gamma^{a}\phi)\\
	&-2\chi(r-2t)\left[(r-t)\chi(r-2t)+\langle r-t\rangle^{2}\chi'(r-2t)\right]\omega_{a}\phi^*\gamma^0\gamma^a \phi.
	\end{aligned}
		\end{equation*}
	Combining the above identities with~\eqref{equ:phi-}, we find
	\begin{equation*}
	\begin{aligned}
	&\pt \left(\langle r-t\rangle^{2}\chi^{2}(r-2t)|\phi|^{2}\right)+\partial_{a}(\langle r-t\rangle^{2}\chi^{2}(r-2t)\phi^{*}\gamma^{0}\gamma^{a}\phi)\\
	&+2\chi(r-2t)\left[(r-t)\chi(r-2t)+\langle r-t\rangle^{2}\chi'(r-2t)\right]\left|\left[\phi\right]_{-}\right|^{2}\\
	&+2\langle r-t\rangle^{2}\chi(r-2t)\chi'(r-2t)|\phi|^{2}=-2 \langle r-t\rangle^{2}\chi^{2}(r-2t){\rm{Im}}(\phi^*\gamma^0G).
	\end{aligned}
	\end{equation*}
	Hence, using the fact that $\chi'$, $\chi\ge 0$, we have 
	\begin{equation*}
	\begin{aligned}
	&\pt \left(\langle r-t\rangle^{2}\chi^{2}(r-2t)|\phi|^{2}\right)+\partial_{a}(\langle r-t\rangle^{2}\chi^{2}(r-2t)\phi^{*}\gamma^{0}\gamma^{a}\phi)\\
	&\le 2 \langle r-t\rangle^{2}\chi^{2}(r-2t)\left|{\rm{Im}}(\phi^*\gamma^0G)\right|.
	\end{aligned}
	\end{equation*}
	Integrating the above inequality over the domain $[0,t]\in \R^{2}$ and then using the Cauchy-Schwarz inequality, we obtain~\eqref{est:cutenergyphi}.
	\end{proof}

Second, we explore the extra decay inside of a cone for the gradient of a Dirac spinor; see also \cite[Lemma 4.3]{DongWyatt}.
\begin{lemma}\label{le:Diracdecay}
	Suppose $\phi=\phi(t,x)$ is the solution to the equation
	\begin{equation}\label{equ:dirac}
	-i\gamma^{\mu}\partial_{\mu}\phi(t,x)=G(t,x),
	\end{equation}
	where $G=G(t,x):\R^{1+2}\to \C^{2}$ is a sufficiently nice function, 
	then we have 
	\begin{equation}\label{est:decaypphi}
	\left|\partial \phi\right|\lesssim\frac{1}{\langle t-r\rangle}\big(\big|\widehat{\Gamma}\phi\big|+\left|\phi\right|\big)+\frac{t}{\langle t-r\rangle}|G|,\quad \quad \mbox{for}\ r\le 3t+3.
	\end{equation}
	\end{lemma}

\begin{remark}
	This result means that a partial derivative $\partial$ on the Dirac field $\phi$ gives an extra $\langle t-r\rangle^{-1}$ decay, provided that the source term $G(t, x)$ has sufficiently fast decay which is exactly satisfied in our case. Moreover, a null condition or a Klein-Gordon field will then turn this $\langle t-r\rangle^{-1}$ decay into a favorable $\langle t+r\rangle^{-1}$ decay.
\end{remark}

\begin{proof}[Proof of Lemma~\ref{le:Diracdecay}] Case I. Let $t\in [0,1]$. It is easy to check that 
	\begin{equation}\label{est:pphi1}
	\left|\partial \phi\right|\lesssim \big|\widehat{{\Gamma}} \phi\big|\lesssim\frac{1}{\langle t-r\rangle}\big|\widehat{\Gamma}\phi\big|+\frac{t}{\langle t-r\rangle}|G|,\quad \quad \mbox{for}\ r\le 3t+3.
	\end{equation}
	
	Case II. Let $t\in (1,\infty)$. Using $\partial_a = t^{-1}L_{a} - (x_a/t) \partial_t$, we can rewrite~\eqref{equ:dirac} as 
	\begin{equation*}
	i\left(\gamma^{0}-\frac{x_{a}}{t}\gamma^{a} \right)\pt\phi=-i\gamma^{a}\frac{L_{a}}{t}\phi-G.
	\end{equation*}
	Multiplying $-i\left(\gamma^{0}-\frac{x_{a}}{t}\gamma^{a} \right)$ to the above identity and then using~\eqref{equ:gamma}, we find
	\begin{equation*}
	\frac{(t-r)(t+r)}{t^{2}}\pt \phi=-\left(\gamma^{0}-\frac{x_{a}}{t}\gamma^{a} \right)\gamma^{b}\frac{L_{b}}{t}\phi+i\left(\gamma^{0}-\frac{x_{a}}{t}\gamma^{a} \right)G,
	\end{equation*}
	which implies
	\begin{equation}\label{est:pphi2}
	|\pt \phi|\lesssim \frac{1}{\langle t-r\rangle}\big|\Gamma\phi\big|+\frac{t}{\langle t-r\rangle}|G|,\quad \quad \mbox{for}\ r\le 3t+3.
	\end{equation}
	Hence, using again $\partial_a = t^{-1}L_{a} - (x_a/t) \partial_t$, we have
	\begin{equation}\label{est:pphi3}
	\left|\partial_{a} \phi\right|\lesssim \frac{1}{t}|L_{a}\phi|+|\pt \phi|\lesssim \frac{1}{\langle t-r\rangle}\big|\Gamma\phi\big|+\frac{t}{\langle t-r\rangle}|G|,\quad \quad \mbox{for}\ r\le 3t+3.
	\end{equation}
	We see that~\eqref{est:decaypphi} follows from~\eqref{est:hatGa},~\eqref{est:pphi1},~\eqref{est:pphi2} and~\eqref{est:pphi3}.
	\end{proof}
\subsection{Estimates for the wave-type fields}\label{Se:Wave}
In this subsection, we recall several different types of energy estimates, the $L_{x}^{2}$ norm estimate and the pointwise decay for the 2D wave or Klein-Gordon equations, respectively. First, we introduce the following standard energy $\mathcal{E}_{1}$ and Alinhac's ghost weight energy $\mathcal{G}_{1}$,
\begin{equation}\label{def:KGenergy}
\begin{aligned}
\mathcal{E}_{1}(t,u)&=\int_{\R^{2}}\left((\pt u)^{2}+(\partial_{1} u)^{2}+(\partial_{2} u)^{2}+u^{2}\right)(t,x)\d x,\\
\mathcal{G}_{1}(t,u)&=\mathcal{E}_{1}(t,u)+\sum_{a=1,2}\int_{0}^{t}\int_{\R^{2}}\frac{|G_{a}u|^{2}}{\langle r-s\rangle^{\frac{6}{5}}}\d x \d s+\int_{0}^{t}\int_{\R^{2}}\frac{u^{2}}{\langle r-s\rangle^{\frac{6}{5}}}\d x \d s.
\end{aligned}
\end{equation}
We recall the following different types of energy estimates for the 2D Klein-Gordon equation. The proofs are similar to~\cite[Theorem 6.4]{Alin} and~\cite[Proposition 3.5]{DongMa}, but they are provided below for the sake of completeness and the better readability.
\begin{lemma}[\cite{Alinhac01b,Alin,DongMa}]\label{le:energywave}
	Suppose $u=u(t,x)$ is the solution to the Cauchy problem
	\begin{align}\label{equ:KGenergy}
	-\Box u+u=F\quad \mbox{with}\quad (u,\pt u)_{|t=0}=(u_{0},u_{1}),
	\end{align}
		where $F=F(t,x):\R^{1+2}\to \R$ is a sufficiently nice function, 
	then the following estimates are true.
	\begin{enumerate}
	\item {\rm{Ghost weight energy estimate.}} It holds
	\begin{equation}\label{est:EnergyKG}
	\mathcal{G}_{1}(t,u)^{\frac{1}{2}}\lesssim \mathcal{E}_{1}(0,u)^{\frac{1}{2}}+\int_{0}^{t}\|F(s,x)\|_{L_{x}^{2}}\d s.
	\end{equation} 
	
	\item {\rm{Modified ghost weight estimate.}} For $0<\delta_{1}\ll 1$, we have 
	\begin{equation}\label{est:ModighostKG}
	\begin{aligned}
	&\sum_{a=1,2}\int_{0}^{t}\langle s\rangle^{-\delta_{1}}\int_{\R^{2}}\left(\frac{u^{2}}{\langle r-s\rangle^{\frac{6}{5}}}+\frac{|G_{a}u|^{2}}{\langle r-s\rangle^{\frac{6}{5}}}\right)\d x \d s\\
	&\lesssim \mathcal{E}_{1}(0,u)+\int_{0}^{t}\langle s\rangle^{-\delta_{1}}\|F(s,x)\|_{L_{x}^{2}}\|\pt u(s,x)\|_{L_{x}^{2}} \d s.
	\end{aligned}
	\end{equation}
	
	\item {\rm{Energy estimate on an exterior region.}} It holds
	\begin{equation}\label{est:cutenergyv}
	\begin{aligned}
	&\left\|\langle r-t\rangle\chi(r-2t)u(t,x)\right\|_{L_{x}^{2}}+\left\|\langle r-t\rangle\chi(r-2t)\partial u(t,x)\right\|_{L_{x}^{2}}\\
	&\lesssim \left\|\langle r\rangle u_{0}(x)\right\|_{H_{x}^{1}}+ \| \langle r\rangle u_{1}(x) \|_{L_{x}^{2}} + \int_{0}^{t}\left\|\langle r-s\rangle\chi(r-2s)F(s,x)\right\|_{L_{x}^{2}}\d s.
	\end{aligned}
	\end{equation}
\end{enumerate}
\end{lemma}

\begin{proof}
	Proof of (i). For $(t,x)\in [0,\infty)\times \R^{2}$, we set 
	\begin{equation*}
	p(t,x)=\int_{-\infty}^{r-t}\langle \tau \rangle^{-\frac{6}{5}}\d \tau,\quad \mbox{where}\ r=|x|.
	\end{equation*}
	Multiplying $e^{p}\partial_{t}u$ to both sides of the equation~\eqref{equ:KGenergy}, we have 
	\begin{equation*}
	\begin{aligned}
	e^{p}(-\Box u+u)\pt u
	&=\frac{1}{2}\pt \left(e^{p}\left(|\partial u|^{2}+u^{2}\right)\right)-\partial^{a}\left(e^{p}\pt u\partial_{a}u\right)\\
	&+\frac{1}{2}\frac{e^{p}}{\langle t-r\rangle^{\frac{6}{5}}}\left((G_{1}u)^{2}+(G_{2}u)^{2}+u^{2}\right)=e^{p}F\pt u.
	\end{aligned}
	\end{equation*}
	Integrating the above identity over the domain $[0,t]\times \R^{2}$ and then using the fact that $e^{p}\sim 1$, we obtain~\eqref{est:EnergyKG}.
	
	\smallskip
	Proof of (ii). Here, we apply the multiplier $\langle t\rangle^{-\delta_{1}}e^{p}\pt u$ instead of $e^{p}\partial_{t}u$. By an elementary computation, we have 
	\begin{equation*}
	\begin{aligned}
		&\langle t\rangle^{-\delta_{1}}e^{p}(-\Box u+u)\pt u\\
		&=\frac{1}{2}\pt \left[\langle t\rangle^{-\delta_{1}}e^{p}\left(|\partial u|^{2}+u^{2}\right)\right]+\frac{\delta_{1}}{2}t\langle t\rangle^{-(\delta_{1}+2)}e^{p}\left(|\partial u|^{2}+u^{2}\right)\\
		&+\frac{1}{2}\frac{\langle t\rangle^{-\delta_{1}}e^{p}}{\langle t-r\rangle^{\frac{6}{5}}}\left((G_{1}u)^{2}+(G_{2}u)^{2}+u^{2}\right)-\partial^{a}\left(\langle t\rangle^{-\delta_{1}}e^{p}\pt u\partial_{a}u\right)=\langle t\rangle^{-\delta_{1}}e^{p}F\pt u.
		\end{aligned}
	\end{equation*}
	Integrating the above identity over the domain $[0,t]\times \R^{2}$ and then using again the fact that $e^{p}\sim 1$, we obtain~\eqref{est:ModighostKG}.
	
	\smallskip
	Proof of (iii). Multiplying $\langle r-t\rangle^{2}\chi^{2}(r-2t)\pt u$ to both sides of the equation~\eqref{equ:KGenergy},
	\begin{equation*}
	\begin{aligned}
	&\langle r-t\rangle^{2}\chi^{2}(r-2t)(-\Box u+u)\pt u\\
	&=\frac{1}{2}\pt \left[\langle r-t\rangle^{2}\chi^{2}(r-2t)(|\partial u|^{2}+u^{2})\right]
	-\partial^{a}\left[\langle r-t\rangle^{2}\chi^{2}(r-2t)(\partial_{a}u)(\pt u)\right]\\
	&+\left[(r-t)\chi^{2}(r-2t)+\langle r-t\rangle^{2}\chi'(r-2t)\chi(r-2t)\right]\left((G_{1}u)^{2}+(G_{2}u)^{2}+u^{2}\right)\\
	&+\langle r-t\rangle^{2}\chi'(r-2t)\chi(r-2t)\left(|\partial u|^{2}+u^{2}\right)=\langle r-t\rangle^{2}\chi^{2}(r-2t)F\pt u.
	\end{aligned}
	\end{equation*}
	Hence, using the fact that $\chi'$, $\chi\ge 0$, we have
	\begin{equation*}
	\begin{aligned}
	&\pt \left[\langle r-t\rangle^{2}\chi^{2}(r-2t)(|\partial u|^{2}+u^{2})\right]\\
	&\le 2\langle r-t\rangle^{2}\chi^{2}(r-2t)|F\pt u|+2\partial^{a}\left[\langle r-t\rangle^{2}\chi^{2}(r-2t)(\partial_{a}u)(\pt u)\right].
	\end{aligned}
	\end{equation*}
	Integrating the above inequality over the domain $[0,t]\times \R^{2}$ and then using the Cauchy-Schwarz inequality, we obtain~\eqref{est:cutenergyv}.
	\end{proof}

Then, we recall the pointwise decay for the Klein-Gordon component from~\cite{Geor}.

\begin{theorem}[\cite{Geor}]\label{thm:KGdecay}
	Suppose $u=u(t,x)$ is the solution to the Cauchy problem
	\begin{equation*}
	\begin{aligned}
	-\Box u+u=F\quad \mbox{with}\quad (u,\pt u)_{|t=0}=(u_{0},u_{1}),
	\end{aligned}
	\end{equation*}
	where $F=F(t,x):\R^{1+2}\to \R$ is a sufficiently nice function, then we have 
	\begin{equation}\label{est:KGLini}
	\begin{aligned}
	\langle t+r\rangle|u(t,x)|
	&\lesssim  \sum_{j=0}^{\infty}\sum_{|I|\le 5}\left\|\langle |x|\rangle \chi_{j}(|x|)\Gamma^{I}u(0,x)\right\|_{L_{x}^{2}}\\
	&+\sum_{j=0}^{\infty}\sum_{|I|\le 4}\max_{0\le s\le t}\chi_{j}(s)\left\|\langle s+|x|\rangle\Gamma^{I}F(s,x)\right\|_{L_{x}^{2}}.
	\end{aligned}
	\end{equation}
\end{theorem}

As a consequence, we have the following simplified version of Theorem~\ref{thm:KGdecay}.
\begin{corollary}\label{coro:KGdecay}
	With the same settings as Theorem~\ref{thm:KGdecay}, let $0<\delta_{1}\ll 1$ and assume 
	\begin{equation}\label{est:assumf}
	\sum_{|I|\le 4}\max_{0\le s\le t}\langle s\rangle^{\delta_{1}}\|\langle s+|x|\rangle \Gamma^{I}F(s,x)\|_{L_{x}^{2}}\le C_{F},
	\end{equation}
	then we have 
	\begin{equation}\label{est:KGLINI}
	\langle t+r\rangle |u(t,x)|\lesssim C_{F}+\sum_{|I|\le 5}\left\| \log (2+|x|)\langle |x|\rangle\Gamma^{I}u(0,x)\right\|_{L_{x}^{2}}.
	\end{equation}
\end{corollary}

\begin{proof}
	First, from~\eqref{est:assumf} and the definition of $\left\{\chi_{j}\right\}_{j=0}^{\infty}$, we have
	\begin{equation*}
	\begin{aligned}
	&\sum_{|I|\le 4}\max_{0\le s\le t}\chi_{j}(s)\left\|\langle s+|x|\rangle \Gamma^{I}F(s,x)\right\|_{L_{x}^{2}}\\
	&\lesssim \sum_{|I|\le 4}\left(\max_{0\le s\le t}\chi_{j}(s)\langle s\rangle^{-\delta_{1}}\right)\left(\max_{0\le s\le t}\langle s\rangle^{\delta_{1}}\|\langle s+|x|\rangle \Gamma^{I}F(s,x)\|_{L_{x}^{2}}\right)\le C_{F}2^{-j\delta_{1}},
	\end{aligned}
	\end{equation*}
	which implies
	\begin{equation*}
	\sum_{j=0}^{\infty}\sum_{|I|\le 4}\max_{0\le s\le t}\chi_{j}(s)\left\|\langle s+|x|\rangle \Gamma^{I}F(s,x)\right\|_{L_{x}^{2}}\lesssim C_{F}\sum_{j=0}^{\infty}2^{-j\delta_{1}}\lesssim \frac{C_{F}}{1-2^{-\delta_{1}}}.
	\end{equation*}
	Second, using again the definition of $\chi_{j}$, we have 
	\begin{equation*}
	\begin{aligned}
	&\sum_{|I|\le 5}\left\|\langle x\rangle \chi_{j}(|x|)\Gamma^{I}u(0,x)\right\|_{L_{x}^{2}}\\
	&\lesssim \sum_{|I|\le 5}\left\|\chi_{j}(|x|)\log (2+|x|)\langle x\rangle \Gamma^{I}u(0,x)\right\|_{L_{x}^{2}}\left\|\log (2+|x|)\textbf{1}_{{\rm{supp}}\chi_{j}}(x)\right\|_{L_{x}^{\infty}}\\
	&\lesssim \frac{1}{j+1}\sum_{|I|\le 5}\left\|\chi_{j}(|x|)\log (2+|x|)\langle x\rangle \Gamma^{I}u(0,x)\right\|_{L_{x}^{2}}.
	\end{aligned}
	\end{equation*}
	Based on the above inequality and the Cauchy-Schwarz inequality, we have 
	\begin{equation*}
	\begin{aligned}
	\sum_{j=0}^{\infty}\sum_{|I|\le 5}\left\|\langle x\rangle \chi_{j}(|x|)\Gamma^{I}u(0,x)\right\|_{L_{x}^{2}}
	\lesssim \sum_{|I|\le 5}\left\|\log (2+|x|)\langle x\rangle \Gamma^{I}u(0,x)\right\|_{L_{x}^{2}}.
	\end{aligned}
	\end{equation*}
	Combining the above estimates with Theorem~\ref{thm:KGdecay}, we obtain~\eqref{est:KGLINI}.
\end{proof}

Next, we recall the extra decay inside of a cone of a Klein-Gordon component from~\cite{DongMa}.
For the sake of completeness and the reader's convenience, we repeat the proof of the following Lemma which is inspired by the previous works~\cite{Hormander,Klainerman93}.
\begin{lemma}[\cite{DongMa,Hormander,Klainerman93}]
	Suppose $u=u(t,x)$ is the solution to the Klein-Gordon equation
	\begin{equation}\label{equ:KG}
	\begin{aligned}
	-\Box u(t,x)+u(t,x)=F(t,x),
	\end{aligned}
	\end{equation}
	where $F=F(t,x):\R^{1+2}\to \R$ is a sufficiently nice function, then we have 
	\begin{equation}\label{est:decayKG}
	|u|\lesssim \frac{\langle t-r\rangle}{\langle t+r\rangle}\left(|\partial\Gamma u|+|\partial u|\right)+|F|,\quad \mbox{for}\ r\le 3t+3.
	\end{equation}
\end{lemma}

\begin{proof}
	Case I. Let $t\in [0,1]$. From~\eqref{equ:KG}, we obtain
	\begin{equation}\label{est:KGdecay1}
	\left|u\right|\leq\left|\Box u\right|+\left|F\right|\lesssim \frac{\langle t-r\rangle}{\langle t+r\rangle}|\partial\Gamma u|+|F|,\quad \mbox{for}\ r\le 3t+3.
	\end{equation}
	
	Case II. Let $t\in (1,\infty)$.~Recall that,~we can rewrite the d'Alembert operator $-\Box$ as 
	\begin{equation*}
	-\Box=\frac{(t-r)(t+r)}{t^{2}}\pt \pt +\frac{x^{a}}{t^{2}}\pt L_{a}+\frac{2}{t}\pt -\frac{1}{t}\partial^{a}L_{a}-\frac{x^{a}}{t^{2}}\partial_{a}.
	\end{equation*}
	Hence, from the triangle inequality, we obtain 
	\begin{equation}\label{est:KGdecay2}
	|u|\lesssim |\Box u|+|F|\lesssim \frac{\langle t-r\rangle}{t}\left(|\partial \Gamma u|+|\partial u|\right)+|F|,\quad \mbox{for}\ r\le 3t+3.
	\end{equation}
	We see that~\eqref{est:decayKG} follows from~\eqref{est:KGdecay1} and~\eqref{est:KGdecay2}.
	
\end{proof}

Then, we introduce the standard energy and the conformal energy for the 2D wave equation,
\begin{equation*}
\begin{aligned}
\mathcal{E}_{0}(t,u)&=\int_{\R^{2}}\left((\pt u)^{2}+(\partial_{1} u)^{2}+(\partial_{2} u)^{2}\right)(t,x)\d x,\\
\mathcal{F}(t,u)&=\int_{\R^{2}}\bigg((Su+u)^{2}+\sum_{a=1,2}(L_{a}u)^{2}+(\Omega u)^{2}\bigg)\d x.
\end{aligned}
\end{equation*}
Now, we recall the following energy estimates for the 2D wave equation from~\cite{Alin}.
\begin{lemma}[\cite{Alin}]\label{le:l2conform}
		Suppose $u=u(t,x)$ is the solution to the Cauchy problem
	\begin{align}
	-\Box u=F\quad \mbox{with}\quad (u,\pt u)_{|t=0}=(u_{0},u_{1}),
	\end{align}
	where $F=F(t,x):\R^{1+2}\to \R$ is a sufficiently nice function, 
	then the following estimates are true.
	\begin{enumerate}
		\item {\rm{Standard energy estimate.}} It holds 
		\begin{equation}\label{est:Energywave}
		\mathcal{E}_{0}(t,u)^{\frac{1}{2}}\lesssim \mathcal{E}_{0}(0,u)^{\frac{1}{2}}+\int_{0}^{t}\|F(s,x)\|_{L_{x}^{2}}\d s.
		\end{equation}
		
		\item {\rm{Conformal energy estimate.}} It holds
	\begin{equation}\label{est:Con}
	\mathcal{F}(t,u)^{\frac{1}{2}}\lesssim \mathcal{F}(0,u)^{\frac{1}{2}}+\int_{0}^{t}\|\langle r+s\rangle F(s,x)\|_{L_{x}^{2}}\d s.
	\end{equation} 
\end{enumerate}
\end{lemma}
\begin{proof}
	Proof of (i). The proof is similar to~\eqref{est:EnergyKG} and it is omitted.
	
	Proof of (ii). The proof relies on the use of nonspacelike multiplier  $K_{0}=(r^{2}+t^{2})\pt +2rt\partial_{r}$.
	In what follows, we provide a sketch of the proof and refer to~\cite[Theorem 6.11]{Alin} for more details. 
	
	Note that, a direct computation yields,
	\begin{equation*}
	\begin{aligned}
	-\Box u(tu+K_{0}u)
	&=\frac{1}{2}\pt \bigg((Su+u)^{2}+\sum_{a=1,2}(L_{a}u)^{2}+(\Omega u)^{2}\bigg)\\
	&-\partial^{a}\left(2rt(\partial_{a}u)(\partial_{r}u)+(r^{2}+t^{2})(\partial_{a}u)(\partial_{t}u)+\frac{1}{2}\pt(x_{a}u^{2})\right)\\
	&-\partial^{a}\left(tu\partial_{a}u-tx_{a}\left(|\partial u|^{2}-2(\pt u)^{2}\right)\right)=F(tu+K_{0}u).
	\end{aligned}
	\end{equation*}
	Integrating the above identity over the domain $[0,t]\times \R^{2}$ and then using the Cauchy-Schwarz inequality, we find
	\begin{equation*}
	\mathcal{F}(t,u)-\mathcal{F}(0,u)\lesssim \int_{0}^{t}\left\|\langle s+r\rangle F(s,x)\right\|_{L_{x}^{2}}\left\|\langle s+r\rangle^{-1}(su+K_{0}u)\right\|_{L_{x}^{2}}\d s.
	\end{equation*}
	Note also that, from the definition of $K_{0}$, we have 
	\begin{equation*}
	|tu+K_{0}u|=|t(Su+u)+r(t\partial_{r}u+r\partial_{t}u)|\lesssim \langle t+r\rangle \left(|Su+u|+|L_{1}u|+|L_{2}u|\right).
	\end{equation*}
	Gathering these estimates, we obtain~\eqref{est:Con}.
	\end{proof}

Last, we introduce the following $L_{x}^{2}$ estimate for the 2D wave equation.
\begin{lemma}\label{le:Lx2wave}
		Suppose $u=u(t,x)$ is the solution to the Cauchy problem
		\begin{equation*}
	-\Box u=F\quad \mbox{with}\quad (u,\pt u)_{|t=0}=(u_{0},u_{1}),
	\end{equation*}
	where $F=F(t,x):\R^{1+2}\to \R$ is a sufficiently nice function, 
	then we have 
	\begin{equation}\label{est:Lx2}
	\begin{aligned}
\|u(t,x)\|_{L_{x}^{2}}&\lesssim \|u_{0}\|_{L_{x}^{2}}+\log^{\frac{1}{2}}(2+t)\left(\|u_{1}\|_{L_{x}^{1}}+\|u_{1}\|_{L_{x}^{2}}\right)\\
&+ \log^{\frac{1}{2}}(2+t)\int_{0}^{t}\left(\left\|F(s,x)\right\|_{L_{x}^{1}}+\left\|F(s,x)\right\|_{L_{x}^{2}}\right)\d s.
\end{aligned}
	\end{equation} 
\end{lemma}

\begin{proof}
	Recall that, from the Duhamel's principle, we have the following expression of the solution $u$,
	\begin{equation}\label{equ:expressionwaveu}
	u(t)=\cos (t\sqrt{-\Delta})u_{0}+\frac{\sin (t\sqrt{-\Delta})}{\sqrt{-\Delta}}u_{1}+\int_{0}^{t}\frac{\sin ((t-s)\sqrt{-\Delta})}{\sqrt{-\Delta}}F(s)\d s.
	\end{equation}
	We claim that 
	\begin{equation}\label{est:cossin}
	\big\|\cos (t\sqrt{-\Delta})\big\|_{L_{x}^{2}\to L_{x}^{2}}+\log^{-\frac{1}{2}}(2+t)
	\left\|\frac{\sin (t\sqrt{-\Delta})}{\sqrt{-\Delta}}\right\|_{L_{x}^{1}\cap L_{x}^{2}\to L_{x}^{2}}\lesssim 1.
	\end{equation}
	Indeed, for any Schwartz function $f\in \mathcal{S}(\R^{2})$, from the polar coordinate transformation and the Plancherel theorem, we have 
	\begin{equation*}
	\big\|\cos (t\sqrt{-\Delta})f\big\|_{L_{x}^{2}}=\big\|\cos (t|\xi|)\hat{f}(\xi)\big\|_{L_{\xi}^{2}}\lesssim \big\|\hat{f}(\xi)\big\|_{L_{\xi}^{2}}\lesssim \|f\|_{L_{x}^{2}},\qquad \qquad \quad \qquad\qquad
	\end{equation*}
	\begin{equation*}
	\begin{aligned}
	\left\|\frac{\sin (t\sqrt{-\Delta})}{\sqrt{-\Delta}}f\right\|_{L_{x}^{2}}
	&\lesssim \left\|\frac{\sin (t|\xi|)}{|\xi|}\hat{f}(\xi)\textbf{1}_{\{|\xi|\le 1\}}\right\|_{L_{\xi}^{2}}+\left\|\frac{\sin (t|\xi|)}{|\xi|}\hat{f}(\xi)\textbf{1}_{\{|\xi|\ge 1\}}\right\|_{L_{\xi}^{2}}\\
	&\lesssim \|\hat{f}(\xi)\|_{L_{\xi}^{\infty}}\left(\int_{0}^{t}\frac{\sin^{2}r}{r}\d r\right)^{\frac{1}{2}}+\|\hat{f}(\xi)\|_{L_{\xi}^{2}}\left\|\frac{\sin (t|\xi|)}{|\xi|}\textbf{1}_{\{|\xi|\ge 1\}}\right\|_{L_{\xi}^{\infty}}\\
	&\lesssim \|f\|_{L_{x}^{2}}+\log^{\frac{1}{2}}(2+t)\|f\|_{L_{x}^{1}}\lesssim \log^{\frac{1}{2}}(2+t)\left(\|f\|_{L_{x}^{1}}+\|f\|_{L_{x}^{2}}\right).
	\end{aligned}
	\end{equation*}
	which means~\eqref{est:cossin}. Last, combining~\eqref{equ:expressionwaveu} and~\eqref{est:cossin}, we obtain~\eqref{est:Lx2}.
	\end{proof}

\subsection{Hidden structures within the Dirac--Klein-Gordon system}\label{Se:Hidden}
In this subsection, we introduce several hidden structures related to the massless spinor $\psi$, the scalar field $v$ and the nonlinear term $\psi^{*}\gamma^{0}\psi$. These hidden structures can be used to obtain several refined energy estimates from a bootstrap assumption.

Following~\cite[Lemma 3]{Boura00} or \cite[Section 5]{Boura}, for any solution $(\psi,v)$ to the 2D Dirac--Klein-Gordon system~\eqref{equ:DKG}, we denote $\Psi=\Psi(t,x)$ by the solution to the following equation, 
\begin{equation}\label{equ:defPsi}
\left\{\begin{aligned}
-\Box\Psi(t,x)&=i\gamma^{\mu}\partial_{\mu}\psi(t,x),\quad \mbox{for}\ (t,x)\in [0,\infty)\times \R^{2},\\
(\Psi,\pt \Psi)_{|t=0}&=(0,i\gamma^{0}\psi),\quad \quad \ \mbox{for}\ x\in \R^{2}.
\end{aligned}\right.
\end{equation}
Then $-i\gamma^{\mu}\partial_{\mu}\Psi=\psi$, since $-i\gamma^{\nu}\partial_{\nu}\left(i\gamma^{\mu}\partial_{\mu}\Psi+\psi\right)=0$ and $\left(i\gamma^{\mu}\partial_{\mu}\Psi+\psi\right)_{|t=0}=0$.

We introduce the following cubic terms that will be present in the hidden structures of the 2D Dirac--Klein-Gordon system,
\begin{equation}\label{equ:Ns00}
\begin{aligned}
\mathcal{N}_{1}(\psi,v)&=i\left(v\gamma^{\mu}\partial_{\mu}(v\psi)\right),\quad \mathcal{N}_{2}(\psi,\psi^{*})=\left(\left(\psi^{*}\gamma^{0}\psi\right)\psi\right),\\
\mathcal{N}_{3}(\psi,v)&=i\left(\partial_{\mu}(v\psi^{*})\gamma^{0}\gamma^{\mu}\psi\right)-i\left(\psi^{*}\gamma^{0}\gamma^{\mu}\partial_{\mu}(v\psi)\right).
\end{aligned}
\end{equation}
We also introduce the following quadratic term which enjoys a null structure,
\begin{equation}\label{equ:Ns01}
\mathcal{N}_{4}(\psi,\psi^{*})=2\eta^{\alpha\beta}\partial_{\alpha}\psi^{*}\gamma^{0}\partial_{\beta}\psi=2\partial_{\alpha}\psi^{*}\gamma^{0}\partial^{\alpha}\psi.
\end{equation}

First, we discuss the hidden structures within the 2D Dirac--Klein-Gordon system.
\begin{lemma}\label{lem:transfm}
	Suppose $(\psi,v)=(\psi,v)(t,x)$ is the solution to the Cauchy problem~\eqref{equ:DKG} with the initial data~\eqref{equ:initial} at $t=0$, then the following identities hold.
	\begin{enumerate}
		\item \emph{Nonlinear transformation of $\psi$.} Let $\widetilde{\psi}=\psi+i\gamma^{\mu}\partial_{\mu}(v\psi)$, then we have 
		\begin{equation}\label{equ:Hiddenpsi}
		-i\gamma^{\mu}\partial_{\mu}\widetilde{\psi}={\mathcal{N}}_{1}(\psi,v)+{\mathcal{N}}_{2}(\psi,\psi^{*})+2Q_{0}(\psi,v).
		\end{equation}
		
		\item \emph{Nonlinear transformation of $\Psi$.} Let $\widetilde{\Psi}=\Psi-v\psi$, then we have 
		\begin{align}\label{equ:HiddenPsi}
		-\Box \widetilde{\Psi}=-\mathcal{N}_{1}(\psi,v)-\mathcal{N}_{2}(\psi,\psi^{*})-2Q_{0}(\psi,v).
		\end{align}
	    \item \emph{Nonlinear transformation of $v$.} Let $\tilde{v}=v-\psi^{*}\gamma^{0}\psi$, then we have 
	    \begin{align}\label{equ:Hiddenv}
	    -\Box\tilde{v}+\tilde{v}=\mathcal{N}_{3}(\psi,v)+\mathcal{N}_{4}(\psi,\psi^{*}).
	    \end{align}
	    
	    \item \emph{Hidden structure of $\big[\widehat{\Gamma}^{I}\psi\big]_{-}$.} We have 
	    \begin{equation}\label{equ:Hiddenpsi-}
	    \big[\widehat{\Gamma}^{I}\psi\big]_{-}=-i\left(I_{2}-\omega_{a}\gamma^{0}\gamma^{a}\right)\gamma^{b}G_{b}\widehat{\Gamma}^{I}\Psi,\quad \mbox{for}\ I\in \mathbb{N}^{6}.
	    \end{equation}
	   
	\end{enumerate}
\end{lemma}

\begin{remark}
	The nonlinear transformations in Lemma \ref{lem:transfm} are in the spirit of normal forms of Shatah \cite{Shatah}; see also \cite{Tsutsumi, DongWyatt}. These transformations are favorable as they reveal a hidden null structure enjoyed by the nonlinearities, i.e., they transform the original nonlinearities (slowly decaying terms) to the sum of quadratic null terms and cubic terms (fast decaying terms).
\end{remark}

\begin{proof}[Proof of Lemma~\ref{lem:transfm}]
	Proof of (i). From~\eqref{equ:DKG} and $(i\gamma^{\mu}\partial_{\mu})^{2}=\Box$, we find 
	\begin{equation*}
	\begin{aligned}
	-i\gamma^{\mu}\partial_{\mu}\widetilde{\psi}
	&=-i\gamma^{\mu}\partial_{\mu}{\psi}-\Box(v\psi)\\
	&=(-\Box v+v)\psi-v((i\gamma^{\mu}\partial_{\mu})^{2}\psi)+2Q_{0}(\psi,v)\\
	&=(\psi^{*}\gamma^{0}\psi)\psi+i\left(v\gamma^{\mu}\partial_{\mu}(v\psi)\right)+2Q_{0}(\psi,v),
	\end{aligned}
	\end{equation*}
	which means~\eqref{equ:Hiddenpsi}.
	
	Proof of (ii). Combining~\eqref{equ:Hiddenpsi} and the definition of $\Psi$ in~\eqref{equ:defPsi}, we obtain~\eqref{equ:HiddenPsi}.
	
	Proof of (iii). Using again~\eqref{equ:DKG},~\eqref{equ:gamma} and $(i\gamma^{\mu}\partial_{\mu})^{2}=\Box$, we have 
	\begin{equation*}
	\begin{aligned}
	-\Box \tilde{v}+\tilde{v}
	&=(-\Box v+v)-\psi^{*}\gamma^{0}\psi+\Box(\psi^{*}\gamma^{0}\psi)\\
	&=i\left(\partial_{\mu}(v\psi^{*})\gamma^{0}\gamma^{\mu}\psi\right)-i\left(\psi^{*}\gamma^{0}\gamma^{\mu}\partial_{\mu}(v\psi)\right)+2\partial_{\alpha}\psi^{*}\gamma^{0}\partial^{\alpha}\psi,
	\end{aligned}
	\end{equation*}
	which means~\eqref{equ:Hiddenv}.
	
	Proof of (iv). From~\eqref{equ:gamma}~\eqref{def:phi-}, $-i\gamma^{\mu}\partial_{\mu}\widehat{{\Gamma}}^{I}\Psi=\widehat{{\Gamma}}^{I}\psi$ and $G_{a}=\partial_{a}+\omega_{a}\pt$, we have
	\begin{equation*}
	\begin{aligned}
	\big[\widehat{{\Gamma}}^{I}\psi\big]_{-}
	&=-i(\gamma^{0}+\omega_{a}\gamma^{a})\partial_{t}\widehat{{\Gamma}}^{I}\Psi
	-i\left(I_{2}-\omega_{a}\gamma^{0}\gamma^{a}\right)\gamma^{b}\partial_{b}\widehat{{\Gamma}}^{I}\Psi\\
	&=-i\left(\gamma^{0}+\omega_{a}\omega_{b}\gamma^{0}\gamma^{a}\gamma^{b}\right)\partial_{t}\widehat{{\Gamma}}^{I}\Psi-i\left(I_{2}-\omega_{a}\gamma^{0}\gamma^{a}\right)\gamma^{b}G_{b}\widehat{\Gamma}^{I}\Psi.
	\end{aligned}
	\end{equation*}
	Using again~\eqref{equ:gamma}, we have 
	\begin{equation*}
	\gamma^{0}+\omega_{a}\omega_{b}\gamma^{0}\gamma^{a}\gamma^{b}=(1-\omega_{1}^{2}-\omega_{2}^{2})\gamma^{0}+\sum_{0<a<b}\omega^{a}\omega^{b}\gamma^{0}(\gamma^{a}\gamma^{b}+\gamma^{b}\gamma^{a})=0.
	\end{equation*}
	Combining the above identities, we obtain~\eqref{equ:Hiddenpsi-}.
	\end{proof}

Second, we deduce a hidden structure for the nonlinear term $\psi^{*}\gamma^{0}\psi$; see for instance \cite{Boura} or the re-interpretation in \cite{DongWyatt} in the hyperboloidal foliation of the spacetime.
\begin{lemma}\label{lem:hidden}
	Let $\Phi_{1}$ and $\Phi_{2}$ be two $\C^{2}$-valued functions on $\R^{1+2}$, then we have 
	\begin{equation}\label{equ:HiddenPhi}
	\Phi_{1}^{*}\gamma^{0}\Phi_{2}=\frac{1}{4}\left([\Phi_{1}]_{-}^{*}\gamma^{0}[\Phi_{2}]_{-}+[\Phi_{1}]_{-}^{*}\gamma^{0}[\Phi_{2}]_{+}+[\Phi_{1}]_{+}^{*}\gamma^{0}[\Phi_{2}]_{-}\right).
	\end{equation}
\end{lemma}
\begin{proof}
	From the definition of $[\cdot]_{+}$ and $[\cdot]_{-}$ in~\eqref{def:phi-}, we see that 
	\begin{equation*}
	2\Phi_{1}=[\Phi_{1}]_{+}+[\Phi_{1}]_{-}\quad \mbox{and}\quad 
	2\Phi_{2}=[\Phi_{2}]_{+}+[\Phi_{2}]_{-}.
	\end{equation*}
	Hence, by an elementary computation, we find 
	\begin{equation*}
	\Phi_{1}^{*}\gamma^{0}\Phi_{2}=\frac{1}{4}\left([\Phi_{1}]_{-}^{*}\gamma^{0}[\Phi_{2}]_{-}+[\Phi_{1}]_{-}^{*}\gamma^{0}[\Phi_{2}]_{+}+[\Phi_{1}]_{+}^{*}\gamma^{0}[\Phi_{2}]_{-}+[\Phi_{1}]_{+}^{*}\gamma^{0}[\Phi_{2}]_{+}\right).
	\end{equation*}
	Note that, for the last term in the above identity, from~\eqref{equ:gamma}, we have 
	\begin{equation*}
	\begin{aligned}
	[\Phi_{1}]_{+}^{*}\gamma^{0}[\Phi_{2}]_{+}
	=&\left(\Phi_{1}^{*}+\omega_{a}\Phi_{1}^{*} \gamma^{0}\gamma^{a}\right)\gamma^{0}\left(\Phi_{2}+\omega_{b} \gamma^{0}\gamma^{b}\Phi_{2}\right)\\
	=&\left(1-\omega_{1}^{2}-\omega_{2}^{2}\right)\Phi_{1}^{*}\gamma^{0}\Phi_{2}+\omega_{a}\Phi_{1}^{*}\gamma^{0}\left(\gamma^{0}\gamma^{a}+\gamma^{a}\gamma^{0}\right)\Phi_{2}\\
	&+\omega_{1}\omega_{2}\Phi_{1}^{*}\gamma^{0}\left(\gamma^{1}\gamma^{2}+\gamma^{2}\gamma^{1}\right)\Phi_{2}=0.
	\end{aligned}
	\end{equation*}
	Gathering the above identities, we have proved~\eqref{equ:HiddenPhi}.
	\end{proof}

\section{Proof of Theorem~\ref{thm:main}}
In this section, we prove the existence of a solution $(\psi,v)$ of~\eqref{equ:DKG} satisfying~\eqref{est:thmpoint} and~\eqref{est:thmsca} in Theorem~\ref{thm:main}. The proof is based on a bootstrap argument of the high-order energy and pointwise decay for solution $(\psi,v)$.

\subsection{Bootstrap Setting}\label{Se:Boot}
Fix $0<\delta\ll 1$. To prove Theorem~\ref{thm:main}, we introduce the following bootstrap assumption of $(\psi,v)$: for $C>1$ and $0<\varepsilon\ll C^{-1}$ to be chosen later, 
\begin{equation}\label{est:Bootpsi}
\left\{\begin{aligned}
\sum_{|I|\le N}\int_{0}^{t}\langle s \rangle^{-3\delta}\int_{\R^{2}}\frac{\big[\widehat{\Gamma}^{I}\psi\big]_{-}^{2}}{\langle r-s\rangle^{\frac{6}{5}}}\d x \d s&\le C\varepsilon,\\
\sum_{|I|\le N}\langle t\rangle^{-\delta}\|\langle r-t\rangle\chi(r-2t)|\widehat{\Gamma}^{I}\psi|\|_{L_{x}^{2}}&\le C\varepsilon,\\
\sum_{|I|\le N-1}\mathcal{E}^{D}(t,\widehat{\Gamma}^{I}\psi)^{\frac{1}{2}}+\sum_{|I|\le N}\langle t\rangle^{-\delta}\mathcal{E}^{D}(t,\widehat{\Gamma}^{I}\psi)^{\frac{1}{2}}&\le C\varepsilon,\\
{\sum_{|I|\le N-7}\langle t+r\rangle^{\frac{3}{2}-\delta} \big|\big[\widehat{\Gamma}^{I}\psi\big]_{-}\big|}+\sum_{|I|\le N-7}\langle t+r\rangle^{\frac{1}{2}-\delta}\langle t-r\rangle\left|\widehat{\Gamma}^{I}\psi\right|&\le C\varepsilon,
\end{aligned}\right.
\end{equation}
\begin{equation}\label{est:Bootv}
\left\{\begin{aligned}
\sum_{|I|\le N}\int_{0}^{t}\langle s \rangle^{-5\delta}\int_{\R^{2}}\frac{\left(\Gamma^{I}v\right)^{2}}{\langle r-s\rangle^{\frac{6}{5}}}\d x \d s&\le C\varepsilon,\\
\sum_{|I|\le N-1}\mathcal{G}_{1}(t,\Gamma^{I}v)^{\frac{1}{2}}+\sum_{|I|\le N}\langle t\rangle^{-2\delta}\mathcal{G}_{1}(t,\Gamma^{I}v)^{\frac{1}{2}}&\le C\varepsilon,\\
\sum_{|I|\le N-6}\sup_{x\in \R^{2}}\langle r+t\rangle \left|\Gamma^{I}v(t,x)\right|+\sum_{|I|\le N}\|\langle r-t\rangle\chi(r-2t)\Gamma^{I}v\|_{L_{x}^{2}}&\le C\varepsilon.
\end{aligned}\right.
\end{equation}

For all initial data $(\psi_{0},\vec{v}_{0})$ satisfying~\eqref{est:Bootpsi} and~\eqref{est:Bootv}, we set
\begin{equation}\label{def:T}
T^{*}(\psi_{0},\vec{v}_{0})=\sup\left\{ t\in [0,\infty): (\psi,v)\ \mbox{satisfy}~\eqref{est:Bootpsi}\ \mbox{and}~\eqref{est:Bootv}\ \mbox{on}\ [0,t]\right\}.
\end{equation}

The heart of the proof of Theorem~\ref{thm:main} is the following proposition.

\begin{proposition}\label{pro:main}
For all initial data $(\psi_{0},\vec{v}_{0})$ satisfying the smallness condition~\eqref{est:smallness} in Theorem~\ref{thm:main}, we have $T^{*}(\psi_{0},\vec{v}_{0})=+\infty$.
\end{proposition}

The rest of this section is organized as follows. First, in \S\ref{Se:Pointwise}-\S\ref{Se:Nonlin}, we establish pointwise estimates and spacetime norm estimates for $(\psi,v)$ and nonlinear terms, respectively. Second, in~\S\ref{Se:End}, we devote to the proof of Proposition~\ref{pro:main} based on the vector field and energy methods. Last, in~\S\ref{Se:Endthm}, we prove Theorem~\ref{thm:main} from Proposition~\ref{pro:main}. In the sequel,  the implied constants in $\lesssim$ do not depend on the constant $C$ and $\varepsilon$ appearing in the bootstrap assumption~\eqref{est:Bootpsi}-\eqref{est:Bootv}.

\subsection{Pointwise estimates under the bootstrap assumption}\label{Se:Pointwise}

Note that 
from \eqref{est:hatGa}, \eqref{est:GloSob}, and the bootstrap assumption~\eqref{est:Bootpsi}-\eqref{est:Bootv}, the following global estimate for $(\psi,v)$ is true:
\begin{equation}\label{est:pointGlobal}
\sum_{|I|\le N-4}\big(|\widehat{{\Gamma}}^{I}\psi(t,x)|+|\Gamma^{I}v(t,x)|\big)\lesssim C\varepsilon\langle t+r\rangle^{-\frac{1}{2}}.
\end{equation}

Consider the following interior and exterior regions:
\begin{equation}
\begin{aligned}
\mathcal{C}_{\rm{int}}&=\left\{(t,x)\in [0,T^{+}(\psi_{0},\vec{v}_{0}))\times \R^{2}:r\le 3t+3\right\},\\
\mathcal{C}_{\rm{ext}}&=\left\{(t,x)\in [0,T^{+}(\psi_{0},\vec{v}_{0}))\times \R^{2}:r\ge 2t+3\right\}.
\end{aligned}
\end{equation}
For future reference, we state the following extra pointwise decay for $(\psi,v)$.
\begin{lemma}\label{le:point}
	For all $t\in [0,T^{*}(\psi_{0},\vec{v}_{0}))$, the following estimates hold.
	\begin{enumerate}
		\item {\rm Extra decay of $\psi$ in the interior region $\mathcal{C}_{\rm{int}}$.} We have
		\begin{align}
		&\sum_{|I|\le N-8}\left|\partial\widehat{\Gamma}^{I}\psi\right|{\bchar}_{\mathcal{C}_{\rm{int}}}\lesssim C\varepsilon\langle t-r\rangle^{-2}\langle t+r\rangle^{-\frac{1}{2}+\delta},\label{est:intpsilow}\\
		\sum_{|I|\le N-2}\big|\partial \widehat{\Gamma}^{I}&\psi\big|\bchar_{\mathcal{C}_{\rm{int}}}\lesssim \frac{1}{\langle t-r\rangle}\sum_{|I|\le N-1}\big|\widehat{{\Gamma}}^{I}\psi\big|+\frac{\langle t\rangle^{\frac{1}{2}+\delta}}{\langle t-r\rangle^{2}}\sum_{|I|\le N-2}\left|\Gamma^{I}v\right|\label{est:intpsiN-2},\\
		\sum_{|I|\le N-1}\big|\partial \widehat{\Gamma}^{I}&\psi\big|\bchar_{\mathcal{C}_{\rm{int}}}\lesssim \frac{1}{\langle t-r\rangle}\sum_{|I|\le N}\big|\widehat{{\Gamma}}^{I}\psi\big|+\frac{\langle t\rangle^{\frac{1}{2}+\delta}}{\langle t-r\rangle^{2}}\sum_{|I|\le N-1}\left|\Gamma^{I}v\right|.\label{est:intpsihigh}
		\end{align}
		\item {\rm Extra decay of $v$ in the interior region $\mathcal{C}_{\rm{int}}$.} We have 
		\begin{align}
		\sum_{|I|\le N-8}\left|\Gamma^{I}v\right| {\bchar}_{\mathcal{C}_{\rm{int}}}&\lesssim C\varepsilon \frac{\langle t-r\rangle}{\langle t+r\rangle^{2}} + C\varepsilon\langle t+r\rangle^{-2+\delta},\label{est:intvlow}\\
		\sum_{|I|\le N-1}| \Gamma^{I}v|\bchar_{\mathcal{C}_{\rm{int}}}&\lesssim \sum_{|I|\le N}\frac{\langle t-r\rangle}{\langle t+r\rangle}\left|\partial \Gamma^{I}v\right|\nonumber\\
		&+\sum_{|I|\le N-1}\left(\langle t\rangle^{-\frac{3}{2}+\delta}\big|\widehat{{\Gamma}}^{I}\psi\big|+\langle t\rangle^{-\frac{1}{2}}\big|\big[\widehat{{\Gamma}}^{I}\psi\big]_{-}\big|\right)\label{est:intvhigh}.
		\end{align}

		\item {\rm Extra decay of $(\psi,v)$ in the exterior region $\mathcal{C}_{\rm{ext}}$.} We have 
		\begin{equation}\label{est:outpsiv}
		\sum_{|I|\le N-2}\left(\left|\Gamma^{I}\psi\right|+\left|\Gamma^{I}v\right|\right)\bchar_{\mathcal{C}_{\rm{ext}}}\lesssim C\varepsilon\langle t+r\rangle^{-\frac{3}{2}+\delta}.
		\end{equation}
	\end{enumerate}
	
\end{lemma}

\begin{proof}
	Proof of (i). First, from~\eqref{equ:DKG},~\eqref{est:hatGa},~\eqref{est:decaypphi},~\eqref{est:Bootpsi} and~\eqref{est:Bootv}, we have
	\begin{equation*}
	\begin{aligned}
	\sum_{|I|\le N-8}\big|\partial \widehat{{\Gamma}}^{I}\psi\big|{\bchar}_{\mathcal{C}_{\rm{int}}}
	&\lesssim \frac{1}{\langle t-r\rangle}\sum_{|I|\le N-8}\big(\big|\widehat{{\Gamma}}\widehat{{\Gamma}}^{I}\psi\big|+\big|\widehat{{\Gamma}}^{I}\psi\big|\big){\bchar}_{\mathcal{C}_{\rm{int}}}\\
	&+\sum_{\substack{|I_{1}|\le N-8\\ |I_{2}|\le N-8}}\frac{t}{\langle t-r\rangle}\left|\Gamma^{I_{1}}v\right|\big|\widehat{{\Gamma}}^{I_{2}}\psi\big|{\bchar}_{\mathcal{C}_{\rm{int}}}\lesssim C\varepsilon \langle t+r\rangle^{-\frac{1}{2}+\delta}\langle t-r\rangle^{-2},
	\end{aligned}
	\end{equation*}
	which implies~\eqref{est:intpsilow}.
	
	Second, using again~\eqref{equ:DKG},~\eqref{est:hatGa},~\eqref{est:decaypphi},~\eqref{est:Bootpsi},~\eqref{est:Bootv} and $N\ge 14$, we have
	\begin{equation*}
	\begin{aligned}
	\sum_{|I|\le N-2}\big|\partial \widehat{{\Gamma}}^{I}\psi\big|{\bchar}_{\mathcal{C}_{\rm{int}}}
	&\lesssim \frac{1}{\langle t-r\rangle}\sum_{|I|\le N-2}\big(\big|\widehat{{\Gamma}}\widehat{{\Gamma}}^{I}\psi\big|+\big|\widehat{{\Gamma}}^{I}\psi\big|\big)\\
	&+\sum_{\substack{|I_{1}|\le N-2\\ |I_{2}|\le N-5}}\frac{t}{\langle t-r\rangle}\left|\Gamma^{I_{1}}v\right|\big|\widehat{{\Gamma}}^{I_{2}}\psi\big|+\sum_{\substack{|I_{1}|\le N-6\\ |I_{2}|\le N-2}}\frac{t}{\langle t-r\rangle}\left|\Gamma^{I_{1}}v\right|\big|\widehat{{\Gamma}}^{I_{2}}\psi\big|\\
	&\lesssim \frac{1}{\langle t-r\rangle}\sum_{|I|\le N-1}\big|\widehat{{\Gamma}}^{I}\psi\big|+\frac{\langle t\rangle^{\frac{1}{2}+\delta}}{\langle t-r\rangle^{2}}\sum_{|I|\le N-2}\left|\Gamma^{I}v\right|,
	\end{aligned}
	\end{equation*}
	which means~\eqref{est:intpsiN-2}. The proof of~\eqref{est:intpsihigh} is similar to~\eqref{est:intpsiN-2}, and we omit it.
	
	Proof of (ii). First, from~\eqref{equ:DKG},~\eqref{est:hatGG},~\eqref{est:decayKG},~\eqref{est:Bootpsi},~\eqref{est:Bootv} and~\eqref{est:pointGlobal}, we have 
	\begin{equation*}
	\begin{aligned}
	\sum_{|I|\le N-8}\left|\Gamma^{I}v\right|{\bchar}_{\mathcal{C}_{\rm{int}}}
	&\lesssim \frac{\langle t-r\rangle}{\langle t+r\rangle}\sum_{|I|\le N-8}\left(\left|\partial \Gamma\Gamma^{I}v\right|+\left|\partial \Gamma^{I}v\right|\right)\\
	&+\sum_{\substack{|I_{1}|\le N-8\\|I_{2}|\le N-8}}\big|\big[\widehat{{\Gamma}}^{I_{1}}\psi\big]_{-}\big|\big|\widehat{{\Gamma}}^{I_{2}}\psi\big|
	\lesssim C\varepsilon \frac{\langle t-r\rangle}{\langle t+r\rangle^{2}} + C\varepsilon\langle t+r\rangle^{-2+\delta},
	\end{aligned}
	\end{equation*}
	which means~\eqref{est:intvlow}.
	
	Second, using again~\eqref{equ:DKG},~\eqref{est:hatGG},~\eqref{est:decayKG},~\eqref{est:Bootpsi},~\eqref{est:Bootv},~\eqref{est:pointGlobal} and $N\ge 14$, we have 
	\begin{equation*}
	\begin{aligned}
	\sum_{|I|\le N-1}\left|\Gamma^{I}v\right|{\bchar}_{\mathcal{C}_{\rm{int}}}
	&\lesssim \frac{\langle t-r\rangle}{\langle t+r\rangle}\sum_{|I|\le N-1}\left(\left|\partial \Gamma^{I}\Gamma v\right|+\left|\partial \Gamma^{I}v\right|\right)\\
	&+\sum_{\substack{|I_{1}|\le N-7\\ |I_{2}|\le N-1}}\big|\big[\widehat{{\Gamma}}^{I_{1}}\psi\big]_{-}\big|\big|\widehat{{\Gamma}}^{I_{2}}\psi\big|
	+\sum_{\substack{|I_{1}|\le N-1\\ |I_{2}|\le N-5}}\big|\big[\widehat{{\Gamma}}^{I_{1}}\psi\big]_{-}\big|\big|\widehat{{\Gamma}}^{I_{2}}\psi\big|\\
	&\lesssim \frac{\langle t-r\rangle}{\langle t+r\rangle}\sum_{|I|\le N}\left|\partial \Gamma^{I}v\right|
	+\sum_{|I|\le N-1}\big(\langle t\rangle^{-\frac{3}{2}+\delta}\big|\widehat{{\Gamma}}^{I}\psi\big|+\langle t\rangle^{-\frac{1}{2}}\big|\big[\widehat{{\Gamma}}^{I}\psi\big]_{-}\big|\big),
	\end{aligned}
	\end{equation*}
	which means~\eqref{est:intvhigh}.
	
	Proof of (iii). Note that, from the definition of $\chi$ in~\eqref{def:chi}, we have 
	\begin{equation*}
	\begin{aligned}
	\sum_{|I|\le N-2}\big|\langle r-t\rangle\chi(r-2t-1)\widehat{{\Gamma}}^{I}\psi\big|&\lesssim \sum_{|I|\le N-2}\big|\langle r-t\rangle\chi(r-2t)\widehat{{\Gamma}}^{I}\psi\big|,\\
	\sum_{|I|\le N-2}\big|\Omega\big(\langle r-t\rangle\chi(r-2t-1)\widehat{{\Gamma}}^{I}\psi\big)\big|&\lesssim \sum_{|I|\le N-1}\big|\langle r-t\rangle\chi(r-2t)\widehat{{\Gamma}}^{I}\psi\big|,\\
	\sum_{|I|\le N-2}\big|\partial_{r}\big(\langle r-t\rangle\chi(r-2t-1)\widehat{{\Gamma}}^{I}\psi\big)\big|&\lesssim \sum_{|I|\le N-1}\big|\langle r-t\rangle\chi(r-2t)\widehat{{\Gamma}}^{I}\psi\big|,\\
	\sum_{|I|\le N-2}\big|\partial_{r}\big(\Omega\big(\langle r-t\rangle\chi(r-2t-1)\widehat{{\Gamma}}^{I}\psi\big)\big)\big|&\lesssim \sum_{|I|\le N}\big|\langle r-t\rangle\chi(r-2t)\widehat{{\Gamma}}^{I}\psi\big|.
	\end{aligned}
	\end{equation*}
	Combining the above estimates with~\eqref{est:Bootpsi}, we have 
	\begin{equation*}
	\begin{aligned}
	&\sum_{\substack{|I|\le 2\\ (i_{1},i_{2})\ne (2,0)}}\sum_{|I_{1}|\le N-2}\big\|\Lambda^{I}\big(\langle r-t\rangle\chi(r-2t-1)\widehat{{\Gamma}}^{I_{1}}\psi\big)\big\|_{L_{x}^{2}}\\
	&\lesssim \sum_{|I|\le N}\big\|\langle r-t\rangle\chi(r-2t)\widehat{{\Gamma}}^{I}\psi\big\|_{L_{x}^{2}}\lesssim C\varepsilon \langle t\rangle^{\delta}.
	\end{aligned}
	\end{equation*}
	Based on~\eqref{est:GloSobr} and the above estimate, we obtain~\eqref{est:outpsiv} for $\psi$. The proof of~\eqref{est:outpsiv} for $v$ follows from a similar argument, and it is omitted.
\end{proof}

\subsection{Nonlinear estimates under the bootstrap assumption}\label{Se:Nonlin}
In this subsection, we present several refined estimates of nonlinear terms under the bootstrap assumption~\eqref{est:Bootpsi}-\eqref{est:Bootv}. Note that, these estimates combined with the energy and conformal energy estimates that have been established in \S\ref{Se:Dirac}-\ref{Se:Wave} will be used to close the bootstrap for $(\psi,v)$ (see more details in \S\ref{Se:End}).

Note that, from the bootstrap assumption~~\eqref{est:Bootpsi}-\eqref{est:Bootv}, we have 
\begin{equation}\label{est:Ghostvweight}
\begin{aligned}
&\sum_{|I|\le N-1}\int_{0}^{t}\langle s\rangle^{-\frac{1}{2}-\frac{1}{2}\delta}\left(\left\|\frac{\Gamma^{I}v}{\langle s-r\rangle^{3\over 5}}\right\|_{L_{x}^{2}}+\left\|\frac{[\widehat{\Gamma}^{I}\psi]_{-}}{\langle s-r\rangle^{3\over 5}}\right\|_{L_{x}^{2}}\right)\d s\\
&\lesssim \sum_{|I|\le N-1}\left(\int_{0}^{t}\langle s\rangle^{-1-\delta}\d s \right)^{\frac{1}{2}}\left(\int_{0}^{t}\left\|\frac{\Gamma^{I}v}{\langle s-r\rangle^{3\over 5}}\right\|^{2}_{L_{x}^{2}}+\left\|\frac{[\widehat{\Gamma}^{I}\psi]_{-}}{\langle s-r\rangle^{3\over 5}}\right\|_{L_{x}^{2}}^{2}\d s\right)^{\frac{1}{2}}\lesssim C\varepsilon.
\end{aligned}
\end{equation}

We start with the $L_{t}^{1}L_{x}^{2}$ and weighted $L_{t}^{1}L_{x}^{2}$ estimates for ${\Gamma}^{I}Q_{0}(\psi,v)$.
\begin{lemma}\label{le:estQ0}
	For all $t\in [0,T^{*}(\psi_{0},\vec{v}_{0}))$, the following estimates hold.
	\begin{enumerate}
		
		\item {\rm{$L_{t}^{1}L_{x}^{2}$ estimate on ${\Gamma}^{I}Q_{0}$}}. It holds
		\begin{equation}\label{est:LtLxQ0}
		\sum_{|I|\le N-1}\int_{0}^{t}\|{\Gamma}^{I}Q_{0}(\psi,v)\|_{L_{x}^{2}}\d s \lesssim C^{2}\varepsilon^{2}.
		\end{equation}
		
		\item {\rm{Weighted $L_{t}^{1}L_{x}^{2}$ estimate on ${\Gamma}^{I}Q_{0}$}}. It holds  
		\begin{equation}\label{est:weightLtLxQ0}
		\sum_{|I|\le N-4}\int_{0}^{t}\left\|\langle s+r\rangle {\Gamma}^{I}Q_{0}(\psi,v)\right\|_{L_{x}^{2}}\d s\lesssim C^{2}\varepsilon^{2}\log (2+t).
		\end{equation}
		
	\end{enumerate}
\end{lemma}
\begin{proof}
	Proof of (i). From \eqref{est:GamQ0} and $1\le \bchar_{\mathcal{C}_{\rm{int}}}+\bchar_{\mathcal{C}_{\rm{ext}}}$, we infer that 
	\begin{equation*}
	\sum_{|I|\le N-1}\int_{0}^{t}\|{\Gamma}^{I}Q_{0}(\psi,v)\|_{L_{x}^{2}}\d s\lesssim 
	\mathcal{I}_{11}+\mathcal{I}_{12},
	\end{equation*}
	where
	\begin{equation*}
	\begin{aligned}
	\mathcal{I}_{11}&=\sum_{|I_{1}|+|I_{2}|\le N-1}\int_{0}^{t}\left\|Q_{0}\left({{\Gamma}}^{I_{1}}\psi,\Gamma^{I_{2}}v\right)\bchar_{\mathcal{C}_{\rm{int}}}\right\|_{L_{x}^{2}}\d s,\\
	\mathcal{I}_{12}&=\sum_{|I_{1}|+|I_{2}|\le N-1}\int_{0}^{t}\left\|Q_{0}\left({{\Gamma}}^{I_{1}}\psi,\Gamma^{I_{2}}v\right)\bchar_{\mathcal{C}_{\rm{ext}}}\right\|_{L_{x}^{2}}\d s.
	\end{aligned}
	\end{equation*}
	 Moreover, according to~\eqref{est:hatGa},~\eqref{est:hatparGa},~\eqref{est:Q0inside} and $N\ge 14$, we have
	\begin{equation*}
	\mathcal{I}_{11}\lesssim \mathcal{I}_{11}^{1}+\mathcal{I}_{11}^{2}+\mathcal{I}_{11}^{3}+\mathcal{I}_{11}^{4},
	\end{equation*}
	where
	\begin{equation*}
	\begin{aligned}
	\mathcal{I}_{11}^{1}&=\sum_{|I_{1}|\le N}\sum_{|I_{2}|\le N-4}\int_{0}^{t}\langle s \rangle^{-1}\big\|\big|\widehat{{\Gamma}}^{I_{1}}\psi\big|\left| \Gamma^{I_{2}}v\right|\bchar_{\mathcal{C}_{\rm{int}}}\big\|_{L_{x}^{2}}\d s,\\
	\mathcal{I}_{11}^{2}&=\sum_{|I_{1}|\le N-4 }\sum_{|I_{2}|\le N}\int_{0}^{t}\langle s \rangle^{-1}\big\|\big|\widehat{{\Gamma}}^{I_{1}}\psi\big|\left| \Gamma^{I_{2}}v\right|\bchar_{\mathcal{C}_{\rm{int}}}\big\|_{L_{x}^{2}}\d s,\quad \quad \quad \quad 
	\end{aligned}
	\end{equation*}
	\begin{equation*}
	\begin{aligned}
	\mathcal{I}_{11}^{3}&=\sum_{|I_{1}|\le N-8 }\sum_{|I_{2}|\le N-1}\int_{0}^{t}\langle s\rangle^{-1}\big\|\langle s-r\rangle\big|\partial \widehat{\Gamma}^{I_{1}}\psi\big|\left| \partial\Gamma^{I_{2}}v\right|\bchar_{\mathcal{C}_{\rm{int}}}\big\|_{L_{x}^{2}}\d s,\\
	\mathcal{I}_{11}^{4}&=\sum_{|I_{1}|\le N-1}\sum_{|I_{2}|\le N-7}\int_{0}^{t}\langle s\rangle^{-1}\big\|\langle s-r\rangle\big|\partial \widehat{\Gamma}^{I_{1}}\psi\big|\left| \partial\Gamma^{I_{2}}v\right|\bchar_{\mathcal{C}_{\rm{int}}}\big\|_{L_{x}^{2}}\d s.
	\end{aligned}
	\end{equation*}
	Using~\eqref{est:Bootpsi},~\eqref{est:Bootv},~\eqref{est:pointGlobal} and $0<\delta\ll 1$, we have  
	\begin{equation*}
	\begin{aligned}
	\mathcal{I}_{11}^{1}&\lesssim \sum_{\substack{|I_{1}|\le N\\ |I_{2}|\le N-4}}\int_{0}^{t}\langle s \rangle^{-1}\big\| \widehat{\Gamma}^{I_{1}}\psi\big\|_{L_{x}^{2}}\left\| \Gamma^{I_{2}}v\right\|_{L_{x}^{\infty}}\d s\lesssim C^{2}\varepsilon^{2}\int_{0}^{t}\langle s\rangle^{-\frac{3}{2}+\delta}\d s \lesssim C^{2}\varepsilon^{2},\\
	\mathcal{I}_{11}^{2}&\lesssim\sum_{\substack{|I_{1}|\le N-4\\ |I_{2}|\le N}}\int_{0}^{t}\langle s \rangle^{-1}\big\| \widehat{\Gamma}^{I_{1}}\psi\big\|_{L_{x}^{\infty}}\left\| \Gamma^{I_{2}}v\right\|_{L_{x}^{2}}\d s\lesssim C^{2}\varepsilon^{2}\int_{0}^{t}\langle s\rangle^{-\frac{3}{2}+\delta}\d s \lesssim C^{2}\varepsilon^{2}.
	\end{aligned}
	\end{equation*}
Then, from~\eqref{est:Bootpsi},~\eqref{est:Bootv},~\eqref{est:intpsihigh} and $0<\delta\ll 1$, we see that 
\begin{equation*}
\begin{aligned}
\mathcal{I}_{11}^{3}
&\lesssim \sum_{\substack{|I_{1}|\le N-8\\ |I_{2}|\le N-1}}\int_{0}^{t}\langle s\rangle^{-1}\big\|\langle s-r\rangle\big|\partial \widehat{\Gamma}^{I_{1}}\psi\big|\big\|_{L_{x}^{\infty}}\left\| \partial\Gamma^{I_{2}}v\right\|_{L_{x}^{2}}\d s\\
&\lesssim C^{2}\varepsilon^{2}\int_{0}^{t}\langle s\rangle^{-1}\langle s\rangle^{-\frac{1}{2}+\delta}\d s \lesssim C^{2}\varepsilon^{2}\int_{0}^{t}\langle s\rangle^{-\frac{3}{2}+\delta} \lesssim C^{2}\varepsilon^{2},\quad \quad \quad \quad 
\end{aligned}
\end{equation*}
	\begin{equation*}
	\begin{aligned}
	\mathcal{I}_{11}^{4}&\lesssim \sum_{\substack{|I_{1}|\le N\\ |I_{2}|\le N-7}}\int_{0}^{t}\langle s\rangle^{-1}\left(\big\|\widehat{{\Gamma}}^{I_{1}}\psi\big\|_{L_{x}^{2}}+\langle s\rangle^{\frac{1}{2}+\delta}\left\|\Gamma^{I_{1}}v\right\|_{L_{x}^{2}}\right)\left\|\partial\Gamma^{I_{2}}v\right\|_{L_{x}^{\infty}}\d s\\
	&\lesssim C^{2}\varepsilon^{2}\int_{0}^{t}\langle s\rangle^{-2}\langle s\rangle^{\frac{1}{2}+2\delta}\d s\lesssim C^{2}\varepsilon^{2}\int_{0}^{t}\langle s\rangle^{-\frac{3}{2}+2\delta} \lesssim C^{2}\varepsilon^{2}.
	\end{aligned}
	\end{equation*}
	
	Next, using~\eqref{est:hatGa},~\eqref{est:hatparGa},~\eqref{est:Bootpsi},~\eqref{est:Bootv},~\eqref{est:outpsiv} and  $N\ge 14$,
	\begin{equation*}
	\begin{aligned}
	\mathcal{I}_{12}&\lesssim \sum_{\substack{|I_{1}|\le N-1\\ |I_{2}|\le N-3}}\int_{0}^{t}\big\|\partial\widehat{{\Gamma}}^{I_{1}}\psi\big\|_{L_{x}^{2}}\left\|\left(\partial\Gamma^{I_{2}}v\right)\bchar_{\mathcal{C}_{\rm{ext}}}\right \|_{L_{x}^{\infty}}\d s\\
	&+\sum_{\substack{|I_{1}|\le N-3\\ |I_{2}|\le N-1}}\int_{0}^{t}\big\|\big(\partial\widehat{{\Gamma}}^{I_{1}}\psi\big)\bchar_{\mathcal{C}_{\rm{ext}}}\big\|_{L_{x}^{\infty}}\left\|\partial\Gamma^{I_{2}}v\right\|_{L_{x}^{2}}\d s\lesssim C^{2}\varepsilon^{2}\int_{0}^{t}\langle s\rangle^{-\frac{3}{2}+2\delta}\d s\lesssim C^{2}\varepsilon^{2}.
	\end{aligned}
	\end{equation*}
	Gathering these estimates, we have proved~\eqref{est:LtLxQ0}.
	
	Proof of (ii). Using again~\eqref{est:GamQ0} and $1\le \bchar_{\cint}+\bchar_{\cext}$, we have 
	\begin{equation*}
\sum_{|I|\le N-4}\int_{0}^{t}\|\langle s+r\rangle{\Gamma}^{I}Q_{0}(\psi,v)\|_{L_{x}^{2}}\d s\lesssim 
\mathcal{I}_{21}+\mathcal{I}_{22},
\end{equation*}
	where
\begin{equation*}
\begin{aligned}
\mathcal{I}_{21}&=\sum_{|I_{1}|+|I_{2}|\le N-4}\int_{0}^{t}\left\|\langle s+r\rangle Q_{0}\left({{\Gamma}}^{I_{1}}\psi,\Gamma^{I_{2}}v\right)\bchar_{\mathcal{C}_{\rm{int}}}\right\|_{L_{x}^{2}}\d s,\\
\mathcal{I}_{22}&=\sum_{|I_{1}|+|I_{2}|\le N-4}\int_{0}^{t}\left\|\langle s+r\rangle Q_{0}\left({{\Gamma}}^{I_{1}}\psi,\Gamma^{I_{2}}v\right)\bchar_{\mathcal{C}_{\rm{ext}}}\right\|_{L_{x}^{2}}\d s.
\end{aligned}
\end{equation*}
Moreover, according to~\eqref{est:hatGa},~\eqref{est:hatparGa},~\eqref{est:Q0inside} and $N\ge 14$, we have
\begin{equation*}
\mathcal{I}_{21}\lesssim \mathcal{I}_{21}^{1}+\mathcal{I}_{21}^{2}+\mathcal{I}_{21}^{3}+\mathcal{I}_{21}^{4},
\end{equation*}
where
\begin{align*}
\mathcal{I}_{21}^{1}&=\sum_{|I_{1}|\le N-3}\sum_{ |I_{2}|\le N-6}\int_{0}^{t}\big\|\big|\widehat{{\Gamma}}^{I_{1}}\psi\big|\left| \Gamma^{I_{2}}v\right|\bchar_{\mathcal{C}_{\rm{int}}}\big\|_{L_{x}^{2}}\d s,\\
\mathcal{I}_{21}^{2}&=\sum_{|I_{1}|\le N-8}\sum_{|I_{2}|\le N-3}\int_{0}^{t}\big\|\big|\widehat{{\Gamma}}^{I_{1}}\psi\big|\left| \Gamma^{I_{2}}v\right|\bchar_{\mathcal{C}_{\rm{int}}}\big\|_{L_{x}^{2}}\d s,\\
\mathcal{I}_{21}^{3}&=\sum_{|I_{1}|\le N-8}\sum_{|I_{2}|\le N-4}\int_{0}^{t}\big\|\langle s-r\rangle\big|\partial \widehat{\Gamma}^{I_{1}}\psi\big|\left| \partial\Gamma^{I_{2}}v\right|\bchar_{\mathcal{C}_{\rm{int}}}\big\|_{L_{x}^{2}}\d s,\\
\mathcal{I}_{21}^{4}&=\sum_{|I_{1}|\le N-4}\sum_{|I_{2}|\le N-9}\int_{0}^{t}\big\|\langle s-r\rangle\big|\partial \widehat{\Gamma}^{I_{1}}\psi\big|\left| \partial\Gamma^{I_{2}}v\right|\bchar_{\mathcal{C}_{\rm{int}}}\big\|_{L_{x}^{2}}\d s.
\end{align*}
First, from~\eqref{est:Bootpsi} and~\eqref{est:Bootv}, we check 
\begin{equation*}
\mathcal{I}^{1}_{21}\lesssim \sum_{\substack{|I_{1}|\le N-3\\ |I_{2}|\le N-6}}\int_{0}^{t}\big\|\widehat{{\Gamma}}^{I_{1}}\psi\big\|_{L_{x}^{2}}\left\| \Gamma^{I_{2}}v\right\|_{L_{x}^{\infty}}\d s\lesssim C^{2}\varepsilon^{2}\int_{0}^{t}\langle s\rangle^{-1}\d s \lesssim C^{2}\varepsilon^{2} \log (2+t).
\end{equation*}
Then, applying~\eqref{est:Bootpsi},~\eqref{est:Bootv},~\eqref{est:Ghostvweight},~\eqref{est:intpsilow} and~\eqref{est:intvhigh}, we can assert that
\begin{equation*}
\begin{aligned}
\mathcal{I}_{21}^{2}
&\lesssim \sum_{\substack{|I_{1}|\le N-8\\ |I_{2}|\le N-2}}\int_{0}^{t}\langle s\rangle^{-1}\big\|\langle s-r\rangle\widehat{{\Gamma}}^{I_{1}}\psi\|_{L_{x}^{\infty}}\|\partial \Gamma^{I_{2}}v\|_{L_{x}^{2}}\d s \\
&+\sum_{\substack{|I_{1}|\le N-8\\ |I_{2}|\le N-3}}\int_{0}^{t}\langle s\rangle^{-\frac{1}{2}}\big\|\widehat{{\Gamma}}^{I_{1}}\psi\big\|_{L_{x}^{\infty}}\big\|\widehat{{\Gamma}}^{I_{2}}\psi\big\|_{L_{x}^{2}}\d s\\
&\lesssim C^{2}\varepsilon^{2}\int_{0}^{t}\langle s\rangle^{-\frac{3}{2}+2\delta}\d s +C^{2}\varepsilon^{2}\int_{0}^{t}\langle s\rangle^{-1}\d s\lesssim C^{2}\varepsilon^{2}\log (2+t),
\end{aligned}
\end{equation*}
\begin{equation*}
\begin{aligned}
\mathcal{I}_{21}^{3}
&\lesssim \sum_{\substack{|I_{1}|\le N-8\\ |I_{2}|\le N-2}}\int_{0}^{t}\langle s\rangle^{-1}\big\|\langle s-r\rangle^{2}\partial\widehat{{\Gamma}}^{I_{1}}\psi\|_{L_{x}^{\infty}}\|\partial \Gamma^{I_{2}}v\|_{L_{x}^{2}}\d s \\
&+\sum_{\substack{|I_{1}|\le N-8\\ |I_{2}|\le N-3}}\int_{0}^{t}\langle s\rangle^{-\frac{3}{2}+\delta}\big\|\langle s-r\rangle\partial\widehat{{\Gamma}}^{I_{1}}\psi\big\|_{L_{x}^{\infty}}\big\|\widehat{{\Gamma}}^{I_{2}}\psi\big\|_{L_{x}^{2}}\d s\\
&+\sum_{\substack{|I_{1}|\le N-8\\ |I_{2}|\le N-3}}\int_{0}^{t}\langle s\rangle^{-\frac{1}{2}}\big\|\langle s-r\rangle^{2}\partial\widehat{{\Gamma}}^{I_{1}}\psi\big\|_{L_{x}^{\infty}}\left\|\frac{[\widehat{{\Gamma}}^{I_{2}}\psi]_{-}}{\langle s-r\rangle}\right\|_{L_{x}^{2}}\d s\lesssim C^{2}\varepsilon^{2}.
\end{aligned}
\end{equation*}
Next, combining~\eqref{est:Bootpsi},~\eqref{est:Bootv},~\eqref{est:intpsiN-2} and~\eqref{est:intvlow}, we get
\begin{equation*}
\begin{aligned}
\mathcal{I}_{21}^{4}&\lesssim \sum_{\substack{|I_{1}|\le N-1\\ |I_{2}|\le N-9}}\int_{0}^{t}\left(\big\|\widehat{{\Gamma}}^{I_{1}}\psi\|_{L_{x}^{2}}\|\partial\Gamma^{I_{2}}v\|_{L_{x}^{\infty}}+\langle s\rangle^{\frac{1}{2}+\delta}\|\Gamma^{I_{1}}v\|_{L_{x}^{2}}\left\|\frac{\partial\Gamma^{I_{2}}v}{\langle s-r\rangle}\right\|_{L_{x}^{\infty}}\right)\d s \\
&\lesssim C^{2}\varepsilon^{2}\int_{0}^{t}\left(\langle s\rangle^{-1}+\langle s\rangle^{-\frac{3}{2}+2\delta}\right)\d s\lesssim C^{2}\varepsilon^{2}\log (2+t).
\end{aligned}
\end{equation*}

Last, using again~\eqref{est:hatGa},~\eqref{est:Bootpsi},~\eqref{est:Bootv},~\eqref{est:outpsiv} and $N\ge 14$,
\begin{equation*}
\begin{aligned}
\mathcal{I}_{22}&\lesssim \sum_{\substack{|I_{1}|\le N-4\\ |I_{2}|\le N-4}}\int_{0}^{t}\big\|\langle r-s\rangle\chi(r-2s)\big|\partial\widehat{{\Gamma}}^{I_{1}}\psi\big|\big\|_{L_{x}^{2}}\left\|\left(\partial\Gamma^{I_{2}}v\right)\bchar_{\mathcal{C}_{\rm{ext}}}\right \|_{L_{x}^{\infty}}\d s\\
&\lesssim C^{2}\varepsilon^{2}\int_{0}^{t}\langle s\rangle^{-\frac{3}{2}+2\delta}\d s\lesssim C^{2}\varepsilon^{2}.
\end{aligned}
\end{equation*}
Gathering these estimates, we have proved~\eqref{est:weightLtLxQ0}.
	\end{proof}

Second, we present the $L_{t}^{1}L_{x}^{2}$ and weighted $L_{t}^{1}L_{x}^{2}$ estimates for $\Gamma^{I}\mathcal{N}_{1}$ and $\Gamma^{I}\mathcal{N}_{2}$ (see \eqref{equ:Ns00} for the expressions of $\mathcal{N}_1$ and $\mathcal{N}_2$).
\begin{lemma}\label{le:estN1N2}
		For all $t\in [0,T^{*}(\psi_{0},\vec{v}_{0}))$, the following estimates hold.
	\begin{enumerate}
		
		\item {\rm{$L_{t}^{1}L_{x}^{2}$ estimates on $\Gamma^{I}\mathcal{N}_{1}$ and $\Gamma^{I}\mathcal{N}_{2}$}}. It holds
		\begin{equation}\label{est:LtLxN1N2}
		\sum_{|I|\le N-1}\int_{0}^{t}\left(\|{\Gamma}^{I}\mathcal{N}_{1}(\psi,v)\|_{L_{x}^{2}}+\|{\Gamma}^{I}\mathcal{N}_{2}(\psi,\psi^{*})\|_{L_{x}^{2}}\right)\d s \lesssim C^{3}\varepsilon^{3}.
		\end{equation}
		
		\item {\rm{Weight $L_{t}^{1}L_{x}^{2}$ estimate on $\Gamma^{I}\mathcal{N}_{1}$}}. It holds
		\begin{equation}\label{est:weightLtLxN1I}
		\sum_{|I|\le N-4}\int_{0}^{t}\left\|\langle s+r\rangle {\Gamma}^{I}\mathcal{N}_{1}(\psi,v)\right\|_{L_{x}^{2}}\d s\lesssim C^{3}\varepsilon^{3}\log (2+t).
		\end{equation}
		
		\item {\rm{Weight $L_{t}^{1}L_{x}^{2}$ estimate on $\Gamma^{I}\mathcal{N}_{2}$}}. It holds
		\begin{equation}\label{est:weightLtLxN2I}
		\sum_{|I|\le N-4}\int_{0}^{t}\left\|\langle s+r\rangle {\Gamma}^{I}\mathcal{N}_{2}(\psi,\psi^{*})\right\|_{L_{x}^{2}}\d s\lesssim C^{2}\varepsilon^{2} \langle t\rangle^{\delta}.
		\end{equation}
		
	\end{enumerate}
\end{lemma}
\begin{proof}
	Proof of (i). From~\eqref{est:hatGa},~\eqref{est:hatGG},~\eqref{equ:HiddenPhi} and $N\ge 14$, we infer that 
		\begin{align*}
	\sum_{|I|\le N-1}\int_{0}^{t}\|{\Gamma}^{I}\mathcal{N}_{1}(\psi,v)\|_{L_{x}^{2}}\d s &\lesssim \mathcal{I}_{31}+\mathcal{I}_{32},\\
	\sum_{|I|\le N-1}\int_{0}^{t}\|{\Gamma}^{I}\mathcal{N}_{2}(\psi,\psi^{*})\|_{L_{x}^{2}}\d s &\lesssim \mathcal{I}_{33}+\mathcal{I}_{34},
	\end{align*}
	where
	\begin{equation*}
	\begin{aligned}
	\mathcal{I}_{31}&=\sum_{{|I_{1}|\le N-6}}\sum_{{|I_{2}|\le N-6}}\sum_{{|I_{3}|\le N}}\int_{0}^{t}\big\|(\Gamma^{I_{1}}v)(\Gamma^{I_{2}}v)(\widehat{{\Gamma}}^{I_{3}}\psi)\big\|_{L_{x}^{2}}\d s,\\
	\mathcal{I}_{32}&=\sum_{{|I_{1}|\le N-6}}\sum_{{|I_{2}|\le N}}\sum_{|I_{3}|\le N-6}\int_{0}^{t}\big\|(\Gamma^{I_{1}}v)(\Gamma^{I_{2}}v)(\widehat{{\Gamma}}^{I_{3}}\psi)\big\|_{L_{x}^{2}}\d s,\\
	\mathcal{I}_{33}&=\sum_{{|I_{1}|\le N-4}}\sum_{{|I_{2}|\le N-1}}\sum_{|I_{3}|\le N-7}\int_{0}^{t}\big\|\big(\widehat{\Gamma}^{I_{1}}\psi\big)\big(\widehat{\Gamma}^{I_{2}}\psi\big)\big(\big[\widehat{{\Gamma}}^{I_{3}}\psi\big]_{-}\big)\big\|_{L_{x}^{2}}\d s,\\
	\mathcal{I}_{34}&=\sum_{{|I_{1}|\le N-4}}\sum_{{|I_{2}|\le N-7}}\sum_{{|I_{3}|\le N-1}}\int_{0}^{t}\big\|\big(\widehat{\Gamma}^{I_{1}}\psi\big)\big(\widehat{\Gamma}^{I_{2}}\psi\big)\big(\big[\widehat{{\Gamma}}^{I_{3}}\psi\big]_{-}\big)\big\|_{L_{x}^{2}}\d s.
	\end{aligned}
	\end{equation*}
	First, from~\eqref{est:Bootpsi},~\eqref{est:Bootv} and~\eqref{est:pointGlobal}, we see that 
	\begin{equation*}
	\begin{aligned}
	\mathcal{I}_{31}&\lesssim\sum_{\substack{|I_{1}|\le N-6 \\ |I_{2}|\le N-6}}\sum_{{|I_{3}|\le N}}\int_{0}^{t}\big\|\Gamma^{I_{1}}v\|_{L_{x}^{\infty}}\left\|\Gamma^{I_{2}}v\right\|_{L_{x}^{\infty}}\big\|\widehat{{\Gamma}}^{I_{3}}\psi\big\|_{L_{x}^{2}}\d s\\
	&\lesssim C^{3}\varepsilon^{3}\int_{0}^{t}\langle s\rangle^{-1}\langle s\rangle^{-1}\langle s\rangle^{\delta}\d s \lesssim C^{3}\varepsilon^{3}\int_{0}^{t}\langle s\rangle^{-2+\delta}\d s \lesssim C^{3}\varepsilon^{3},
	\end{aligned}
	\end{equation*}
		\begin{equation*}
	\begin{aligned}
	\mathcal{I}_{32}&\lesssim\sum_{\substack{|I_{1}|\le N-6 \\ |I_{3}|\le N-6}}\sum_{{|I_{2}|\le N}}\int_{0}^{t}\big\|\Gamma^{I_{1}}v\|_{L_{x}^{\infty}}\left\|\Gamma^{I_{2}}v\right\|_{L_{x}^{2}}\big\|\widehat{{\Gamma}}^{I_{3}}\psi\big\|_{L_{x}^{\infty}}\d s\\
	&\lesssim C^{3}\varepsilon^{3}\int_{0}^{t}\langle s\rangle^{-1}\langle s\rangle^{2\delta}\langle s\rangle^{-\frac{1}{2}}\d s \lesssim C^{3}\varepsilon^{3}\int_{0}^{t}\langle s\rangle^{-\frac{3}{2}+2\delta}\d s \lesssim C^{3}\varepsilon^{3},
	\end{aligned}
	\end{equation*}
		\begin{equation*}
	\begin{aligned}
	\mathcal{I}_{33}&\lesssim\sum_{\substack{|I_{1}|\le N-4 \\ |I_{3}|\le N-7}}\sum_{{|I_{2}|\le N-1}}\int_{0}^{t}\big\|\widehat{\Gamma}^{I_{1}}\psi\big\|_{L_{x}^{\infty}}\big\|\widehat{\Gamma}^{I_{2}}\psi\big\|_{L_{x}^{2}}\big\|\big[\widehat{{\Gamma}}^{I_{3}}\psi\big]_{-}\big\|_{L_{x}^{\infty}}\d s\\
	&\lesssim C^{3}\varepsilon^{3}\int_{0}^{t}\langle s\rangle^{-\frac{1}{2}}\langle s\rangle^{-\frac{3}{2}+\delta}\d s \lesssim C^{3}\varepsilon^{3}\int_{0}^{t}\langle s\rangle^{-2+\delta}\d s \lesssim C^{3}\varepsilon^{3}.
	\end{aligned}
	\end{equation*}
	Then, from~\eqref{est:Bootpsi} and~\eqref{est:Ghostvweight}, we get
	\begin{equation*}
	\begin{aligned}
	\mathcal{I}_{34}
	&\lesssim \sum_{\substack{|I_{1}|\le N-4\\ |I_{2}|\le N-7}}\sum_{{|I_{3}|\le N-1}}\int_{0}^{t}\big\|\widehat{\Gamma}^{I_{1}}\psi\big\|_{L_{x}^{\infty}}\big\|\langle s-r\rangle\widehat{\Gamma}^{I_{2}}\psi\big\|_{L_{x}^{\infty}}\left\|\frac{\big[\widehat{{\Gamma}}^{I_{3}}\psi\big]_{-}}{\langle s-r\rangle}\right\|_{L_{x}^{2}}\d s\\
	&\lesssim C^{2}\varepsilon^{2}\sum_{|I_{3}|\le N-1}\int_{0}^{t}\langle s\rangle^{-1+\delta}\left\|\frac{\big[\widehat{{\Gamma}}^{I_{3}}\psi\big]_{-}}{\langle s-r\rangle}\right\|_{L_{x}^{2}}\d s\lesssim C^{3}\varepsilon^{3}.
	\end{aligned}
	\end{equation*}
	Gathering these estimates, we have proved~\eqref{est:LtLxN1N2}.
	
	Proof of (ii). Using again $N\ge 14$ and $1\le \bchar_{\cint}+\bchar_{\cext}$, we have 
	\begin{equation*}
\sum_{|I|\le N-4}\int_{0}^{t}\|\langle s+r\rangle{\Gamma}^{I}\mathcal{N}_{1}(\psi,v)\|_{L_{x}^{2}}\d s \lesssim \mathcal{I}_{41}+\mathcal{I}_{42}+\mathcal{I}_{43},
\end{equation*}
where
\begin{equation*}
\begin{aligned}
	\mathcal{I}_{41}&=\sum_{{|I_{1}|\le N-6}}\sum_{{|I_{2}|\le N-6}}\sum_{{|I_{3}|\le N-3}}\int_{0}^{t}\big\|\langle s+r\rangle(\Gamma^{I_{1}}v)(\Gamma^{I_{2}}v)(\widehat{{\Gamma}}^{I_{3}}\psi)\big\|_{L_{x}^{2}}\d s,\\
\mathcal{I}_{42}&=\sum_{{|I_{1}|\le N-8}}\sum_{{|I_{2}|\le N-3}}\sum_{|I_{3}|\le N-7}\int_{0}^{t}\big\|\langle s+r\rangle(\Gamma^{I_{1}}v)(\Gamma^{I_{2}}v)(\widehat{{\Gamma}}^{I_{3}}\psi)\bchar_{\cint}\big\|_{L_{x}^{2}}\d s,\\
\mathcal{I}_{43}&=\sum_{{|I_{1}|\le N-8}}\sum_{{|I_{2}|\le N-3}}\sum_{|I_{3}|\le N-7}\int_{0}^{t}\big\|\langle s+r\rangle(\Gamma^{I_{1}}v)(\Gamma^{I_{2}}v)(\widehat{{\Gamma}}^{I_{3}}\psi)\bchar_{\cext}\big\|_{L_{x}^{2}}\d s.
\end{aligned}
\end{equation*}

First, from~\eqref{est:Bootpsi} and~\eqref{est:Bootv}, we get
\begin{equation*}
\begin{aligned}
	\mathcal{I}_{41}&
	\lesssim \sum_{\substack{|I_{1}|\le N-6\\ |I_{2}|\le N-6}}\sum_{{|I_{3}|\le N-3}}\int_{0}^{t}\left\|\langle s+r\rangle\Gamma^{I_{1}}v\right\|_{L_{x}^{\infty}}\left\|\Gamma^{I_{2}}v\right\|_{L_{x}^{\infty}}\big\|\widehat{{\Gamma}}^{I_{3}}\psi\big\|_{L_{x}^{2}}\d s\\
	&\lesssim C^{3}\varepsilon^{3}\int_{0}^{t}\langle s\rangle^{-1}\d s \lesssim C^{3}\varepsilon^{3}\log (2+t).
	\end{aligned}
\end{equation*}
Then, from~\eqref{est:Bootpsi},~\eqref{est:Bootv} and~\eqref{est:intvlow},
\begin{equation*}
\begin{aligned}
\mathcal{I}_{42}&\lesssim \sum_{\substack{|I_{1}|\le N-8\\|I_{3}|\le N-7 }}\sum_{{|I_{2}|\le N-3}}\int_{0}^{t}\left\|\frac{\langle s+r\rangle}{\langle s-r\rangle}(\Gamma^{I_{1}}v)\bchar_{\cint}\right\|_{L_{x}^{\infty}}\left\|\Gamma^{I_{2}}v\right\|_{L_{x}^{2}}\big\|\langle s-r\rangle\widehat{{\Gamma}}^{I_{3}}\psi\big\|_{L_{x}^{\infty}}\d s\\
&\lesssim C^{3}\varepsilon^{3}\int_{0}^{t}\left(\langle s\rangle^{-1}+\langle s\rangle^{-1+\delta}\right)\langle s\rangle^{-\frac{1}{2}+\delta}\d s \lesssim C^{3}\varepsilon^{3}\int_{0}^{t}\langle s\rangle^{-\frac{3}{2}+2\delta}\d s\lesssim C^{3}\varepsilon^{3}.
\end{aligned}
\end{equation*}
Last, using again~\eqref{est:Bootpsi},~\eqref{est:Bootv} and~\eqref{est:outpsiv}, 
\begin{equation*}
\begin{aligned}
\mathcal{I}_{43}
&\lesssim \sum_{\substack{|I_{1}|\le N-8\\ |I_{3}|\le N-7}}\sum_{|I_{2}|\le N-3}\int_{0}^{t}\left\|\langle s+r\rangle\Gamma^{I_{1}}v\right\|_{L_{x}^{\infty}}\left\|\Gamma^{I_{2}}v\right\|_{L_{x}^{2}}\big\|(\widehat{{\Gamma}}^{I_{3}}\psi)\bchar_{\cext}\big\|_{L_{x}^{\infty}}\d s\\
&\lesssim C^{3}\varepsilon^{3}\int_{0}^{t} \langle s\rangle^{-\frac{3}{2}}\d s \lesssim C^{3}\varepsilon^{3}.
\end{aligned}
\end{equation*}
Gathering these estimates, we have proved~\eqref{est:weightLtLxN1I}.
	
	Proof of (iii). {\textbf{Step 1.}} Preliminary estimate for $\Gamma^{I}\mathcal{N}_{2}$. We claim that 
	\begin{equation}\label{est:prelestG2}
	\begin{aligned}
	&\sum_{|I|\le N-4}\int_{0}^{t}\left\|\langle s+r\rangle\Gamma^{I}\mathcal{N}_{2}(\psi,\psi^{*})\right\|_{L_{x}^{2}}\d s\\
	&\lesssim C^{2}\varepsilon^{2}\sum_{|I|\le N-4}\int_{0}^{t}\langle s\rangle^{-1}\left(\left\|S\Gamma^{I}\Psi\right\|_{L_{x}^{2}}+\left\|\Gamma\Gamma^{I}\Psi\right\|_{L_{x}^{2}}\right)\d s.
	\end{aligned}
	\end{equation}
	Indeed, from~\eqref{equ:HiddenPhi}, we see that 
	\begin{equation*}
	\sum_{|I|\le N-4}\int_{0}^{t}\left\|\langle s+r\rangle\Gamma^{I}\mathcal{N}_{2}(\psi,\psi^{*})\right\|_{L_{x}^{2}}\d s\lesssim \mathcal{I}_{5},
	\end{equation*}
	where
	\begin{equation*}
	\begin{aligned}
	\mathcal{I}_{5}&=\sum_{|I_{1}|\le N-4}\sum_{|I_{2}|\le N-4}\sum_{|I_{3}|\le N-4}\int_{0}^{t}\big\|\langle s+r\rangle \big(\widehat{\Gamma}^{I_{1}}\psi\big)\big(\widehat{\Gamma}^{I_{2}}\psi\big)\big(\big[\widehat{\Gamma}^{I_{3}}\psi\big]_{-}\big)\big\|_{L_{x}^{2}}\d x.
	\end{aligned}
	\end{equation*}
	Note that, from~\eqref{est:hatGa},~\eqref{est:comm},~\eqref{est:Gaparf},~\eqref{equ:Hiddenpsi-} and~\eqref{est:pointGlobal}, we have
	\begin{equation*}
	\begin{aligned}
	\mathcal{I}_{5}
	&\lesssim \sum_{\substack{|I_{1}|\le N-4\\ |I_{2}|\le N-4}}\sum_{|I_{3}|\le N-4}\int_{0}^{t}\big\| \widehat{\Gamma}^{I_{1}}\psi\big\|_{L_{x}^{\infty}}\big\|\widehat{\Gamma}^{I_{2}}\psi\big\|_{L_{x}^{\infty}}\big\|\langle s+r\rangle\big(\big[\widehat{\Gamma}^{I_{3}}\psi\big]_{-}\big)\big\|_{L_{x}^{2}}\d x\\
	&\lesssim C^{2}\varepsilon^{2}\sum_{|I|\le N-4}\int_{0}^{t}\langle s\rangle^{-1}\left(\left\|S\Gamma^{I}\Psi\right\|_{L_{x}^{2}}+\left\|\Gamma\Gamma^{I}\Psi\right\|_{L_{x}^{2}}\right)\d s,
	\end{aligned}
	\end{equation*}
	which implies~\eqref{est:prelestG2}.
	
	{\textbf{Step 2.}} Preliminary estimate for $\left\|S\Gamma^{I}\Psi\right\|_{L_{x}^{2}}$ and $\left\|\Gamma\Gamma^{I}\Psi\right\|_{L_{x}^{2}}$. We claim that 
    \begin{equation}\label{est:StildeS}
    \begin{aligned}
    &\sum_{|I|\le N-4}\left(\left\|S\Gamma^{I}\Psi\right\|_{L_{x}^{2}}+\left\|\Gamma\Gamma^{I}\Psi\right\|_{L_{x}^{2}}\right)\\
    &\lesssim C^{2}\varepsilon^{2}+\sum_{|I|\le N-4}\left(\left\|S\Gamma^{I}\widetilde{\Psi}\right\|_{L_{x}^{2}}+\left\|\Gamma\Gamma^{I}\widetilde{\Psi}\right\|_{L_{x}^{2}}\right).
    \end{aligned}
    \end{equation}
    Indeed, from $\|S\|\le \|\langle t+r\rangle\partial\|$ and the definition of $\Psi$ and $\widetilde{\Psi}$ in \S\ref{Se:Hidden}, we see that 
    \begin{equation*}
    \begin{aligned}
    &\sum_{|I|\le N-4}\left(\left\|S\Gamma^{I}\Psi\right\|_{L_{x}^{2}}+\left\|\Gamma\Gamma^{I}\Psi\right\|_{L_{x}^{2}}\right)\\
    &\lesssim \sum_{|I|\le N-4}\left(\left\|\langle t+r\rangle\partial\Gamma^{I}(v\psi)\right\|_{L_{x}^{2}}+\|\Gamma\Gamma^{I}(v\psi)\|_{L_{x}^{2}}+\big\|S\Gamma^{I}\widetilde{\Psi}\big\|_{L_{x}^{2}}+\big\|\Gamma\Gamma^{I}\widetilde{\Psi}\big\|_{L_{x}^{2}}\right).
    \end{aligned}
    \end{equation*}
    Moreover, using~\eqref{est:hatGa}, $1\le \bchar_{\cint}+\bchar_{\cext}$ and $N\ge 14$, we have 
    \begin{equation*}
    \sum_{|I|\le N-4}\left(\left\|\langle t+r\rangle\partial\Gamma^{I}(v\psi)\right\|_{L_{x}^{2}}+\|\Gamma\Gamma^{I}(v\psi)\|_{L_{x}^{2}}\right)\lesssim I_{61}+\mathcal{I}_{62}+\mathcal{I}_{63}+\mathcal{I}_{64},
    \end{equation*}
    where
    \begin{equation*}
    \begin{aligned}
    \mathcal{I}_{61}&=\sum_{|I_{1}|\le N-3}\sum_{|I_{2}|\le N-4}\big\|\big(\Gamma^{I_{1}}v\big)\big(\widehat{{\Gamma}}^{I_{2}}\psi\big)\big\|_{L_{x}^{2}},\\
    \mathcal{I}_{62}&=\sum_{|I_{1}|\le N-6}\sum_{|I_{2}|\le N-3}\big\|\langle t+r\rangle\big(\Gamma^{I_{1}}v\big)\big(\widehat{{\Gamma}}^{I_{2}}\psi\big)\big\|_{L_{x}^{2}},\\
    \mathcal{I}_{63}&=\sum_{|I_{1}|\le N-3}\sum_{|I_{2}|\le N-7}\big\|\langle t+r\rangle\big(\Gamma^{I_{1}}v\big)\big(\widehat{{\Gamma}}^{I_{2}}\psi\big)\bchar_{\cint}\big\|_{L_{x}^{2}},\\
    \mathcal{I}_{64}&=\sum_{|I_{1}|\le N-3}\sum_{|I_{2}|\le N-7}\big\|\langle t+r\rangle\big(\Gamma^{I_{1}}v\big)\big(\widehat{{\Gamma}}^{I_{2}}\psi\big)\bchar_{\cext}\big\|_{L_{x}^{2}}.
    \end{aligned}
    \end{equation*}
    Note that, from~\eqref{est:Bootpsi},~\eqref{est:Bootv} and~\eqref{est:pointGlobal}, we have
    \begin{equation*}
    \begin{aligned}
    \mathcal{I}_{61}&\lesssim \sum_{|I_{1}|\le N-3}\sum_{|I_{2}|\le N-4}\big\|\Gamma^{I_{1}}v\big\|_{L_{x}^{2}}\big\|\widehat{{\Gamma}}^{I_{2}}\psi\big\|_{L_{x}^{\infty}}\lesssim C^{2}\varepsilon^{2},\\
    \mathcal{I}_{62}&\lesssim \sum_{|I_{1}|\le N-6}\sum_{|I_{2}|\le N-3}\big\|\langle t+r\rangle\big(\Gamma^{I_{1}}v\big)\|_{L_{x}^{\infty}}\big\|\big(\widehat{{\Gamma}}^{I_{2}}\psi\big)\big\|_{L_{x}^{2}}\lesssim C^{2}\varepsilon^{2}.
    \end{aligned}
    \end{equation*}
    Then, using~\eqref{est:Bootpsi},~\eqref{est:Bootv},~\eqref{est:intvhigh} and~\eqref{est:outpsiv},
    \begin{equation*}
    \begin{aligned}
    \mathcal{I}_{63}
    &\lesssim \sum_{\substack{|I_{1}|\le N\\ |I_{2}|\le N-7}}\|\partial \Gamma^{I_{1}}v\|_{L_{x}^{2}}\big\|\langle t-r\rangle \big(\widehat{{\Gamma}}^{I_{2}}\psi\big)\big\|_{L_{x}^{\infty}}\\
    &+\sum_{\substack{|I_{1}|\le N-1\\ |I_{2}|\le N-7}}\langle t\rangle^{\frac{1}{2}}\left\|\Gamma^{I_{1}}\psi\right\|_{L_{x}^{2}}\big\|\widehat{{\Gamma}}^{I_{2}}\psi\big\|_{L_{x}^{\infty}}\lesssim C^{2}\varepsilon^{2},
    \end{aligned}
    \end{equation*}
    \begin{equation*}
    \mathcal{I}_{64}\lesssim \sum_{\substack{|I_{1}|\le N-3\\ |I_{2}|\le N-7}}\left\|\langle r-t\rangle\chi(r-2t)\Gamma^{I_{1}}v\right\|_{L_{x}^{2}}\big\|\widehat{{\Gamma}}^{I_{2}}\psi\big\|_{L_{x}^{\infty}}\lesssim C^{2}\varepsilon^{2}.
    \end{equation*}
    We see that~\eqref{est:StildeS} follows from above estimates. 
    
    Note that, combining~\eqref{est:prelestG2} and~\eqref{est:StildeS}, we have 
    \begin{equation}\label{est:prelestG22}
    \begin{aligned}
    &\sum_{|I|\le N-4}\int_{0}^{t}\left\|\langle s+r\rangle\Gamma^{I}\mathcal{N}_{2}(\psi,\psi^{*})\right\|_{L_{x}^{2}}\d s\\
    &\lesssim C^{2}\varepsilon^{2}\bigg(\log (2+t)+\sum_{|I|\le N-4}\int_{0}^{t}\langle s\rangle^{-1}\left(\big\|S\Gamma^{I}\widetilde{\Psi}\big\|_{L_{x}^{2}}+\big\|\Gamma\Gamma^{I}\widetilde{\Psi}\big\|_{L_{x}^{2}}\right)\d s\bigg).
    \end{aligned}
    \end{equation}

	To control $\big\|S\Gamma^{I}\widetilde{\Psi}\big\|_{L_{x}^{2}}$ and $\big\|\Gamma\Gamma^{I}\widetilde{\Psi}\big\|_{L_{x}^{2}}$ that appear in~\eqref{est:prelestG22}, we introduce the following two quantities: for $n=1,2,\cdots,N-3$, we set 
		\begin{equation}\label{def:WH}
	\begin{aligned}
	\mathcal{H}_{n}\big(t,\widetilde{\Psi}\big)&=\sum_{|I|=n-1}\left(\mathcal{F}\big(t,\Gamma^{I}\widetilde{\Psi}\big)+\mathcal{E}_{0}\big(t,\Gamma^{I}\widetilde{\Psi}\big)\right),\\
	\mathcal{W}\big(t,\widetilde{\Psi}\big)&=\sum_{n=0}^{N-3}2^{-n}\mathcal{H}_{n}\big(t,\widetilde{\Psi}\big),\quad \mbox{where}\ \ \mathcal{H}_{0}\big(t,\widetilde{\Psi}\big)=\big\|\widetilde{\Psi}(t,x)\big\|_{L_{x}^{2}}^{2}.
	\end{aligned}
	\end{equation}
	We start with the estimate of $\mathcal{H}_{0}\big(t,\widetilde{\Psi}\big)^{\frac{1}{2}}$  (i.e. $\big\|\widetilde{\Psi}\big\|_{L_{x}^{2}}$), and then use Gronwall's inequality~\eqref{est:Gron} to obtain the estimates for $\mathcal{W}\big(t,\widetilde{\Psi}\big)^{\frac{1}{2}}$.

	{\textbf{Step 3.}} Control of $\|\widetilde{\Psi}\|_{L_{x}^{2}}$. We claim that 
	\begin{equation}\label{est:PsiL2}
	\|\widetilde{\Psi}\|_{L_{x}^{2}}\lesssim \varepsilon\log^{\frac{3}{2}}(2+t),\quad \mbox{for all}\ t\in [0,T^{*}(\psi_{0},\vec{v}_{0})).
	\end{equation}
	Indeed, from~\eqref{equ:defPsi} and the smallness condition~\eqref{est:smallness}, we see that 
	\begin{equation*}
	\big\|\widetilde{\Psi}(0,x)\big\|_{L_{x}^{2}}+\big\|\pt\widetilde{\Psi}(0,x)\big\|_{L_{x}^{1}}+\big\|\pt\widetilde{\Psi}(0,x)\big\|_{L_{x}^{2}}\lesssim \varepsilon.
	\end{equation*}
	Then, from~\eqref{est:Bootpsi} and~\eqref{est:Bootv}, we get
	\begin{equation*}
	\begin{aligned}
	&\int_{0}^{t}\left(\|\mathcal{N}_{1}\|_{L_{x}^{1}}+\|\mathcal{N}_{1}\|_{L_{x}^{2}} +\|\mathcal{N}_{2}\|_{L_{x}^{1}}+\|\mathcal{N}_{2}\|_{L_{x}^{2}}+\|Q_{0}(\psi,v)\|_{L_{x}^{2}}\right)\d s\\
	&\lesssim \int_{0}^{t}\left(\|v\|_{L_{x}^{\infty}}+\|\partial v\|_{L_{x}^{\infty}}\right)\left(\|v\|_{L_{x}^{2}}+\| v\|_{L_{x}^{\infty}}+1\right)\left(\|\psi\|_{L_{x}^{2}}+\|\partial \psi\|_{L_{x}^{2}}\right)\d s\\
	&+\int_{0}^{t}\left\|[\psi]_{-}\right\|_{L_{x}^{\infty}}\left(\|\psi\|_{L_{x}^{2}}+\|\psi\|_{L_{x}^{\infty}}\right)\|\psi\|_{L_{x}^{2}}\d s\lesssim C^{2}\varepsilon^{2}\int_{0}^{t}\langle s\rangle^{-1}\d s \lesssim C^{2}\varepsilon^{2}\log (2+t).
	\end{aligned}
	\end{equation*}
	Then, from~\eqref{est:hatGa},~\eqref{est:Q0inside},~\eqref{est:Bootpsi},~\eqref{est:Bootv},~\eqref{est:intpsilow},~\eqref{est:intvlow},~\eqref{est:outpsiv} and $\d x=r\d r$, we have 
	\begin{equation*}
	\begin{aligned}
	&\int_{0}^{t}\left\|Q_{0}(\psi,v)\bchar_{\cint}\right\|_{L_{x}^{1}}\d s \\
	&\lesssim 
	C^{2}\varepsilon^{2}\int_{0}^{t}\langle s\rangle^{-1}\left\|\frac{\langle s-r\rangle^{-1}}{\langle s+r\rangle^{\frac{1}{2}-\delta}}\left(\frac{\langle s-r\rangle}{\langle s+r\rangle^{2}}+\langle s+r\rangle^{-2+\delta}\right)\right\|_{L_{x}^{1}}\d s\\
	&+\int_{0}^{t}\langle s\rangle^{-1}\big(\big\|\widehat{\Gamma}\psi\big\|_{L_{x}^{2}}+\left\|\psi\right\|_{L_{x}^{2}}\big)\|{\Gamma}v\|_{L_{x}^{2}}\d s\lesssim C^{2}\varepsilon^{2}\log (2+t),
	\end{aligned}
	\end{equation*}
	\begin{equation*}
	\begin{aligned}
	\int_{0}^{t}\left\|Q_{0}(\psi,v)\bchar_{\cext}\right\|_{L_{x}^{1}}\d s
	&\lesssim \int_{0}^{t}\langle s\rangle^{-1}\left\|\langle r-s\rangle\chi(r-2s)(\partial v)\right\|_{L_{x}^{2}}\left\|\partial \psi\right\|_{L_{x}^{2}}\d s\\
	&\lesssim C^{2}\varepsilon^{2}\int_{0}^{t}\langle s\rangle^{-1}\d s \lesssim C^{2}\varepsilon^{2}\log (2+t).
	\end{aligned}
	\end{equation*}
	It follows immediately that
	\begin{equation*}
	\begin{aligned}
	\int_{0}^{t}\left\|Q_{0}(\psi,v)\right\|_{L_{x}^{1}}\d s
	&\lesssim \int_{0}^{t}\left\|Q_{0}(\psi,v)\bchar_{\cint}\right\|_{L_{x}^{1}}\d s\\
	&+\int_{0}^{t}\left\|Q_{0}(\psi,v)\bchar_{\cext}\right\|_{L_{x}^{1}}\d s
	\lesssim C^{2}\varepsilon^{2}\log (2+t).
	\end{aligned}
	\end{equation*}
	Combining the above inequalities with~\eqref{est:Lx2}, for $\varepsilon$ small enough, we get
	\begin{equation*}
	\begin{aligned}
\big\|\widetilde{\Psi}(t,x)\big\|_{L_{x}^{2}}&\lesssim \big\|\widetilde{\Psi}(0,x)\big\|_{L_{x}^{2}}+\log^{\frac{1}{2}}(2+t)\left(\big\|\pt\widetilde{\Psi}(0,x)\big\|_{L_{x}^{1}}+\big\|\pt\widetilde{\Psi}(0,x)\big\|_{L_{x}^{2}}\right)\\
&+\log^{\frac{1}{2}} (2+t)\int_{0}^{t}\left(\|\mathcal{N}_{1}\|_{L_{x}^{1}}+\|\mathcal{N}_{1}\|_{L_{x}^{2}}+\|\mathcal{N}_{2}\|_{L_{x}^{1}}+\|\mathcal{N}_{2}\|_{L_{x}^{2}}\right)\d s\\
&+\log^{\frac{1}{2}}(2+t)\int_{0}^{t}\left(\|Q_{0}(\psi,v)\|_{L_{x}^{1}}+\|Q_{0}(\psi,v)\|_{L_{x}^{2}}\right)\d s\lesssim 
\varepsilon\log^{\frac{3}{2}} (2+t),
\end{aligned}
	\end{equation*}
	which means~\eqref{est:PsiL2}.
	
	{\textbf{Step 4.}} Control of $\mathcal{W}(t,\widetilde{\Psi})^{\frac{1}{2}}$. 
	On the one hand, we notice that, for all $n=1,\cdots,N-3$, 
	\begin{equation*}
	\begin{aligned}
	\mathcal{H}_{n}(t,\widetilde{\Psi})
	&=\sum_{|I|=n-1}\|S\Gamma^{I}\widetilde{\Psi}+\Gamma^{I}\widetilde{\Psi}\|^{2}_{L_{x}^{2}}+\sum_{|I|=n-1}\|\partial \Gamma^{I}\widetilde{\Psi}\|^{2}_{L_{x}^{2}}\\
	&+\sum_{|I|=n-1}\left(\big\|\Omega \Gamma^{I}\widetilde{\Psi}\big\|^{2}_{L_{x}^{2}}+\big\|L_{1}\Gamma^{I}\widetilde{\Psi}\big\|^{2}_{L_{x}^{2}}+\big\|L_{2}\Gamma^{I}\widetilde{\Psi}\big\|^{2}_{L_{x}^{2}}\right).
	\end{aligned}
	\end{equation*}
	Therefore, from the basic inequality $(a+b)^{2}\ge \frac{1}{2}a^{2}-b^{2}$, we obtain,
	\begin{equation*}
	\mathcal{H}_{n}(t,\widetilde{\Psi})\ge \frac{1}{2}\sum_{|I|= n-1}\|S\Gamma^{I}\widetilde{\Psi}\|^{2}_{L_{x}^{2}}+\sum_{|I|=n}\big\|\Gamma^{I}\widetilde{\Psi}\big\|^{2}_{L_{x}^{2}}-\sum_{|I|=n-1}\big\|\Gamma^{I}\widetilde{\Psi}\big\|^{2}_{L_{x}^{2}},
	\end{equation*}
	for $n=1,\cdots,N-3$. Summing over $n$ and using the fact that $\mathcal{H}_{0}(t,\widetilde{\Psi})=\big\|\widetilde{\Psi}\big\|_{L_{x}^{2}}^{2}$, we find
	\begin{equation*}
	\mathcal{W}(t,\widetilde{\Psi})\ge 2^{-N}\bigg(\sum_{|I|\le N-4}\|S\Gamma^{I}\widetilde{\Psi}\|^{2}_{L_{x}^{2}}+\sum_{|I|\le N-3}\|\Gamma^{I}\widetilde{\Psi}\|^{2}_{L_{x}^{2}}\bigg),
	\end{equation*}
	which implies 
	\begin{equation}\label{est:Wlowerbound}
	\sum_{|I|\le N-4}\left(\|S\Gamma^{I}\widetilde{\Psi}\|_{L_{x}^{2}}+\|\Gamma\Gamma^{I}\widetilde{\Psi}\|_{L_{x}^{2}}\right)\lesssim \mathcal{W}(t,\widetilde{\Psi})^{\frac{1}{2}}.
	\end{equation}
	
	On the other hand, from~\eqref{est:smallness},~\eqref{est:Energywave},~\eqref{est:Con},~\eqref{equ:HiddenPsi},~\eqref{est:weightLtLxQ0},~\eqref{est:weightLtLxN1I},~\eqref{est:prelestG22} and~\eqref{est:Wlowerbound}, for $\varepsilon$ small enough (depending on $C$), we see that 
	\begin{equation*}
	\begin{aligned}
	&\sum_{|I|\le N-4}\mathcal{F}(t,\Gamma^{I}\widetilde{\Psi})^{\frac{1}{2}}+\sum_{|I|\le N-4}\mathcal{E}_{0}(t,\Gamma^{I}\widetilde{\Psi})^{\frac{1}{2}}\\
	&\lesssim \sum_{|I|\le N-4}\left(\mathcal{F}(0,\Gamma^{I}\widetilde{\Psi})^{\frac{1}{2}}+\mathcal{E}_{0}(t,\Gamma^{I}\widetilde{\Psi})^{\frac{1}{2}}\right)+\sum_{|I|\le N-4}\int_{0}^{t}\left\|\langle s+r\rangle \Gamma^{I}Q_{0}(\psi,v)\right\|_{L_{x}^{2}}\d s \\
	&+\sum_{|I|\le N-4}\int_{0}^{t}\left(\left\|\langle s+r\rangle \Gamma^{I}\mathcal{N}_{1}(\psi,v)\right\|_{L_{x}^{2}}+\left\|\langle s+r\rangle \Gamma^{I}\mathcal{N}_{2}(\psi,\psi^{*})\right\|_{L_{x}^{2}}\right)\d s\\
	&\lesssim \varepsilon+C^{2}\varepsilon^{2}\log (2+t)+C^{2}\varepsilon^{2}\int_{0}^{t}\langle s\rangle^{-1}\mathcal{W}(s,\widetilde{\Psi})^{\frac{1}{2}}\d s.
	\end{aligned}
	\end{equation*}
	 Combining the above inequality with~\eqref{est:PsiL2}, we have 
	\begin{equation*}
	\begin{aligned}
	\mathcal{W}(t,\widetilde{\Psi})^{\frac{1}{2}}
	&\lesssim \sum_{|I|\le N-4}\mathcal{F}(t,\Gamma^{I}\widetilde{\Psi})^{\frac{1}{2}}+\sum_{|I|\le N-4}\mathcal{E}_{0}(t,\Gamma^{I}\widetilde{\Psi})^{\frac{1}{2}}+\big\|\widetilde{\Psi}\|_{L_{x}^{2}}\\
	&\lesssim  \varepsilon\log^{\frac{3}{2}} (2+t)+C^{2}\varepsilon^{2}\int_{0}^{t}\langle s\rangle^{-1}\mathcal{W}(s,\widetilde{\Psi})^{\frac{1}{2}}\d s,
	\end{aligned}
	\end{equation*}
	which implies
	\begin{equation*}
	\mathcal{W}(t,\widetilde{\Psi})^{\frac{1}{2}}\le \varepsilon^{\frac{1}{2}}\log^{\frac{3}{2}}(2+t)+\varepsilon^{\frac{1}{2}}\int_{0}^{t}\langle s\rangle^{-1}\mathcal{W}(s,\widetilde{\Psi})^{\frac{1}{2}}\d s,
	\end{equation*}
	for $\varepsilon$ small enough (depending on $C$).
	Based on above estimates and the Gronwall's inequality~\eqref{est:Gron}, we obtain
	\begin{equation}\label{est:W}
	\begin{aligned}
	\mathcal{W}(t,\widetilde{\Psi})^{\frac{1}{2}}
	&\le \varepsilon^{\frac{1}{2}}\log^{\frac{3}{2}} (2+t)\left(1+\int_{0}^{t}\left(\varepsilon^{\frac{1}{2}}\langle s\rangle^{-1}e^{\varepsilon^{\frac{1}{2}}\int_{s}^{t}\langle \tau\rangle^{-1}\d \tau}\right)\d s\right)\\
	&\le \varepsilon^{\frac{1}{2}}\log^{\frac{3}{2}} (2+t)e^{\varepsilon^{\frac{1}{2}}\int_{0}^{t}\langle s\rangle^{-1}\d s}\lesssim \varepsilon^{\frac{1}{2}}\langle t\rangle^{\delta},\quad \mbox{for}\ 0<\varepsilon^{\frac{1}{2}}<\frac{\delta}{2}.
	\end{aligned}
	\end{equation}
	
	{\textbf{Step 5.}} Conclusion. Combining~\eqref{est:prelestG22},~\eqref{est:Wlowerbound} and~\eqref{est:W}, we conclude that
	\begin{equation*}
	\begin{aligned}
	&\sum_{|I|\le N-4}\int_{0}^{t}\left\|\langle s+r\rangle\Gamma^{I}\mathcal{N}_{2}(\psi,\psi^{*})\right\|_{L_{x}^{2}}\d s\\
	&\lesssim C^{2}\varepsilon^{2}\log (2+t)+C^{2}\varepsilon^{2}\int_{0}^{t}\langle s\rangle^{-1}\mathcal{W}(s,\Psi)^{\frac{1}{2}}\d s  \lesssim C^{2}\varepsilon^{2}\langle t\rangle^{\delta}.
	\end{aligned}
	\end{equation*}
	The proof of~\eqref{est:weightLtLxN2I} is complete.
	\end{proof}

As a consequence of the Lemmas~\ref{le:GlobalSobolev},~\ref{le:l2conform},~\ref{le:estQ0} and~\ref{le:estN1N2}, we also have the following $L_{x}^{2}$ and pointwise estimates for $S\Gamma^{I}\Psi$ and $\Gamma\Gamma^{I}\Psi$.
\begin{corollary}
	For any $t\in [0,T^{*}(\psi_{0},\vec{v}_{0}))$, the following estimates are true.
	\begin{enumerate}
		\item {\rm{$L_{x}^{2}$ estimates for $S\Gamma^{I}\Psi$ and $\Gamma\Gamma^{I}\Psi$}}. It holds
		\begin{equation}\label{est:L2SG}
		\sum_{|I|\le N-4}\left(\|S\Gamma^{I}\Psi\|_{L_{x}^{2}}+\|\Gamma\Gamma^{I}\Psi\|_{L_{x}^{2}}\right)\lesssim \varepsilon \langle t\rangle^{\delta}.
		\end{equation}
		
		\item {\rm{Pointwise estimates for $S\Gamma^{I}\Psi$ and $\Gamma\Gamma^{I}\Psi$}.} It holds
		\begin{equation}\label{est:pointwiseSG}
		\sum_{|I|\le N-7}\left(|S\Gamma^{I}\Psi(t,x)|+|\Gamma\Gamma^{I}\Psi(t,x)|\right)\lesssim \varepsilon \langle t+r\rangle^{-\frac{1}{2}+\delta}.
		\end{equation}
	\end{enumerate}
\end{corollary}

\begin{proof}
	Proof of (i). Recall that, for $n=1,\dots,N-3$, we set 
	\begin{equation*}
	\begin{aligned}
	\mathcal{H}_{n}\big(t,\widetilde{\Psi}\big)&=\sum_{|I|=n-1}\left(\mathcal{F}\big(t,\Gamma^{I}\widetilde{\Psi}\big)+\mathcal{E}_{0}\big(t,\Gamma^{I}\widetilde{\Psi}\big)\right),\\
	\mathcal{W}\big(t,\widetilde{\Psi}\big)&=\sum_{n=0}^{N-3}2^{-n}\mathcal{H}_{n}\big(t,\widetilde{\Psi}\big),\quad \mbox{where}\ \ \mathcal{H}_{0}\big(t,\widetilde{\Psi}\big)=\big\|\widetilde{\Psi}(t,x)\big\|_{L_{x}^{2}}^{2}.
	\end{aligned}
	\end{equation*}
	Recall also that, in the Step 2, 3 and 4 of the proof of Lemma~\ref{le:estN1N2} (iii), we have shown that
	\begin{equation*}
	\begin{aligned}
	\big\|\mathcal{H}_{0}(t,\widetilde{\Psi})\big\|_{L_{x}^{2}}=\big\|\widetilde{\Psi}(t,x)\big\|_{L_{x}^{2}}&\lesssim \varepsilon \log^{\frac{3}{2}}(2+t),\\
	\sum_{|I|\le N-4}\big(\big\|S\Gamma^{I}\widetilde{\Psi}\big\|_{L_{x}^{2}}+\big\|\Gamma\Gamma^{I}\widetilde{\Psi}\big\|_{L_{x}^{2}}\big)&\lesssim \mathcal{W}(t,\widetilde{\Psi})^{\frac{1}{2}},
	\end{aligned}
	\end{equation*}
	as well as 
	\begin{equation*}
	\begin{aligned}
	\sum_{|I|\le N-4}\left(\big\|S\Gamma^{I}{\Psi}\big\|_{L_{x}^{2}}+\big\|\Gamma\Gamma^{I}{\Psi}\big\|_{L_{x}^{2}}\right)\lesssim C^{2}\varepsilon^{2}+\sum_{|I|\le N-4}\big(\big\|S\Gamma^{I}\widetilde{\Psi}\big\|_{L_{x}^{2}}+\big\|\Gamma\Gamma^{I}\widetilde{\Psi}\big\|_{L_{x}^{2}}\big).
	\end{aligned}
	\end{equation*}
	Note that, from~\eqref{est:smallness},~\eqref{est:Energywave},~\eqref{est:Con},~Lemma~\ref{le:estQ0} and Lemma~\ref{le:estN1N2}, we get
	\begin{equation*}
	\begin{aligned}
&\sum_{|I|\le N-4}\mathcal{F}\big(t,\Gamma^{I}\widetilde{\Psi}\big)^{\frac{1}{2}}+\sum_{|I|\le N-4}\mathcal{E}_{0}\big(t,\Gamma^{I}\widetilde{\Psi}\big)^{\frac{1}{2}}\\
&\lesssim 	\sum_{|I|\le N-4}\left(\mathcal{F}\big(0,\Gamma^{I}\widetilde{\Psi}\big)^{\frac{1}{2}}+\mathcal{E}_{0}\big(0,\Gamma^{I}\widetilde{\Psi}\big)^{\frac{1}{2}}\right)+\sum_{|I|\le N-4}\int_{0}^{t}\left\|\langle s+r\rangle\Gamma^{I}Q_{0}(\psi,v)\right\|_{L_{x}^{2}}\d s \\
&+\sum_{|I|\le N-4}\int_{0}^{t}\left(\left\|\langle s+r\rangle\Gamma^{I}\mathcal{N}_{1}(\psi,v)\right\|_{L_{x}^{2}}+\left\|\langle s+r\rangle\Gamma^{I}\mathcal{N}_{2}(\psi,\psi^{*})\right\|_{L_{x}^{2}}\right)\d s\\
&\lesssim \varepsilon+C^{2}\varepsilon^{2}\log (2+t)+C^{2}\varepsilon^{2}\langle t\rangle^{\delta}\lesssim \varepsilon\langle t\rangle^{\delta},
\end{aligned}
	\end{equation*}
	for $\varepsilon$ small enough (depending on $C$). Gathering above estimates, we obtain~\eqref{est:L2SG}.
	
	Proof of (ii). The inequality~\eqref{est:pointwiseSG} is a direct consequence of~\eqref{est:GloSob} and~\eqref{est:L2SG}.
	\end{proof}

Next, we introduce the weighted $L^{2}_{x}$ and $L_{t}^{1}L_{x}^{2}$ estimates for $\Gamma^{I}\mathcal{N}_{3}$ and $\Gamma^{I}\mathcal{N}_{4}$ (see \eqref{equ:Ns00} and \eqref{equ:Ns01} for the expressions of $\mathcal{N}_3$ and $\mathcal{N}_4$).

\begin{lemma}
	For any $t\in [0,T^{*}(\psi_{0},\vec{v}_{0}))$, the following estimates are true.
	\begin{enumerate}
	\item {\rm{Weighted $L_{x}^{2}$ estimates for $\Gamma^{I}\mathcal{N}_{3}$ and $\Gamma^{I}\mathcal{N}_{4}$.}} It holds
	\begin{equation}\label{est:N3N4}
	\begin{aligned}
      \sum_{|I|\le N-2}\left\|\langle t+r\rangle \Gamma^{I}\mathcal{N}_{3}(\psi,v)\right\|_{L_{x}^{2}}&\lesssim C^{3}\varepsilon^{3}\langle t\rangle^{-\frac{1}{2}},\\
      \sum_{|I|\le N-2}\left\|\langle t+r\rangle \Gamma^{I}\mathcal{N}_{4}(\psi,\psi^{*})\right\|_{L_{x}^{2}}&\lesssim 
      C\varepsilon^{2}\langle t\rangle^{-\frac{1}{2}+\delta}.
      \end{aligned}
	\end{equation}
	
	\item {\rm{$L_{t}^{1}L_{x}^{2}$ estimates for $\Gamma^{I}\mathcal{N}_{3}$ and $\Gamma^{I}\mathcal{N}_{4}$.}} It holds
	\begin{equation}\label{est:LtLxN3N4}
	\sum_{|I|\le N-1}\int_{0}^{t}\left(\left\|\Gamma^{I}\mathcal{N}_{3}(\psi,v)\right\|_{L_{x}^{2}}+\left\|\Gamma^{I}\mathcal{N}_{4}(\psi,\psi^{*})\right\|_{L_{x}^{2}}\right)\d s\lesssim C^{2}\varepsilon^{2}.
	\end{equation}
	
\end{enumerate}
\end{lemma}

\begin{proof}
	Proof of (i). First, from~\eqref{est:hatGa}, $1\le \bchar_{\cint}+\bchar_{\cext}$ and $N\ge 14$, we see that 
\begin{equation*}
\sum_{|I|\le N-2}\left\|\langle t+r\rangle \Gamma^{I}\mathcal{N}_{3}(\psi,v)\right\|_{L_{x}^{2}}\lesssim \mathcal{I}_{71}+\mathcal{I}_{72}+\mathcal{I}_{73},
\end{equation*}
where
\begin{equation*}
\begin{aligned}
	\mathcal{I}_{71}&=\sum_{|I_{1}|\le N-6}\sum_{|I_{2}|\le N-6}\sum_{|I_{3}|\le N-1}\big\|\langle t+r\rangle\left|\Gamma^{I_{1}}v\right|\big|\widehat{{\Gamma}}^{I_{2}}\psi\big|\big|\widehat{{\Gamma}}^{I_{3}}\psi\big|\big\|_{L_{x}^{2}},\\
\mathcal{I}_{72}&=\sum_{|I_{1}|\le N-1}\sum_{|I_{2}|\le N-7}\sum_{|I_{3}|\le N-4}\big\|\langle t+r\rangle\left|\Gamma^{I_{1}}v\right|\big|\widehat{{\Gamma}}^{I_{2}}\psi\big|\big|\widehat{{\Gamma}}^{I_{3}}\psi\big|\bchar_{\cint}\big\|_{L_{x}^{2}},\\
\mathcal{I}_{73}&=\sum_{|I_{1}|\le N-1}\sum_{|I_{2}|\le N-7}\sum_{|I_{3}|\le N-4}\big\|\langle t+r\rangle\left|\Gamma^{I_{1}}v\right|\big|\widehat{{\Gamma}}^{I_{2}}\psi\big|\big|\widehat{{\Gamma}}^{I_{3}}\psi\big|\bchar_{\cext}\big\|_{L_{x}^{2}}.
\end{aligned}
\end{equation*}
From~\eqref{est:Bootpsi},~\eqref{est:Bootv} and~\eqref{est:pointGlobal}, we check 
\begin{equation*}
\begin{aligned}
\mathcal{I}_{71}
&\lesssim \sum_{\substack{|I_{1}|\le N-6\\ |I_{2}|\le N-6}}\sum_{|I_{3}|\le N-1}\big\|\langle t+r\rangle\Gamma^{I_{1}}v\big\|_{L_{x}^{\infty}}\big\|\widehat{{\Gamma}}^{I_{2}}\psi\big\|_{L_{x}^{\infty}}\big\|\widehat{{\Gamma}}^{I_{3}}\psi\big\|_{L_{x}^{2}}\lesssim C^{3}\varepsilon^{3}\langle t\rangle^{-\frac{1}{2}}.
\end{aligned}
\end{equation*}
Then, using~\eqref{est:Bootpsi},~\eqref{est:Bootv},~\eqref{est:pointGlobal},~\eqref{est:intvhigh},~\eqref{est:outpsiv} and $0<\delta \ll 1$, we have 
\begin{equation*}
\begin{aligned}
\mathcal{I}_{72}
&\lesssim \sum_{|I_{1}|\le N}\sum_{\substack{|I_{2}|\le N-7\\ |I_{3}|\le N-4}}\left\|\partial \Gamma^{I_{1}}v\right\|_{L_{x}^{2}}\big\|\langle t-r\rangle\widehat{{\Gamma}}^{I_{2}}\psi\big\|_{L_{x}^{\infty}}\big\|\widehat{{\Gamma}}^{I_{3}}\psi\big\|_{L_{x}^{\infty}}\\
&+\sum_{|I_{1}|\le N-1}\sum_{\substack{|I_{2}|\le N-7\\ |I_{3}|\le N-4}}\langle t\rangle^{\frac{1}{2}}\big\|\widehat{{\Gamma}}^{I_{1}}\psi\big\|_{L_{x}^{2}}\big\|\widehat{\Gamma}^{I_{2}}\psi \big\|_{L_{x}^{\infty}}\big\|\widehat{\Gamma}^{I_{3}}\psi \big\|_{L_{x}^{\infty}}\lesssim C^{3}\varepsilon^{3}\langle t\rangle^{-\frac{1}{2}},
\end{aligned}
\end{equation*}
\begin{equation*}
\mathcal{I}_{73}\lesssim \sum_{|I_{1}|\le N-1}\sum_{\substack{|I_{2}|\le N-7\\ |I_{3}|\le N-4}}\left\|\Gamma^{I_{1}}v\right\|_{L_{x}^{2}}\big\|\langle t+r\rangle\big|\widehat{{\Gamma}}^{I_{2}}\psi\big|\big|\widehat{{\Gamma}}^{I_{3}}\psi\big|\bchar_{\cext}\big\|_{L_{x}^{\infty}}\lesssim C^{3}\varepsilon^{3}\langle t\rangle^{-\frac{1}{2}}.
\end{equation*}
Gathering these estimates, we have proved~\eqref{est:N3N4} for $\Gamma^{I}\mathcal{N}_{3}$.

Second, we notice that the nonlinear term $\mathcal{N}_{4}$ enjoy a null structure. Hence, from~\eqref{est:GaQ0fg},~\eqref{est:pointwiseSG}, $N\ge 14$ and $[\partial_{\alpha},\Gamma_{i}]\in \rm{Span}\left\{\partial_{t},\partial_{1},\partial_{2}\right\}$, we have 
\begin{equation*}
\begin{aligned}
&\sum_{|I|\le N-2}\left\|\langle t+r\rangle \Gamma^{I}\mathcal{N}_{4}(\psi,\psi^{*})\right\|_{L_{x}^{2}}\\
&\lesssim 
\sum_{\substack{|I_{1}|\le N-8\\ |I_{2}|\le N-1}}\big\|\big(|S\Gamma^{I_{1}}\psi|+|\Gamma\Gamma^{I_{1}}\psi|)\big|\Gamma^{I_{2}}\psi\big|\big\|_{L_{x}^{2}}\\
&\lesssim \sum_{\substack{|I_{1}|\le N-7\\ |I_{2}|\le N-1}}\left(\big\|S\Gamma^{I_{1}}\Psi\big\|_{L_{x}^{\infty}}+\big\|\Gamma\Gamma^{I_{1}}\Psi|\big\|_{L_{x}^{\infty}}\right)\big\|\Gamma^{I_{2}}\psi\big\|_{L_{x}^{2}}\lesssim C\varepsilon^{2}\langle t\rangle^{-\frac{1}{2}+\delta},
\end{aligned}
\end{equation*}
which means~\eqref{est:N3N4} for $\Gamma^{I}\mathcal{N}_{4}$.

Proof of (ii). Using again~\eqref{est:hatGa},~\eqref{est:GaQ0fg} and $N\ge 14$, we see that 
\begin{equation*}
	\sum_{|I|\le N-1}\int_{0}^{t}\left(\left\|\Gamma^{I}\mathcal{N}_{3}(\psi,v)\right\|_{L_{x}^{2}}+\left\|\Gamma^{I}\mathcal{N}_{4}(\psi,\psi^{*})\right\|_{L_{x}^{2}}\right)\d s\lesssim \mathcal{I}_{74}+\mathcal{I}_{75}+\mathcal{I}_{76},
\end{equation*}
where
\begin{equation*}
\begin{aligned}
\mathcal{I}_{74}&=\sum_{|I_{1}|\le N-6}\sum_{|I_{2}|\le N-6}\sum_{|I_{3}|\le N}\int_{0}^{t}\big\|\left|\Gamma^{I_{1}}v\right|\big|\widehat{{\Gamma}}^{I_{2}}\psi\big|\big|\widehat{{\Gamma}}^{I_{3}}\psi\big|\big\|_{L_{x}^{2}}\d s,\\
\mathcal{I}_{75}&=\sum_{|I_{1}|\le N}\sum_{|I_{2}|\le N-7}\sum_{|I_{3}|\le N-7}\int_{0}^{t}\big\|\left|\Gamma^{I_{1}}v\right|\big|\widehat{{\Gamma}}^{I_{2}}\psi\big|\big|\widehat{{\Gamma}}^{I_{3}}\psi\big|\big\|_{L_{x}^{2}}\d s,\\
\mathcal{I}_{76}&=\sum_{|I_{1}|\le N-8}\sum_{|I_{2}|\le N}\int_{0}^{t}\langle s\rangle^{-1}\big\|\big(|S\Gamma^{I_{1}}\psi|+|\Gamma\Gamma^{I_{1}}\psi|)\big|\Gamma^{I_{2}}\psi\big|\big\|_{L_{x}^{2}}\d s.
\end{aligned}
\end{equation*}
Note that, from the bootstrap assumption~\eqref{est:Bootpsi}-\eqref{est:Bootv}, we have 
\begin{equation*}
\begin{aligned}
&\sum_{|I|\le N}\int_{0}^{t}\langle s\rangle^{-\frac{3}{4}}\left\|\frac{\Gamma^{I}v}{\langle s-r\rangle}\right\|_{L_{x}^{2}}\d s\\
&\lesssim \sum_{|I|\le N}\left(\int_{0}^{t}\langle s\rangle^{-\frac{3}{2}+5\delta}\d s \right)^{\frac{1}{2}}\left(\int_{0}^{t}\langle s\rangle^{-5\delta}\left\|\frac{\Gamma^{I}v}{\langle s-r\rangle}\right\|^{2}_{L_{x}^{2}}\d s\right)^{\frac{1}{2}}\lesssim C^{\frac{1}{2}}\varepsilon^{\frac{1}{2}}.
\end{aligned}
\end{equation*}
Therefore, from~\eqref{est:Bootpsi},~\eqref{est:Bootv} and~\eqref{est:pointwiseSG}, we have 
\begin{equation*}
\begin{aligned}
\mathcal{I}_{74}
&\lesssim \sum_{\substack{|I_{1}|\le N-6\\|I_{2}|\le N-6}}\sum_{|I_{3}|\le N}\int_{0}^{t}\left\|\Gamma^{I_{1}}v\right\|_{L_{x}^{\infty}}\big\|\widehat{{\Gamma}}^{I_{2}}\psi\big\|_{L_{x}^{\infty}}\big\|\widehat{{\Gamma}}^{I_{3}}\psi\big\|_{L_{x}^{2}}\d s\\
&\lesssim C^{3}\varepsilon^{3}\int_{0}^{t}\langle s\rangle^{-1}\langle s\rangle^{-\frac{1}{2}}\langle s\rangle^{\delta}\d s \lesssim C^{3}\varepsilon^{3}\int_{0}^{t}\langle s\rangle^{-\frac{3}{2}+\delta}\d s\lesssim C^{3}\varepsilon^{3},
\end{aligned}
\end{equation*}
\begin{equation*}
\begin{aligned}
\mathcal{I}_{75}
&\lesssim \sum_{|I_{1}|\le N}\sum_{\substack{|I_{2}|\le N-7\\ |I_{3}|\le N-7}}\int_{0}^{t}\left\|\frac{\Gamma^{I_{1}}v}{\langle s-r\rangle}\right\|_{L_{x}^{2}}\big\|\langle s-r\rangle\widehat{{\Gamma}}^{I_{2}}\psi\big\|_{L_{x}^{\infty}}\big\|\widehat{{\Gamma}}^{I_{3}}\psi\big\|_{L_{x}^{\infty}}\d s\\
&\lesssim C^{2}\varepsilon^{2} \sum_{|I_{1}|\le N}\int_{0}^{t}\langle s\rangle^{-1+\delta}\left\|\frac{\Gamma^{I_{1}}v}{\langle s-r\rangle}\right\|_{L_{x}^{2}}\d s\lesssim C^{\frac{5}{2}}\varepsilon^{\frac{5}{2}},
\end{aligned}
\end{equation*}
\begin{equation*}
\begin{aligned}
\mathcal{I}_{76}&\lesssim \sum_{\substack{|I_{1}|\le N-7\\ |I_{2}|\le N}}\int_{0}^{t}\langle s\rangle^{-1}\left(\|S\Gamma^{I_{1}}\Psi\|_{L_{x}^{\infty}}+\|\Gamma\Gamma^{I_{1}}\Psi\|_{L_{x}^{\infty}}\right)\big\|\Gamma^{I_{2}}\psi\big\|_{L_{x}^{2}}\d s\\
&\lesssim C\varepsilon^{2}\int_{0}^{t}\langle s\rangle^{-1}\langle s\rangle^{-\frac{1}{2}+\delta}\langle s\rangle^{\delta}\d s \lesssim C\varepsilon^{2}\int_{0}^{t}\langle s\rangle^{-\frac{3}{2}+2\delta}\d s\lesssim C\varepsilon^{2}.
\end{aligned}
\end{equation*}
Gathering these estimates, we have proved~\eqref{est:LtLxN3N4}.

	\end{proof}

Last, we introduce the estimates for highest-order nonlinear terms.
\begin{lemma}\label{le:highest}
	For all $t\in [0,T^{*}(\psi_{0},\vec{v}_{0}))$, the following estimates are true.
	\begin{enumerate}
		\item {\rm{$L_{t}^{1}L_{x}^{1}$ and $L_{t}^{1}L_{x}^{2}$ estimates for highest-order nonlinear terms.}} We have 
		\begin{align}
		\sum_{|I|\le N}\int_{0}^{t}\left\|\Gamma^{I}\left(\psi^{*}\gamma^{0}\psi\right)\right\|_{L_{x}^{2}}\d s&\lesssim C^{2}\varepsilon^{2}\langle t\rangle^{2\delta},\label{est:LtLx2psiN}\\
		\sum_{|I|\le N}\int_{0}^{t}\int_{\R^{2}}\big|\big(\widehat{{\Gamma}}^{I}\psi\big)^{*}\gamma^{0}\widehat{{\Gamma}}^{I}(v\psi)\big|\d x \d s&\lesssim C^{\frac{5}{2}}\varepsilon^{\frac{5}{2}}\langle t\rangle^{2\delta}.\label{est:LtLx1vpsiN}
		\end{align}
		
		\item {\rm{Weighted estimates for highest-order nonlinear terms.}} We have 
	\end{enumerate}
	\begin{align}
		\sum_{|I|\le N}	\int_{0}^{t}\langle s\rangle^{-3\delta}\left\|\Gamma^{I}\left(\psi^{*}\gamma^{0}\psi\right)\right\|_{L_{x}^{2}}\d s&\lesssim C^{\frac{3}{2}}\varepsilon^{\frac{3}{2}},\label{est:GhostpsiN}\\
	\sum_{|I|\le N}\int_{0}^{t}\langle s\rangle^{-3\delta}\int_{\R^{2}}\big|\big(\widehat{{\Gamma}}^{I}\psi\big)^{*}\gamma^{0}\widehat{{\Gamma}}^{I}(v\psi)\big|\d x \d s &\lesssim C^{2}\varepsilon^{2}.\label{est:GhostvN}
	\end{align}
\end{lemma}
\begin{proof}
	Proof of (i). First, from~\eqref{est:hatGG},~\eqref{equ:HiddenPhi} and $N\ge 14$, we find
	\begin{equation*}
\sum_{|I|\le N}	\int_{0}^{t}\left\|\Gamma^{I}\left(\psi^{*}\gamma^{0}\psi\right)\right\|_{L_{x}^{2}}\d s\lesssim \mathcal{I}_{81}+\mathcal{I}_{82},
	\end{equation*}
	where
	\begin{equation*}
	\begin{aligned}
	\mathcal{I}_{81}&=\sum_{|I_{1}|\le N}\sum_{|I_{2}|\le N-7}\int_{0}^{t}\big\|\big|\big[\widehat{{\Gamma}}^{I_{1}}\psi\big]_{-}\big|\big|\widehat{{\Gamma}}^{I_{2}}\psi\big|\big\|_{L_{x}^{2}}\d s,\\
	\mathcal{I}_{82}&=\sum_{|I_{1}|\le N-7}\sum_{|I_{2}|\le N}\int_{0}^{t}\big\|\big|\big[\widehat{{\Gamma}}^{I_{1}}\psi\big]_{-}\big|\big|\widehat{{\Gamma}}^{I_{2}}\psi\big|\big\|_{L_{x}^{2}}\d s.
	\end{aligned}
	\end{equation*}
	Note that, from the bootstrap assumption~\eqref{est:Bootpsi}, we have 
	\begin{equation*}
	\begin{aligned}
	&\sum_{|I|\le N}\int_{0}^{t}\langle s\rangle^{-\frac{1}{2}+\delta}\left\|\frac{\big[\widehat{{\Gamma}}^{I}\psi\big]_{-}}{\langle s-r\rangle}\right\|_{L_{x}^{2}}\d s\\
	&\lesssim \sum_{|I|\le N}\left(\int_{0}^{t}\langle s\rangle^{-1+2\delta}\d s\right)^{\frac{1}{2}}\left(\int_{0}^{t}\left\|\frac{\big[\widehat{{\Gamma}}^{I}\psi\big]_{-}}{\langle s-r\rangle}\right\|^{2}_{L_{x}^{2}}\d s\right)^{\frac{1}{2}}\lesssim C\varepsilon\langle t\rangle^{2\delta}.
	\end{aligned}
	\end{equation*}
	Therefore, using again~\eqref{est:Bootpsi} and the Cauchy-Schwarz inequality, we have 
	\begin{equation*}
	\begin{aligned}
	\mathcal{I}_{81}
	&\lesssim \sum_{\substack{|I_{1}|\le N\\|I_{2}|\le N-7}}\int_{0}^{t}\left\|\frac{\big[\widehat{{\Gamma}}^{I_{1}}\psi\big]_{-}}{\langle s-r\rangle}\right\|_{L_{x}^{2}}\big\|\langle s-r\rangle\widehat{{\Gamma}}^{I_{2}}\psi\big\|_{L_{x}^{\infty}}\d s\\
	&\lesssim C\varepsilon\sum_{|I_{1}|\le N}\int_{0}^{t}\langle s\rangle^{-\frac{1}{2}+\delta}\left\|\frac{\big[\widehat{{\Gamma}}^{I_{1}}\psi\big]_{-}}{\langle s-r\rangle}\right\|_{L_{x}^{2}}\d s\lesssim C^{2}\varepsilon^{2}\langle t\rangle^{2\delta},
	\end{aligned}
	\end{equation*}
		\begin{equation*}
	\mathcal{I}_{82}\lesssim \sum_{\substack{|I_{1}|\le N-7\\ |I_{2}|\le N}}\int_{0}^{t}\big\|\big[\widehat{{\Gamma}}^{I_{1}}\psi\big]_{-}\big\|_{L_{x}^{\infty}}\big\|\widehat{{\Gamma}}^{I_{2}}\psi\big\|_{L_{x}^{2}}\d s\lesssim C^{2}\varepsilon^{2}\int_{0}^{t}\langle s\rangle^{-\frac{3}{2}+2\delta}\d s \lesssim C^{2}\varepsilon^{2}.
	\end{equation*}
	Gathering these estimates, we have proved~\eqref{est:LtLx2psiN}.
	
	Second, from,~\eqref{est:hatGa},~\eqref{est:hatGfPhi},~\eqref{equ:HiddenPhi} and $N\ge 14$, we have 
	\begin{equation*}
	\sum_{|I|\le N}\int_{0}^{t}\int_{\R^{2}}\big|\big(\widehat{{\Gamma}}^{I}\psi\big)^{*}\gamma^{0}\widehat{{\Gamma}}^{I}(v\psi)\big|\d x \d s\lesssim \mathcal{I}_{83}+\mathcal{I}_{84}+\mathcal{I}_{85},
	\end{equation*}
	
		where
	\begin{equation*}
	\begin{aligned}
	\mathcal{I}_{83}&=\sum_{|I|\le N}\sum_{|I_{1}|\le N}\sum_{|I_{2}|\le N-6}\int_{0}^{t}\int_{\R^{2}}\big|\widehat{{\Gamma}}^{I}\psi\big|\big|\big[\widehat{{\Gamma}}^{I_{1}}\psi\big]_{-}\big|\left|\Gamma^{I_{2}}v\right|\d x \d s,\\
	\mathcal{I}_{84}&=\sum_{|I|\le N}\sum_{|I_{1}|\le N-7}\sum_{|I_{2}|\le N}\int_{0}^{t}\int_{\R^{2}}\big|\widehat{{\Gamma}}^{I}\psi\big|\big|\big[\widehat{{\Gamma}}^{I_{1}}\psi\big]_{-}\big|\left|\Gamma^{I_{2}}v\right|\d x \d s,\\
	\mathcal{I}_{85}&=\sum_{|I|\le N}\sum_{|I_{1}|\le N-7}\sum_{|I_{2}|\le N}\int_{0}^{t}\int_{\R^{2}}\big|\big[\widehat{{\Gamma}}^{I}\psi\big]_{-}\big|\big|\widehat{{\Gamma}}^{I_{1}}\psi\big|\left|\Gamma^{I_{2}}v\right|\d x \d s.
	\end{aligned}
	\end{equation*}
	
	Using~\eqref{est:Bootpsi} and~\eqref{est:Bootv}, we find
	\begin{equation*}
	\begin{aligned}
	\mathcal{I}_{83}
	&\lesssim \sum_{\substack{|I|\le N\\ |I_{1}|\le N}}\sum_{|I_{2}|\le N-6}\int_{0}^{t}\big\|\widehat{{\Gamma}}^{I}\psi\big\|_{L_{x}^{2}}\big\|\widehat{{\Gamma}}^{I_{1}}\psi\big\|_{L_{x}^{2}}\left
	\|\Gamma^{I_{2}}v\right
	\|_{L_{x}^{\infty}} \d s\\
	&\lesssim C^{3}\varepsilon^{3}\int_{0}^{t}\langle s\rangle^{\delta}\langle s\rangle^{\delta}\langle s\rangle^{-1}\d s \lesssim C^{3}\varepsilon^{3}\int_{0}^{t}\langle s\rangle^{-1+2\delta}\d s\lesssim C^{3}\varepsilon^{3}\langle t\rangle^{2\delta},
	\end{aligned}
	\end{equation*}
	
	\begin{equation*}
	\begin{aligned}
	\mathcal{I}_{84}
	&\lesssim \sum_{\substack{|I|\le N\\ |I_{2}|\le N}}\sum_{|I_{1}|\le N-7}\int_{0}^{t}\big\|\widehat{{\Gamma}}^{I}\psi\big\|_{L_{x}^{2}}\big\|\big[\widehat{{\Gamma}}^{I_{1}}\psi\big]_{-}\big\|_{L_{x}^{\infty}}\left\|\Gamma^{I_{2}}v\right\|_{L_{x}^{2}} \d s\\
	&\lesssim C^{3}\varepsilon^{3}\int_{0}^{t}\langle s\rangle^{\delta}\langle s\rangle^{-\frac{3}{2}+\delta}\langle s \rangle^{2\delta}\d s \lesssim C^{3}\varepsilon^{3}\int_{0}^{t}\langle s\rangle^{-\frac{3}{2}+4\delta}\d s \lesssim C^{3}\varepsilon^{3}.
	\end{aligned}
	\end{equation*}
	Then, using again~\eqref{est:Bootpsi} and~\eqref{est:Bootv}, we have 
	\begin{equation*}
	\begin{aligned}
	\mathcal{I}_{85}
	&\lesssim C\varepsilon \sum_{\substack{|I|\le N\\ |I_{2}|\le N}}\int_{0}^{t}\langle s\rangle^{-\frac{3}{10}+\delta}
	\left\|\frac{\big[\widehat{{\Gamma}}^{I}\psi\big]_{-}}{\langle r-s\rangle^{\frac{3}{5}}}\right\|_{L_{x}^{2}}\left\|\frac{\Gamma^{I_{2}}v}{\langle r-s\rangle^{\frac{3}{5}}}\right\|_{L_{x}^{2}}\d s\\
	&\lesssim C\varepsilon \sum_{\substack{|I|\le N\\ |I_{2}|\le N}}\left(\int_{0}^{t}	\left\|\frac{\big[\widehat{{\Gamma}}^{I}\psi\big]_{-}}{\langle r-s\rangle^{\frac{3}{5}}}\right\|_{L_{x}^{2}}^{2}\d s\right)^{\frac{1}{2}}\left(\int_{0}^{t}\langle s\rangle^{-5\delta}\left\|\frac{\Gamma^{I_{2}}v}{\langle r-s\rangle^{\frac{3}{5}}}\right\|_{L_{x}^{2}}^{2}\d s \right)^{\frac{1}{2}}\lesssim C^{\frac{5}{2}}\varepsilon^{\frac{5}{2}}\langle t\rangle^{\delta}.
	\end{aligned}
	\end{equation*}
	
	Gathering these estimates, we have proved~\eqref{est:LtLx1vpsiN}.

	Proof of (ii). First, from~\eqref{est:hatGG},~\eqref{equ:HiddenPhi} and $N\ge 14$, we find
	
	\begin{equation*}
	\sum_{|I|\le N}	\int_{0}^{t}\langle s\rangle^{-3\delta}\left\|\Gamma^{I}\left(\psi^{*}\gamma^{0}\psi\right)\right\|_{L_{x}^{2}}\d s\lesssim \mathcal{I}_{91}+\mathcal{I}_{92},
	\end{equation*}
	where
	\begin{equation*}
	\begin{aligned}
	\mathcal{I}_{91}&=\sum_{|I_{1}|\le N-7}\sum_{|I_{2}|\le N}\int_{0}^{t}\langle s\rangle^{-3\delta}\big\|\big|\big[\widehat{{\Gamma}}^{I_{1}}\psi\big]_{-}\big|\big|\widehat{{\Gamma}}^{I_{2}}\psi\big|\big\|_{L_{x}^{2}}\d s,\\
	\mathcal{I}_{92}&=\sum_{|I_{1}|\le N}\sum_{|I_{2}|\le N-7}\int_{0}^{t}\langle s\rangle^{-3\delta}\big\|\big|\big[\widehat{{\Gamma}}^{I_{1}}\psi\big]_{-}\big|\big|\widehat{{\Gamma}}^{I_{2}}\psi\big|\big\|_{L_{x}^{2}}\d s.
	\end{aligned}
	\end{equation*}
	Note that, using~\eqref{est:Bootpsi} and the Cauchy-Schwarz inequality, we have 
	\begin{equation*}
	\begin{aligned}
&	\sum_{|I|\le N}\int_{0}^{t}\langle s\rangle^{-\frac{1}{2}-2\delta}\left\|\frac{\big[\widehat{{\Gamma}}^{I}\psi\big]_{-}}{\langle r-s\rangle}\right\|_{L_{x}^{2}}\d s \\
&\lesssim \sum_{|I|\le N}\left(\int_{0}^{t}\langle s\rangle^{-1-\delta}\d s\right)^{\frac{1}{2}}\left(\int_{0}^{t}\langle s\rangle^{-3\delta}\left\|\frac{\big[\widehat{{\Gamma}}^{I}\psi\big]_{-}}{\langle r-s\rangle}\right\|_{L_{x}^{2}}^{2}\d s\right)^{\frac{1}{2}}\lesssim C^{\frac{1}{2}}\varepsilon^{\frac{1}{2}}.
	\end{aligned}
	\end{equation*}
	Therefore, using again~\eqref{est:Bootpsi} and~\eqref{est:Bootv}, we check
	\begin{equation*}
	\begin{aligned}
	\mathcal{I}_{91}&\lesssim \sum_{\substack{|I_{1}|\le N-7\\ |I_{2}|\le N}}\int_{0}^{t}\langle s\rangle^{-3\delta}\big\|\big[\widehat{{\Gamma}}^{I_{1}}\psi\big]_{-}\big\|_{L_{x}^{\infty}}\big\|\widehat{{\Gamma}}^{I_{2}}\psi\big\|_{L_{x}^{2}}\d s\\
	&\lesssim C^{2}\varepsilon^{2}\int_{0}^{t}\langle s\rangle^{-3\delta}\langle s\rangle^{-\frac{3}{2}+\delta}\langle s\rangle^{\delta}\d s \lesssim C^{2}\varepsilon^{2}\int_{0}^{t}\langle s\rangle^{-\frac{3}{2}-\delta}\d s \lesssim C^{2}\varepsilon^{2},
	\end{aligned}
	\end{equation*}
	\begin{equation*}
	\begin{aligned}
	\mathcal{I}_{92}&\lesssim \sum_{\substack{|I_{1}|\le N\\ |I_{2}|\le N-7}}\int_{0}^{t}\langle s\rangle^{-3\delta}\left\|\frac{\big[\widehat{{\Gamma}}^{I_{1}}\psi\big]_{-}}{\langle s-r\rangle}\right\|_{L_{x}^{2}}\big\|\langle s-r\rangle\widehat{{\Gamma}}^{I_{2}}\psi\big\|_{L_{x}^{\infty}}\d s\\
	&\lesssim C\varepsilon\sum_{|I_{1}|\le N}\int_{0}^{t}\langle s\rangle^{-\frac{1}{2}-2\delta}\left\|\frac{\big[\widehat{{\Gamma}}^{I_{1}}\psi\big]_{-}}{\langle s-r\rangle}\right\|_{L_{x}^{2}}\d s \lesssim C^{\frac{3}{2}}\varepsilon^{\frac{3}{2}}.
	\end{aligned}
	\end{equation*}
	Gathering these estimates, we have proved~\eqref{est:GhostpsiN}.
	
	Second, from,~\eqref{est:hatGa},~\eqref{est:hatGfPhi},~\eqref{equ:HiddenPhi} and $N\ge 14$, we have 
\begin{equation*}
\sum_{|I|\le N}\int_{0}^{t}\langle s\rangle^{-3\delta}\int_{\R^{2}}\big|\big(\widehat{{\Gamma}}^{I}\psi\big)^{*}\gamma^{0}\widehat{{\Gamma}}^{I}(v\psi)\big|\d x \d s\lesssim \mathcal{I}_{93}+\mathcal{I}_{94}+\mathcal{I}_{95},
\end{equation*}

where
\begin{equation*}
\begin{aligned}
\mathcal{I}_{93}&=\sum_{|I|\le N}\sum_{|I_{1}|\le N}\sum_{|I_{2}|\le N-6}\int_{0}^{t}\langle s\rangle^{-3\delta}\int_{\R^{2}}\big|\widehat{{\Gamma}}^{I}\psi\big|\big|\big[\widehat{{\Gamma}}^{I_{1}}\psi\big]_{-}\big|\left|\Gamma^{I_{2}}v\right|\d x \d s,\\
\mathcal{I}_{94}&=\sum_{|I|\le N}\sum_{|I_{1}|\le N-7}\sum_{|I_{2}|\le N}\int_{0}^{t}\langle s\rangle^{-3\delta}\int_{\R^{2}}\big|\widehat{{\Gamma}}^{I}\psi\big|\big|\big[\widehat{{\Gamma}}^{I_{1}}\psi\big]_{-}\big|\left|\Gamma^{I_{2}}v\right|\d x \d s,\\
\mathcal{I}_{95}&=\sum_{|I|\le N}\sum_{|I_{1}|\le N-7}\sum_{|I_{2}|\le N}\int_{0}^{t}\langle s\rangle^{-3\delta}\int_{\R^{2}}\big|\big[\widehat{{\Gamma}}^{I}\psi\big]_{-}\big|\big|\widehat{{\Gamma}}^{I_{1}}\psi\big|\left|\Gamma^{I_{2}}v\right|\d x \d s.
\end{aligned}
\end{equation*}
	Using~\eqref{est:Bootpsi} and~\eqref{est:Bootv}, we find
\begin{equation*}
\begin{aligned}
\mathcal{I}_{93}
&\lesssim \sum_{\substack{|I|\le N\\ |I_{1}|\le N}}\sum_{|I_{2}|\le N-6}\int_{0}^{t}\langle s\rangle^{-3\delta}\big\|\widehat{{\Gamma}}^{I}\psi\big\|_{L_{x}^{2}}\big\|\widehat{{\Gamma}}^{I_{1}}\psi\big\|_{L_{x}^{2}}\left
\|\Gamma^{I_{2}}v\right
\|_{L_{x}^{\infty}} \d s\\
&\lesssim C^{3}\varepsilon^{3}\int_{0}^{t}\langle s\rangle^{-3\delta}\langle s\rangle^{\delta}\langle s\rangle^{\delta}\langle s\rangle^{-1}\d s \lesssim C^{3}\varepsilon^{3}\int_{0}^{t}\langle s\rangle^{-1-\delta}\d s\lesssim C^{3}\varepsilon^{3},
\end{aligned}
\end{equation*}

\begin{equation*}
\begin{aligned}
\mathcal{I}_{94}
&\lesssim \sum_{\substack{|I|\le N\\ |I_{2}|\le N}}\sum_{|I_{1}|\le N-7}\int_{0}^{t}\langle s\rangle^{-3\delta}\big\|\widehat{{\Gamma}}^{I}\psi\big\|_{L_{x}^{2}}\big\|\big[\widehat{{\Gamma}}^{I_{1}}\psi\big]_{-}\big\|_{L_{x}^{\infty}}\left\|\Gamma^{I_{2}}v\right\|_{L_{x}^{2}} \d s\\
&\lesssim C^{3}\varepsilon^{3}\int_{0}^{t}\langle s\rangle^{-3\delta}\langle s\rangle^{\delta}\langle s\rangle^{-\frac{3}{2}+\delta}\langle s \rangle^{2\delta}\d s \lesssim C^{3}\varepsilon^{3}\int_{0}^{t}\langle s\rangle^{-\frac{3}{2}+\delta}\d s \lesssim C^{3}\varepsilon^{3}.
\end{aligned}
\end{equation*}
Then, using again~\eqref{est:Bootpsi} and~\eqref{est:Bootv}, we have 
\begin{equation*}
\begin{aligned}
\mathcal{I}_{95}
&\lesssim C\varepsilon \sum_{\substack{|I|\le N\\ |I_{2}|\le N}}\int_{0}^{t}\langle s\rangle^{-\frac{3}{10}-2\delta}
\left\|\frac{\big[\widehat{{\Gamma}}^{I}\psi\big]_{-}}{\langle r-s\rangle^{\frac{3}{5}}}\right\|_{L_{x}^{2}}\left\|\frac{\Gamma^{I_{2}}v}{\langle r-s\rangle^{\frac{3}{5}}}\right\|_{L_{x}^{2}}\d s\\
&\lesssim C\varepsilon \sum_{|I|\le N}\int_{0}^{t}\langle s\rangle^{-5\delta}\left(\left\|\frac{\big[\widehat{{\Gamma}}^{I}\psi\big]_{-}}{\langle r-s\rangle^{\frac{3}{5}}}\right\|_{L_{x}^{2}}^{2}+\left\|\frac{\Gamma^{I}v}{\langle r-s\rangle^{\frac{3}{5}}}\right\|_{L_{x}^{2}}^{2}\right)\d s\lesssim C^{2}\varepsilon^{2}.
\end{aligned}
\end{equation*}
Gathering these estimates, we have proved~\eqref{est:GhostvN}.
	\end{proof}

\subsection{Proof of Proposition~\ref{pro:main}}\label{Se:End}
We prove the Proposition~\ref{pro:main} by improving all the estimates of $(\psi,v)$ in~\eqref{est:Bootpsi} and~\eqref{est:Bootv}.

\begin{proof}[Proof of Proposition~\ref{pro:main}]
	For any initial data $(\psi_{0},\vec{v}_{0})$ satisfying the smallness condition~\eqref{est:smallness}, we consider the corresponding solution $(\psi,v)$ of~\eqref{equ:DKG}. From the smallness condition~\eqref{est:smallness}, we see that 
	\begin{equation}\label{est:initialenergy}
	\begin{aligned}
	\sum_{|I|\le N-1}\mathcal{E}^{D}(0,\widehat{\Gamma}^{I}{\widetilde{\psi}})^{\frac{1}{2}}+\sum_{|I|\le N}\left(\mathcal{E}^{D}(0,\widehat{\Gamma}^{I}{\psi})^{\frac{1}{2}}+\|\langle r\rangle \widehat{\Gamma}^{I}\psi(0,x)\|_{L_{x}^{2}}\right)&\lesssim \varepsilon,\\
	\sum_{|I|\le N-1}\mathcal{G}_{1}(0,{{\Gamma}}^{I}\tilde{v})^{\frac{1}{2}}+\sum_{|I|\le N}\left(\mathcal{G}_{1}(0,{{\Gamma}}^{I}v)^{\frac{1}{2}}+\|\langle r\rangle \Gamma^{I}v(0,x)\|_{H_{x}^{1}}\right)&\lesssim \varepsilon.
	\end{aligned}
	\end{equation}

	\textbf{Step 1.} Closing the estimates in $\mathcal{E}^{D}(t,\widehat{\Gamma}^{I}{\psi})$.
	Note that, from \eqref{est:hatGa}, \eqref{est:Energyphi1}, \eqref{equ:Hiddenpsi}, \eqref{est:LtLxQ0}, \eqref{est:LtLxN1N2} and \eqref{est:initialenergy}, we have 
	\begin{equation*}
	\begin{aligned}
	\sum_{|I|\le N-1}\mathcal{E}^{D}(t,\widehat{{\Gamma}}^{I}\widetilde{\psi})^{\frac{1}{2}}
	&\lesssim \sum_{|I|\le N-1}\mathcal{E}^{D}(0,\widehat{{\Gamma}}^{I}\widetilde{\psi})^{\frac{1}{2}}+\sum_{|I|\le N-1}\int_{0}^{t}\big\|\Gamma^{I}Q_{0}(\psi,v)\big\|_{L_{x}^{2}}\d s \\
	&+\sum_{|I|\le N-1}\int_{0}^{t}\left(\big\|\Gamma^{I}\mathcal{N}_{1}(\psi,v)\big\|_{L_{x}^{2}}+\big\|\Gamma^{I}\mathcal{N}_{2}(\psi,v)\big\|_{L_{x}^{2}}\right)\d s\\
	&\lesssim \varepsilon+C^{2}\varepsilon^{2}+C^{3}\varepsilon^{3},
	\end{aligned}
	\end{equation*}
	 Then, from~\eqref{est:hatGa},~\eqref{est:pointGlobal} and $N\ge 14$, 
	 \begin{equation*}
	 \begin{aligned}
	 &\sum_{|I|\le N-1}\big\|\widehat{{\Gamma}}^{I}\big(\gamma^{\mu}\partial_{\mu}(v\psi)\big)\big\|_{L_{x}^{2}}\\
	 &\lesssim \sum_{\substack{|I_{1}|\le N\\ |I_{2}|\le N-4}}\left\|{{\Gamma}}^{I_{1}}v\right\|_{L_{x}^{2}} \big\|\widehat{{\Gamma}}^{I_{2}}\psi\big\|_{L_{x}^{\infty}}+\sum_{\substack{|I_{1}|\le N-4\\ |I_{2}|\le N}}\left\|{{\Gamma}}^{I_{1}}v\right\|_{L_{x}^{\infty}} \big\|\widehat{{\Gamma}}^{I_{2}}\psi\big\|_{L_{x}^{2}}\lesssim C^{2}\varepsilon^{2}.
	 \end{aligned}
	 \end{equation*}
	 Next, using~\eqref{est:hatGfPhi},~\eqref{est:Bootpsi},~\eqref{est:Bootv} and $N\ge 14$, we have 
	 \begin{equation*}
	 \begin{aligned}
	 &\sum_{|I|\le N}\bigg(\int_{0}^{t}\int_{\R^{2}}\frac{\big|\big[\widehat{{\Gamma}}^{I}(v\psi)\big]_{-}\big|^{2}}{\langle r-s\rangle^{\frac{6}{5}}}\d x \d s \bigg)^{\frac{1}{2}}\\
	 &\lesssim \sum_{\substack{|I_{1}|\le N\\ |I_{2}|\le N-7}}\left(\int_{0}^{t}\left(\big\|\big[\widehat{{\Gamma}}^{I_{2}}\psi\big]_{-}\big\|^{2}_{L_{x}^{\infty}}\int_{\R^{2}}\frac{\left|\Gamma^{I_{1}}v\right|^{2}}{\langle r-s\rangle^{\frac{6}{5}}}\d x\right)\d s \right)^{\frac{1}{2}}\\
	 &+\sum_{\substack{|I_{1}|\le N-6\\ |I_{2}|\le N}}\left(\int_{0}^{t}\left\|\Gamma^{I_{1}}v\right\|_{L_{x}^{\infty}}^{2}\big\|\widehat{{\Gamma}}^{I_{2}}\psi\big\|_{L_{x}^{2}}^{2}\d s \right)^{\frac{1}{2}}\lesssim C^{\frac{3}{2}}\varepsilon^{\frac{3}{2}}.
	 \end{aligned}
	 \end{equation*}
	 Combining the above estimates with the definition of $\widetilde{\psi}$ in \S\ref{Se:Hidden}, we obtain 
	 \begin{equation}\label{est:EnergypsiN-1}
	 \begin{aligned}
	 \sum_{|I|\le N-1}\mathcal{E}^{D}(t,\widehat{{\Gamma}}^{I}{\psi})^{\frac{1}{2}}&\lesssim \sum_{|I|\le N}\bigg(\int_{0}^{t}\int_{\R^{2}}\frac{\big|\big[\widehat{{\Gamma}}^{I}(v\psi)\big]_{-}\big|^{2}}{\langle r-s\rangle^{\frac{6}{5}}}\d x \d s \bigg)^{\frac{1}{2}}\\
	 &+\sum_{|I|\le N-1}\left(\mathcal{E}^{D}(t,\widehat{{\Gamma}}^{I}\widetilde{\psi})^{\frac{1}{2}}+\big\|\widehat{{\Gamma}}^{I}\big(\gamma^{\mu}\partial_{\mu}(v\psi)\big)\big\|_{L_{x}^{2}}\right)\\
	 &\lesssim \varepsilon+C^{\frac{3}{2}}\varepsilon^{\frac{3}{2}}+C^{2}\varepsilon^{2}+C^{3}\varepsilon^{3}\lesssim \varepsilon,
	 \end{aligned}
	 \end{equation}
	  for $\varepsilon$ small enough (depending on $C$).
	 Then, from~\eqref{est:Energyphi2},~\eqref{est:LtLx1vpsiN} and~\eqref{est:initialenergy}, we see that 
	 \begin{equation}\label{est:EnergypsiN}
	 \begin{aligned}
	 \sum_{|I|\le N}\mathcal{E}^{D}(t,\widehat{{\Gamma}}^{I}\psi)&\lesssim \sum_{|I|\le N}\int_{0}^{t}\int_{\R^{2}}\big|\big(\widehat{{\Gamma}}^{I}\psi\big)^{*}\gamma^{0}\widehat{{\Gamma}}^{I}(v\psi)\big|\d x \d s\\
	 &+\sum_{|I|\le N}\mathcal{E}^{D}(0,\widehat{{\Gamma}}^{I}\psi)\lesssim \varepsilon^{2}+C^{\frac{5}{2}}\varepsilon^{\frac{5}{2}}\langle t\rangle^{2\delta}.
	 \end{aligned}
	 \end{equation}
	 
	 Estimates~\eqref{est:EnergypsiN-1} and~\eqref{est:EnergypsiN} strictly improve the estimates of $\mathcal{E}^{D}(t,\widehat{{\Gamma}}^{I}\psi)$ in~\eqref{est:Bootpsi}.
	 
	 \smallskip
	 \textbf{Step 2.} Closing the estimates in $\mathcal{G}_{1}(t,{\Gamma}^{I}{v})$.
	 Note that, from~\eqref{est:hatGa}, \eqref{est:EnergyKG}, \eqref{equ:Hiddenv}, \eqref{est:LtLxN3N4} and~\eqref{est:initialenergy}, we have 
	 \begin{equation*}
	 \begin{aligned}
	 \sum_{|I|\le N-1}\mathcal{G}_{1}(t,{{\Gamma}}^{I}\tilde{v})^{\frac{1}{2}}
	 &\lesssim \sum_{|I|\le N-1}\mathcal{G}_{1}(0,{{\Gamma}}^{I}\tilde{v})^{\frac{1}{2}}+\sum_{|I|\le N-1}\int_{0}^{t}\big\|\Gamma^{I}\mathcal{N}_{3}(\psi,v)\big\|_{L_{x}^{2}}\d s \\
	 &+\sum_{|I|\le N-1}\int_{0}^{t}\big\|\Gamma^{I}\mathcal{N}_{4}(\psi,v)\big\|_{L_{x}^{2}}\d s\lesssim \varepsilon+C^{2}\varepsilon^{2}\lesssim \varepsilon,
	 \end{aligned}
	 \end{equation*}
	 for $\varepsilon$ small enough (depending on $C$). Then, from~\eqref{est:hatGa},~\eqref{est:pointGlobal} and $N\ge 14$, 
	 \begin{equation*}
	 \begin{aligned}
	 \sum_{|I|\le N}\big\|{{\Gamma}}^{I}\big(\psi^{*}\gamma^{0}\psi\big)\big\|_{L_{x}^{2}}
	 &\lesssim \sum_{\substack{|I_{1}|\le N\\ |I_{2}|\le N-4}}\left\|{{\Gamma}}^{I_{1}}v\right\|_{L_{x}^{2}} \big\|\widehat{{\Gamma}}^{I_{2}}\psi\big\|_{L_{x}^{\infty}}\lesssim C^{2}\varepsilon^{2}\lesssim \varepsilon.
	 \end{aligned}
	 \end{equation*}
	 Next, from~\eqref{est:hatGG},~\eqref{est:Bootpsi} and~\eqref{est:pointGlobal}, we have 
	 \begin{equation*}
	 \begin{aligned}
	 &\sum_{|I|\le N}\left(\int_{0}^{t}\int_{\R^{2}}\frac{\left|\Gamma^{I}\left(\psi^{*}\gamma^{0}\psi\right)\right|^{2}}{\langle r-s\rangle^{\frac{6}{5}}}\d x\d s \right)^{\frac{1}{2}}\\
	 &\lesssim \sum_{\substack{|I_{1}|\le N\\ |I_{2}|\le N-4}}\left(\int_{0}^{t}\left(\big\|\widehat{{\Gamma}}^{I_{2}}\psi\big\|^{2}_{L_{x}^{\infty}}\int_{\R^{2}}\frac{\big|\big[\widehat{\Gamma}^{I_{1}}\psi\big]_{-}\big|^{2}}{\langle r-s\rangle^{\frac{6}{5}}}\d x\right)\d s \right)^{\frac{1}{2}}\\
	 &+\sum_{\substack{|I_{1}|\le N-7\\ |I_{2}|\le N}}\left(\int_{0}^{t}\big\|\big[\widehat{{\Gamma}}^{I_{1}}\psi\big]_{-}\big\|^{2}_{L_{x}^{\infty}}\big\|\widehat{{\Gamma}}^{I_{2}}\psi\big\|^{2}_{L_{x}^{2}}\d s \right)^{\frac{1}{2}}\lesssim C^{\frac{3}{2}}\varepsilon^{\frac{3}{2}}.
	 \end{aligned}
	 \end{equation*}
	 Combining the above estimates with the definition of $\widetilde{v}$ in \S\ref{Se:Hidden}, we obtain 
	 \begin{equation}\label{est:EnergyvN-1}
	 \begin{aligned}
	 \sum_{|I|\le N-1}\mathcal{G}_{1}(t,{{\Gamma}}^{I}{v})^{\frac{1}{2}}
	 &\lesssim \sum_{|I|\le N-1}\mathcal{G}_{1}(t,{{\Gamma}}^{I}\widetilde{v})^{\frac{1}{2}}+\sum_{|I|\le N}\big\|{{\Gamma}}^{I}\big(\psi^{*}\gamma^{0}\psi\big)\big\|_{L_{x}^{2}}\\
	 &+\sum_{|I|\le N}\left(\int_{0}^{t}\int_{\R^{2}}\frac{\left|\Gamma^{I}\left(\psi^{*}\gamma^{0}\psi\right)\right|^{2}}{\langle r-s\rangle^{\frac{6}{5}}}\d x\d s \right)^{\frac{1}{2}}\lesssim \varepsilon.
	 \end{aligned}
	 \end{equation}
	 Then, from~\eqref{est:EnergyKG},~\eqref{est:LtLx2psiN} and~\eqref{est:initialenergy}, we see that 
	 \begin{equation}\label{est:EnergyvN}
	 \begin{aligned}
	 &\sum_{|I|\le N}\mathcal{G}_{1}(t,{{\Gamma}}^{I}v)^{\frac{1}{2}}\\
	 &\lesssim \sum_{|I|\le N}\left(\mathcal{G}_{1}(0,{{\Gamma}}^{I}v)^{\frac{1}{2}}+\int_{0}^{t}\left\|\Gamma^{I}\left(\psi^{*}\gamma^{0}\psi\right)\right\|_{L_{x}^{2}}\d x\right)\lesssim \varepsilon+C^{2}\varepsilon^{2}\langle t\rangle^{2\delta}.
	 \end{aligned}
	 \end{equation}
	 
	 Estimates~\eqref{est:EnergyvN-1} and~\eqref{est:EnergyvN} strictly improve the estimates of $\mathcal{G}_{1}(t,{{\Gamma}}^{I}v)$ in~\eqref{est:Bootv}.
	 
	 \smallskip
	 \textbf{Step 3.} Closing the energy estimates of $(\psi,v)$ in the exterior region. Note that, from~\eqref{est:Bootpsi}, we have 
	 \begin{equation*}
	 \sum_{|I|\le N-7}\left\|\widehat{{\Gamma}}^{I}\psi(t,x)\bchar_{\{r-2t\ge 1\}}\right\|_{L_{x}^{\infty}}\lesssim C\varepsilon\langle t\rangle^{-\frac{3}{2}+\delta}.
	 \end{equation*}
	 Therefore, from,~\eqref{est:hatGa},~\eqref{est:Bootpsi},~\eqref{est:Bootv}, $N\ge 14$, and the Cauchy-Schwarz inequality, 
	 \begin{equation*}
	 \begin{aligned}
	 &\sum_{|I|\le N}\int_{0}^{t}\|\langle r-s\rangle\chi(r-2s)|\widehat{\Gamma}^{I}(v\psi)|\|_{L_{x}^{2}}\d s \\
	 &\lesssim \sum_{\substack{|I_{1}|\le N\\|I_{2}|\le N-7}}\int_{0}^{t}\|\langle r-s\rangle\chi(r-2s)\Gamma^{I_{1}}v\|_{L_{x}^{2}}\big\|\big(\widehat{{\Gamma}}^{I_{2}}\psi\big)\bchar_{\{r-2s\ge 1\}}\big\|_{L_{x}^{\infty}}\d s \\
	 &+\sum_{\substack{|I_{1}|\le N-6\\|I_{2}|\le N}}\int_{0}^{t}\|\langle r-s\rangle\chi(r-2s)\widehat{\Gamma}^{I_{2}}\psi\|_{L_{x}^{2}}\big\|{{\Gamma}}^{I_{1}}v\big\|_{L_{x}^{\infty}}\d s \lesssim C^{2}\varepsilon^{2}\langle t\rangle^{\delta},
	 \end{aligned}
	 \end{equation*}
	 	 \begin{equation*}
	 \begin{aligned}
	 &\sum_{|I|\le N}\int_{0}^{t}\|\langle r-s\rangle\chi(r-2s)|{\Gamma}^{I}(\psi^{*}\gamma^{0}\psi)|\|_{L_{x}^{2}}\d s \\
	 &\lesssim \sum_{\substack{|I_{1}|\le N\\|I_{2}|\le N-7}}\int_{0}^{t}\|\langle r-s\rangle\chi(r-2s)\widehat{\Gamma}^{I_{1}}\psi\|_{L_{x}^{2}}\big\|\big(\widehat{{\Gamma}}^{I_{2}}\psi\big)\bchar_{\{r-2s\ge 1\}}\big\|_{L_{x}^{\infty}}\d s \lesssim C^{2}\varepsilon^{2}.
	 \end{aligned}
	 \end{equation*}
	 Based on~\eqref{equ:DKG},~\eqref{est:cutenergyphi},~\eqref{est:cutenergyv},~\eqref{est:initialenergy} and above two estimates, we obtain
	 \begin{equation*}
	 \begin{aligned}
	& \sum_{|I|\le N}\big\|\langle r-s\rangle\chi(r-2s)\big|\widehat{{\Gamma}}^{I}\psi\big|\big\|_{L_{x}^{2}}\\
	&\lesssim \sum_{|I|\le N}\big(\big\|\langle r\rangle \widehat{{\Gamma}}^{I}\psi(0,x)\big\|_{L_{x}^{2}}+\int_{0}^{t}\|\langle r-s\rangle\chi(r-2s)|\widehat{\Gamma}^{I}(v\psi)|\|_{L_{x}^{2}}\d s\big)\lesssim \varepsilon+C^{2}\varepsilon^{2}\langle t\rangle^{\delta},
	 \end{aligned}
	 \end{equation*}
	 \begin{equation*}
	 \begin{aligned}
	 & \sum_{|I|\le N}\big\|\langle r-s\rangle\chi(r-2s)\big|{{\Gamma}}^{I}v\big|\big\|_{L_{x}^{2}}\\
	 &\lesssim \sum_{|I|\le N}\big\|\langle r\rangle {{\Gamma}}^{I}v(0,x)\big\|_{H_{x}^{1}}+\sum_{|I|\le N}\big\|\langle r\rangle {{\Gamma}}^{I}\pt v(0,x)\big\|_{L_{x}^{2}}\\
	 &+\sum_{|I|\le N}\int_{0}^{t}\|\langle r-s\rangle\chi(r-2s)|{\Gamma}^{I}(\psi^{*}\gamma^{0}\psi)|\|_{L_{x}^{2}}\d s\lesssim \varepsilon+C^{2}\varepsilon^{2}.
	 \end{aligned}
	 \end{equation*}
	 These strictly improve the energy estimates of $(\psi,v)$ in the exterior region in the bootstrap assumption~\eqref{est:Bootpsi}-\eqref{est:Bootv} for $C$ large enough and $\varepsilon$ sufficiently small.
	 
	 \smallskip
	 {\textbf{Step 4.}} Closing the modified ghost weight estimates of $(\psi,v)$. From~\eqref{est:ModiGhostphi},~\eqref{est:ModighostKG}, \eqref{est:Bootv}, \eqref{est:GhostpsiN}, \eqref{est:GhostvN} and~\eqref{est:initialenergy}, we see that 
	 \begin{equation*}
	 \begin{aligned}
	 &\sum_{|I|\le N}\int_{0}^{t}\langle s\rangle^{-3\delta}\int_{\R^{2}}\frac{\big[\widehat{{\Gamma}}^{I}\psi\big]^{2}_{-}}{\langle r-s\rangle^{\frac{6}{5}}}\d x \d s\\
	 &\lesssim \sum_{|I|\le N}\left(\int_{\R^{2}}\big|\widehat{{\Gamma}}^{I}\psi_{0}\big|^{2}\d x +
	 \int_{0}^{t}\langle s\rangle^{-3\delta}\int_{\R^{2}}\big|\big(\widehat{{\Gamma}}^{I}\psi\big)^{*}\gamma^{0}\big(\widehat{{\Gamma}}^{I}(v\psi)\big)\big|\d x \d s
	 \right)\lesssim \varepsilon^{2}+C^{2}\varepsilon^{2},
	  \end{aligned}
	 \end{equation*}
	  \begin{equation*}
	 \begin{aligned}
	 &\sum_{|I|\le N}\int_{0}^{t}\langle s\rangle^{-5\delta}\int_{\R^{2}}\frac{\left(\Gamma^{I}v\right)^{2}}{\langle r-s\rangle^{\frac{6}{5}}}\d x \d s\\
	 &\lesssim \sum_{|I|\le N}\left(\mathcal{E}_{1}(0,\Gamma^{I}v) +
	 C\varepsilon\int_{0}^{t}\langle s\rangle^{-3\delta}\big\|\Gamma^{I}\big(\psi^{*}\gamma^{0}\psi\big)\big\|_{L_{x}^{2}} \d s
	 \right)\lesssim \varepsilon^{2}+C^{\frac{5}{2}}\varepsilon^{\frac{5}{2}}.
	 \end{aligned}
	 \end{equation*}
	  These strictly improve the modified ghost weight estimates of $(\psi,v)$ in the bootstrap assumption~\eqref{est:Bootpsi}-\eqref{est:Bootv} for $C$ large enough and $\varepsilon$ small enough.

	 \smallskip
	 {\textbf{Step 5.}} Closing the pointwise estimates of $(\psi,v)$. First, from~\eqref{est:Gaparf},~\eqref{est:hatparGa},~\eqref{est:pointwiseSG} and the definition of $\Psi$ in~\S\ref{Se:Hidden}, we see that 
	 \begin{equation*}
	 \begin{aligned}
	 \sum_{|I|\le N-7}\big|\widehat{{\Gamma}}^{I}\psi\big|
	 &\lesssim \langle t-r\rangle^{-1}\sum_{|I|\le N-7}\left\|S\Gamma^{I}\Psi\right\|_{L_{x}^{\infty}}\\
	 &+ \langle t-r\rangle^{-1}\sum_{|I|\le N-7}\left\|\Gamma\Gamma^{I}\Psi\right\|_{L_{x}^{\infty}}
	 \lesssim \varepsilon\langle t\rangle^{-\frac{1}{2}+\delta}\langle t-r\rangle^{-1}.
	 \end{aligned}
	 \end{equation*}
	 
	 Second, using again~\eqref{equ:Hiddenpsi-},~\eqref{est:pointwiseSG} and the definition of $\Psi$ in~\S\ref{Se:Hidden}, we have 
	 \begin{equation*}
	 \begin{aligned}
	 \sum_{|I|\le N-7}\big|[\widehat{{\Gamma}}^{I}\psi]_-\big|
	 &\lesssim \langle t+r\rangle^{-1}\sum_{|I|\le N-7}\left\|S\Gamma^{I}\Psi\right\|_{L_{x}^{\infty}}\\
	 &+ \langle t+r\rangle^{-1}\sum_{|I|\le N-7}\left\|\Gamma\Gamma^{I}\Psi\right\|_{L_{x}^{\infty}}
	 \lesssim \varepsilon\langle t+r\rangle^{-\frac{3}{2}+\delta}.
	 \end{aligned}
	 \end{equation*}
	 
	 Third, from~\eqref{est:smallness},~\eqref{est:KGLINI} and~\eqref{est:N3N4}, we find
	 \begin{equation*}
	 \begin{aligned}
	 \sum_{|I|\le N-6}\left|\Gamma^{I}\tilde{v}\right|&\lesssim \left(\varepsilon+C\varepsilon^{2}+C^{3}\varepsilon^{3}\right)\langle t+r\rangle^{-1}.
	 \end{aligned}
	 \end{equation*}
	 
	 On the other hand, from~\eqref{est:hatGG} and~\eqref{est:pointGlobal}, we see that 
	 \begin{equation*}
	 \sum_{|I|\le N-6}\big|\Gamma^{I}\big(\psi^{*}\gamma^{0}\psi\big)\big|\lesssim \sum_{\substack{|I_{1}|\le N-4\\ |I_{2}|\le N-4}}\big|\widehat{{\Gamma}}^{I_{1}}\psi\big|\big|\widehat{{\Gamma}}^{I_{2}}\psi\big|\lesssim C^{2}\varepsilon^{2}\langle t+r\rangle^{-1}.
	 \end{equation*}
	 
	 Combining the above two estimates with the definition of $\tilde{v}$ in \S\ref{Se:Hidden}, we obtain
	 \begin{equation*}
	 \sum_{|I|\le N-6}\left|\Gamma^{I}{v}\right|\lesssim \sum_{|I|\le N-6}\left(\left|\Gamma^{I}\tilde{v}\right|+\big|\Gamma^{I}\big(\psi^{*}\gamma^{0}\psi\big)\big|\right)\lesssim \left(\varepsilon+C^{2}\varepsilon^{2}+C^{3}\varepsilon^{3}\right)\langle t+r\rangle^{-1}.
	 \end{equation*}
	 These strictly improve the pointwise estimates of $(\psi,v)$ in the bootstrap assumption~\eqref{est:Bootpsi}-\eqref{est:Bootv} for $C$ large enough and $\varepsilon$ small enough.

	 \smallskip
	 \textbf{Step 6.} Conclusion. As a consequence of improving all the estimates in the bootstrap assumption \eqref{est:Bootpsi}-\eqref{est:Bootv}, for any initial data $(\psi_{0},\vec{v}_{0})$ satisfying \eqref{est:smallness}, we conclude that $T^{*}(\psi_{0},\vec{v}_{0})=\infty$. The proof of Proposition~\ref{pro:main} is complete.
	\end{proof}

\subsection{End of the proof of Theorem~\ref{thm:main}}\label{Se:Endthm}

We are in a position to complete the proof of Theorem~\ref{thm:main}.

\begin{proof}[End of the proof of Theorem~\ref{thm:main}]
	Note that, the global existence of $(\psi,v)$ and sharp pointwise decay~\eqref{est:thmpoint} are consequences of the Proposition~\ref{pro:main}. What is left is to show that the solution $(\psi,v)$ scatters to a free solution in the energy space $\mathcal{X}_{N}$.
	We split the proof of scattering as the following three steps.
	
	\smallskip
	\textbf{Step 1.} Linear theory of the Dirac and Klein-Gordon equations. 
	Recall that, for the Cauchy problem of the 2D linear Dirac equation,
	\begin{equation*}
	-i\gamma^{\mu}\partial_{\mu}\phi=G\quad\mbox{with}\quad \phi(0,x)=\phi_{0},
	\end{equation*}
	the solution $\phi$ is given by 
	\begin{equation*}
	\phi(t,x)=e^{-it\mathcal{D}}\phi_{0}+\int_{0}^{t}e^{-i(t-s)\mathcal{D}} \gamma^0 G(s)\d s,
	\end{equation*}
	where $e^{-it\mathcal{D}}$ can be written as 
	\begin{equation*}
	e^{-it\mathcal{D}}=\cos (t\sqrt{-\Delta})-i\frac{\sin (t\sqrt{-\Delta})}{\sqrt{-\Delta}}\mathcal{D},\quad \mbox{with}\ 
	\mathcal{D}=-i\gamma^{0}\gamma^{1}\partial_{1}-i\gamma^{0}\gamma^{2}\partial_{2}.
	\end{equation*}
	
	Recall also that, for the Cauchy problem of the 2D linear Klein-Gordon equation,
	\begin{equation*}
	-\Box u+u=F\quad\mbox{with}\quad (u,\pt u)_{|t=0}=(u_{0},u_{1}),
	\end{equation*}
	the solution $\vec{u}=(u,\pt u)$ is given by 
	\begin{equation*}
	\vec{u}(t,x)=\mathcal{S}(t)
	\left(\begin{aligned}
	u_{0}\\ u_{1}
	\end{aligned}\right)+\int_{0}^{t}\mathcal{S}(t-s)
	\left(\begin{aligned}
	&0\\ F&(s)
	\end{aligned}\right)\d s,
	\end{equation*}
	where $\mathcal{S}(t)$ can be written as 
	\begin{equation*}
	\mathcal{S}(t)=
	\begin{pmatrix}
	\cos \left(t\langle \nabla \rangle\right)& \langle \nabla \rangle^{-1}\sin  \left(t\langle \nabla\rangle\right)\\
	 -\langle \nabla \rangle\sin  \left(t\langle \nabla\rangle\right)& \cos \left(t\langle \nabla \rangle\right)
	\end{pmatrix}.
	\end{equation*}
	By a standard energy argument, we know that the Dirac flow $e^{-it\mathcal{D}}$ and the Klein-Gordon flow $\mathcal{S}(t)$ are unitary group on ${H}^{M-1}$ and $H^{M}\times H^{M-1}$ for $M\in \mathbb{N}^{+}$, respectively. More precisely, for any $M\in \mathbb{N}^{+}$, $f\in H^{M}$ and $\vec{f}\in \mathcal{H}^{M}$, we have 
	\begin{equation}\label{est:freeflow}
	\left\|e^{it\mathcal{D}}\right\|_{H^{M}\to H^{M}}=\left\|S(t)\right\|_{\mathcal{H}^{M}\to \mathcal{H}^{M}}=1,
	\end{equation}
	where $\mathcal{H}^{M}=H^{M}\times H^{M-1}$ for $M\in \mathbb{N}^{+}$. Recall that, for any $(t,s)\in (0,\infty)^{2}$, we have the following basic property of the operators,
	\begin{equation}\label{equ:operatorts}
	e^{-i(t+s)\mathcal{D}}=e^{-it\mathcal{D}}e^{-is\mathcal{D}}\quad \mbox{and}\quad 
	\mathcal{S}(t+s)=\mathcal{S}(t)\mathcal{S}(s).
	\end{equation}
	
	\smallskip
	{\textbf{{Step 2.}}} Scattering of $(\widetilde{\psi},\tilde{v})$. From \eqref{est:LtLxQ0}, \eqref{est:LtLxN1N2} and~\eqref{est:LtLxN3N4}, we see that 
	\begin{equation}\label{est:LtLxinfty}
	\begin{aligned}
	\int_{0}^{\infty}\left(\left\|\mathcal{N}_{3}(\psi,v)\right\|_{H^{N-1}}+\left\|\mathcal{N}_{4}(\psi,\psi^{*})\right\|_{H^{N-1}}\right)\d s &\lesssim C^{2}\varepsilon^{2},\\
	\int_{0}^{\infty}\left(\left\|\mathcal{N}_{1}(\psi,v)\right\|_{H^{N-1}}+\left\|\mathcal{N}_{2}(\psi,\psi^{*})\right\|_{H^{N-1}}+\left\|Q_{0}(\psi,v)\right\|_{H^{N-1}}\right)\d s &\lesssim C^{2}\varepsilon^{2}.
	\end{aligned}
	\end{equation}
	We set 
	\begin{equation*}
	\psi_{0\ell}=\widetilde{\psi}(0)+\int_{0}^{\infty}e^{is\mathcal{D}}\left(\mathcal{N}_{1}(\psi,v)+\mathcal{N}_{2}(\psi,\psi^{*})+2Q_{0}(\psi,v)\right)\d s,
	\end{equation*}
	\begin{equation*}
	\left(\begin{aligned}
	v_{0\ell}\\ v_{1\ell}
	\end{aligned}\right)=	\begin{pmatrix}
\tilde{v}\\ \pt \tilde{v}
	\end{pmatrix}(0)+\int_{0}^{\infty}\mathcal{S}(-s)\begin{pmatrix}
	0\\ \mathcal{N}_{3}(\psi,v)+\mathcal{N}_{4}(\psi,\psi^{*})
	\end{pmatrix}\d s.
	\end{equation*}
	Therefore, from~\eqref{equ:HiddenPsi},~\eqref{equ:Hiddenv},~\eqref{est:freeflow} and~\eqref{equ:operatorts}, we see that
	\begin{equation*}
	\begin{aligned}
	&\big\|\widetilde{\psi}(t)-e^{-it\mathcal{D}}\psi_{0\ell}\big\|_{H^{N-1}}\\
	&\le \int_{t}^{\infty}\left(\left\|\mathcal{N}_{1}(\psi,v)\right\|_{H^{N-1}}+\left\|\mathcal{N}_{2}(\psi,\psi^{*})\right\|_{H^{N-1}}+\left\|Q_{0}(\psi,v)\right\|_{H^{N-1}}\right)\d s,
	\end{aligned}
	\end{equation*} 
	\begin{equation*}
	\left\|	\begin{pmatrix}
	\tilde{v}\\ \pt \tilde{v}
	\end{pmatrix}(t)-\mathcal{S}(t)	\left(\begin{aligned}
	v_{0\ell}\\ v_{1\ell}
	\end{aligned}\right)\right\|_{\mathcal{H}^{N}}\le \int_{t}^{\infty}\left(\left\|\mathcal{N}_{3}(\psi,v)\right\|_{H^{N-1}}+\left\|\mathcal{N}_{4}(\psi,\psi^{*})\right\|_{H^{N-1}}\right)\d s.
	\end{equation*}
	Based on the above estimates and~\eqref{est:LtLxinfty}, we have 
	\begin{equation*}
	\lim_{t\to \infty}\big\|\widetilde{\psi}(t)-e^{-it\mathcal{D}}\psi_{0\ell}\big\|_{H^{N-1}}=\lim_{t\to \infty}	\left\|	\begin{pmatrix}
	\tilde{v}\\ \pt \tilde{v}
	\end{pmatrix}(t)-\mathcal{S}(t)	\left(\begin{aligned}
	v_{0\ell}\\ v_{1\ell}
	\end{aligned}\right)\right\|_{\mathcal{H}^{N}}=0,
	\end{equation*}
	which means that $(\widetilde{\psi},\tilde{v})$ scatters to a free solution in the energy space $\mathcal{X}_{N}=H^{N-1}\times H^{N}\times H^{N-1}$.
	
	\smallskip
	{\textbf{Step 3.}} Conclusion. Note that, from~\eqref{est:hatGa},~\eqref{est:Bootpsi},~\eqref{est:Bootv},~\eqref{est:pointGlobal} and $N\ge 14$, we have 
	\begin{equation*}
	\begin{aligned}
	\left\|i\gamma^{\mu}\partial_{\mu}(v\psi)\right\|_{H^{N-1}}
	&\lesssim \sum_{|I|\le N}\left\|\Gamma^{I}(v\psi)\right\|_{L_{x}^{2}}\\
	&\lesssim \sum_{\substack{|I_{1}|\le N\\ |I_{2}|\le N-4}}\left\|\Gamma^{I_{1}}v\right\|_{L_{x}^{2}}
	\big\|\widehat{\Gamma}^{I_{2}}\psi\big\|_{L_{x}^{\infty}}+\sum_{\substack{|I_{1}|\le N-6\\ |I_{2}|\le N}}\left\|\Gamma^{I_{1}}v\right\|_{L_{x}^{\infty}}
	\big\|\widehat{\Gamma}^{I_{2}}\psi\big\|_{L_{x}^{2}}\\
	&\lesssim C^{2}\varepsilon^{2}\langle t\rangle^{-\frac{1}{2}+2\delta}
	+C^{2}\varepsilon^{2}\langle t\rangle^{-1+\delta},
	\end{aligned}
	\end{equation*}
	\begin{equation*}
	\begin{aligned}
	&\left\|\psi^{*}\gamma^{0}\psi\right\|_{H^{N}}+\left\|\pt (\psi^{*}\gamma^{0}\psi)\right\|_{H^{N-1}}\\
	&\lesssim \sum_{|I|\le N}\left\|\Gamma^{I}\left(\psi^{*}\gamma^{0}\psi\right)\right\|_{L_{x}^{2}}\lesssim \sum_{\substack{|I_{1}|\le N\\ |I_{2}|\le N-4}}\big\|\widehat{{\Gamma}}^{I_{1}}\psi\big\|_{L_{x}^{2}}\big\|\widehat{{\Gamma}}^{I_{2}}\psi\big\|_{L_{x}^{\infty}}\lesssim C^{2}\varepsilon^{2}\langle t\rangle^{-\frac{1}{2}+\delta}.
	\end{aligned}
	\end{equation*}
	Based on the above two estimates, we have 
	\begin{equation}\label{est:limenergy}
	\lim_{t\to \infty}\left\|\left(i\gamma^{\mu}\partial_{\mu}(v\psi),\psi^{*}\gamma^{0}\psi,\pt (\psi^{*}\gamma^{0}\psi)\right)\right\|_{\mathcal{X}_{N}}=0.
	\end{equation}
	Gathering~\eqref{est:limenergy} and Step 2, we have shown that the solution $(\psi,v)$ scatters to a free solution in the energy space $\mathcal{X}_{N}$.
	\end{proof}

\section*{Acknowledgements}
The author S.D. owes great thanks to Zoe Wyatt (KCL) for helpful discussions.

\end{document}